\documentclass[a4paper,11pt,reqno]{amsart}
\usepackage[utf8]{inputenc}
\usepackage[T1]{fontenc}
\usepackage{lmodern}
\usepackage[english]{babel}
\usepackage{microtype}
\usepackage{amsmath,amssymb,amsfonts,amsthm,esint}
\usepackage{mathtools,accents}
\usepackage{mathrsfs}
\usepackage{aliascnt}
\usepackage{braket}
\usepackage{bm}
\usepackage[a4paper,margin=2.5cm]{geometry}
\usepackage[citecolor=blue, linkcolor=black, colorlinks]{hyperref}
\usepackage{enumerate}
\usepackage{xcolor}
\usepackage{tikz}
\usetikzlibrary{arrows.meta,decorations.pathreplacing}
\usepackage{comment}
\makeatletter
\g@addto@macro\@floatboxreset\centering
\makeatother

\makeatletter
\def\newaliasedtheorem#1[#2]#3{
	\newaliascnt{#1@alt}{#2}
	\newtheorem{#1}[#1@alt]{#3}
	\expandafter\newcommand\csname #1@altname\endcsname{#3}
}
\makeatother
\numberwithin{equation}{section}
\newtheoremstyle{slanted}{\topsep}{\topsep}{\slshape}{}{\bfseries}{.}{.5em}{}
\theoremstyle{plain}
\newtheorem{theorem}{Theorem}[section]
\newaliasedtheorem{proposition}[theorem]{Proposition}
\newaliasedtheorem{lemma}[theorem]{Lemma}
\newaliasedtheorem{corollary}[theorem]{Corollary}
\newaliasedtheorem{counterexample}[theorem]{Counterexample}
\theoremstyle{definition}
\newaliasedtheorem{definition}[theorem]{Definition}
\newaliasedtheorem{question}[theorem]{Question}
\newaliasedtheorem{openquestion}[theorem]{Open Question}
\theoremstyle{remark}
\newaliasedtheorem{remark}[theorem]{Remark}
\newaliasedtheorem{example}[theorem]{Example}

\let\altphi\phi
\let\phi\varphi
\let\varphi\altphi
\let\altphi\undefined

\let\div\undefined
\DeclareMathOperator{\div}{div}

\DeclareMathOperator{\supp}{supp}

\newfont{\tmpf}{cmsy10 scaled 2500}

\def\XXint#1#2#3{{\setbox0=\hbox{$#1{#2#3}{\int}$ }
		\vcenter{\hbox{$#2#3$ }}\kern-.6\wd0}}

\makeatletter
\def\author@andify{%
  \nxandlist
    {\unskip,\penalty-1\hspace{0.7em}\ignorespaces}%
    {\unskip,\penalty-2\hspace{0.7em}\ignorespaces}%
    {\unskip,\penalty-2\hspace{0.7em}\ignorespaces}%
}
\makeatother
\title[Flexibility for SQG with an $L^{4/3+}$ active scalar]
{Flexibility for the SQG Equation with an $L^{4/3+}$ Active Scalar}

\author{Elia Bru\`e}
\address{
Department of Decision Sciences,
Bocconi University,
Via Roentgen 1,
20136 Milan, Italy
}
\email{elia.brue@unibocconi.it}

\author{Rui Jin}
\address{
School of Mathematical Sciences, 
Shanghai Jiao Tong University,
200240 Shanghai, PR China
}
\address{
Department of Decision Sciences,
Bocconi University,
Via Roentgen 1,
20136 Milan, Italy
}
\email{jinruisjtu@gmail.com}

\author{Quoc-Hung Nguyen}
\address{
State Key Laboratory of Mathematical Sciences, Academy of Mathematics and Systems Science, Chinese Academy of Sciences, Beijing 100190, China
}
\address{Institute of mathematics, Academy of Mathematics and Systems Science, the Chinese Academy of Sciences, Beijing 100190, China
}
\email{qhnguyen@amss.ac.cn}

\begin{document}
	\begin{abstract}
		We develop a new convex-integration scheme, inspired by \cite{BCK26},
		for the inviscid surface quasi-geostrophic equation on the
		two-dimensional torus. For the explicit, nonoptimized exponent
			$\bar p=\frac{4}{3}+10^{-5},$
		we prove a flexibility theorem for weak solutions in the standard
		momentum formulation with active scalar
		\[
			\theta\in C([0,1];L^{\bar p}(\mathbb T^2)).
		\]
		More precisely, any two prescribed mean-zero states in
		$L^{\bar p}(\mathbb T^2)$ can be approximated at the initial and final
		times by such a solution. The perturbations are constructed from
		localized, concentrated traveling SQG profiles whose centers move along
		rational directions and whose radii depend on the Reynolds stress. Time
		averages of auxiliary sources along these trajectories reconstruct the
		preceding-stage stress, while a two-dimensional bilinear null-form
		estimate compensates for the derivative loss caused by the nonlocal
		constitutive law. Exploiting the time-locality of the iteration, we also
		obtain a dense subset of the mean-zero space
		$L^{\bar p}(\mathbb T^2)$ such that every initial datum in this subset
		admits at least two distinct momentum weak solutions. Thus, the
		construction establishes both flexibility and nonuniqueness beyond the
		concentration-critical exponent $p=4/3$.
	\end{abstract}
	\maketitle 
    
	\section{Introduction}
	We consider the inviscid surface quasi-geostrophic (SQG) equation on the
	two-dimensional torus $\mathbb T^2:=\mathbb R^2/\mathbb Z^2$:
	\begin{equation}\label{eq:SQG-introduction}
		\left\{
		\begin{aligned}
			\partial_t\theta+u\cdot\nabla\theta&=0,\\
			u&=\nabla^\perp\Lambda^{-1}\theta,
		\end{aligned}
		\right.
	\end{equation}
	where $\theta=\theta(x,t)$ is the active scalar,
	$\Lambda=(-\Delta)^{1/2}$, and
	$\nabla^\perp=(-\partial_2,\partial_1)$. The transport velocity
	$u$ is divergence-free. Throughout the paper, the active scalar is
	assumed to have zero spatial mean at every time:
	\[
	\int_{\mathbb T^2}\theta(x,t)\,dx=0.
	\]
	
	The SQG equation arises as an idealized model for boundary buoyancy
    (or potential temperature) in rapidly rotating, stably stratified
    geophysical flows~\cite{HPGS95}. It also shares the formal
    transport--stretching structure of the three-dimensional Euler
    vorticity equation, as emphasized by Constantin, Majda, and
    Tabak~\cite{CMT94}. Global regularity of smooth inviscid SQG solutions remains open: it is unknown whether smooth initial data can develop a singularity in finite time. Here we instead address nonuniqueness of low-regularity momentum weak solutions.

	\subsection{The momentum formulation and the
		\texorpdfstring{$L^{4/3}$}{L4/3} endpoint}
	
	In addition to the active scalar $\theta$ and the transport velocity
	$u$, we introduce the potential velocity
	\begin{equation}\label{eq:potential-velocity-introduction}
		v
		:=
		\nabla^\perp(-\Delta)^{-1}\theta
		=
		\Lambda^{-1}u.
	\end{equation}
	Then
	\[
	\operatorname{div}v=0,
	\qquad
	\theta=-\operatorname{curl}v,
	\qquad
	u=\Lambda v.
	\]
	
	For smooth solutions, equation \eqref{eq:SQG-introduction} is equivalent
	to the SQG momentum equation
	\begin{equation}\label{eq:SQG-momentum-introduction}
		\left\{
		\begin{aligned}
			\partial_t v
			+u\cdot\nabla v
			-(\nabla v)^Tu
			+\nabla p
			&=0,\\
			\operatorname{div}v&=0,\\
			u&=\Lambda v.
		\end{aligned}
		\right.
	\end{equation}
	This formulation was used in the convex integration construction of
	Buckmaster, Shkoller, and Vicol \cite{BSV19}.
	
	The nonlinear term in \eqref{eq:SQG-momentum-introduction} admits a
	commutator formulation which remains meaningful at considerably lower
	regularity than its pointwise expression suggests. More precisely,
	let $I\subset\mathbb R$ be an open interval. A
	divergence-free vector field
	\[
	v\in L^2_{\mathrm{loc}}(I;H^{1/2}(\mathbb T^2))
	\]
	with associated scalar $\theta=-\operatorname{curl}v$
	is a weak solution of \eqref{eq:SQG-momentum-introduction} if
	\begin{align}\label{eq:weak-momentum-introduction}
		\int_I
		\Big(
		\langle v_i,\partial_t\phi_i\rangle
		+
		\langle\Lambda v_j,v_i\partial_j\phi_i\rangle
		-
		\frac12
		\langle\partial_i v_j,[\Lambda,\phi_i]v_j\rangle
		\Big)\,dt
		=0
	\end{align}
	for every divergence-free test function
	$\phi\in C_c^\infty(I\times\mathbb T^2;\mathbb R^2)$.
	Here $\langle\cdot,\cdot\rangle$ denotes the spatial distributional
	pairing on $\mathbb T^2$, and repeated indices are summed.
	In this paper, the expression \emph{weak solution of SQG} always means
	a solution in this momentum sense. When $p<2$, the product $u\theta$
	in the scalar transport equation need not be integrable a priori, so
	the momentum formulation is not merely a change of notation.
	
	Suppose now that $1<p<\infty$ and $\theta\in L^p(\mathbb T^2)$. By
	\eqref{eq:potential-velocity-introduction} and the boundedness of the
	Riesz transforms,
	\[
	v\in W^{1,p}(\mathbb T^2),
	\qquad
	C_p^{-1}\|\theta\|_{L^p}
\leq \|\nabla v\|_{L^p}
\leq C_p\|\theta\|_{L^p}.
	\]
	The Sobolev embedding
	\[
	W^{1,p}(\mathbb T^2)\hookrightarrow H^{1/2}(\mathbb T^2)
	\]
	holds when $p\ge4/3$. Thus $L^{4/3}$ is the natural lower threshold,
	within the Lebesgue scale, to obtain the standard momentum class
	from an arbitrary $L^p$ scalar. 

    Global weak solutions on $\mathbb R^2$ with initial data in
$L^2(\mathbb R^2)$ were first constructed by Resnick~\cite{Res95}.
Marchand~\cite{Mar08} subsequently extended this existence theory to initial data in $L^p(\mathbb R^2)$, $p>4/3$, 
and also established global weak existence for initial data in
$\dot H^{-1/2}(\mathbb R^2)$.
More recently, De Rosa, Latocca, and Park~\cite{DRLP25} treated the
endpoint $p=4/3$ on $\mathbb T^2$: for every initial datum
$\theta_0\in L^{4/3}(\mathbb T^2)$, they constructed a global weak
solution obtained as a vanishing-viscosity limit.
Building on this work, De Rosa and Yuzbasioglu~\cite{DRY26} extended the endpoint vanishing-viscosity limit theory to the generalized SQG family.
In a complementary direction, Constantin, Ignatova, and
Nguyen~\cite{CIN26} established global weak existence on $\mathbb R^2$
and on smooth bounded domains for arbitrary initial data in the
critical Lorentz space $L^{4/3,2}$, using a different
approximation that regularizes both the advecting velocity and the
initial data.

\subsection{Main result}
   This paper develops a new convex-integration scheme
for weak SQG solutions whose active scalar $\theta\in C_t L^{\bar{p}}$ with 
\begin{equation}\label{eq:explicit-pbar-introduction}
		\bar p
		:=
		\frac43+\frac{1}{100\,000}
		>\frac43,
	\end{equation}
and therefore crosses the $L^{4/3}$ thresholds. The exponent $\bar{p}$ is
explicit but not optimized.
The reason why the gain is
small is explained in the overview below and quantified in
Section~\ref{subsec:parameter-compatibility}.

Our first theorem shows flexibility within the class $L^{\bar{p}}$: any two mean-zero states can be
approximately connected by a weak solution of SQG. 
	\begin{theorem}[Flexibility]\label{thm:introduction-flexibility}
		Let $\bar p$ be given by
		\eqref{eq:explicit-pbar-introduction}. Let $\theta_{\mathrm{start}},
		\theta_{\mathrm{end}}
		\in L^{\bar p}(\mathbb T^2)$
		be mean-zero scalar functions. Then for every $\varepsilon>0$, there
		exists a weak solution
		\[
		\theta\in
		C([0,1];L^{\bar p}(\mathbb T^2))
		\]
        to \eqref{eq:SQG-introduction} in the momentum sense of
		\eqref{eq:weak-momentum-introduction}, with zero spatial mean at every time,
		such that
		\begin{equation}\label{eq:introduction-endpoint-flexibility}
			\|\theta(\cdot,0)-\theta_{\mathrm{start}}\|_{L^{\bar p}}
			\le\varepsilon,
			\qquad
			\|\theta(\cdot,1)-\theta_{\mathrm{end}}\|_{L^{\bar p}}
			\le\varepsilon.
		\end{equation}
		Consequently, the corresponding potential and transport velocity satisfy
		\[
		v=\nabla^\perp(-\Delta)^{-1}\theta
		\in C([0,1];W^{1,\bar p}(\mathbb T^2))
		\hookrightarrow C([0,1];L^4(\mathbb T^2)),
		\quad
		u=\Lambda  v
		\in C([0,1];L^{\bar p}(\mathbb T^2)).
		\]
	\end{theorem}
	\vspace{7mm}
    The flexibility statement in
Theorem~\ref{thm:introduction-flexibility} has an immediate consequence
for the Hamiltonian
\begin{equation*}
    \mathcal H(\theta(t))
    :=
    \frac12\|\Lambda^{-1/2}\theta(t)\|_{L^2}^2.
\end{equation*}
Although smooth SQG solutions conserve the Hamiltonian and all Casimir
functionals
\(\int_{\mathbb T^2}F(\theta(x,t))\,dx\), Hamiltonian conservation can
fail at the regularity considered in our theorem. Indeed, choose
\(\theta_{\mathrm{start}}\) and \(\theta_{\mathrm{end}}\) such that
\[
    \mathcal H(\theta_{\mathrm{start}})
    \neq
    \mathcal H(\theta_{\mathrm{end}}).
\]
Since
\(L^{\bar p}(\mathbb T^2)\hookrightarrow H^{-1/2}(\mathbb T^2)\),
the functional \(\mathcal H\) is continuous in the \(L^{\bar p}\)
topology. Therefore, for sufficiently small \(\varepsilon\), the solution
provided by Theorem~\ref{thm:introduction-flexibility} satisfies
\(\mathcal H(\theta(0))\neq\mathcal H(\theta(1))\), and hence does not
conserve the Hamiltonian.

The approximation procedures in \cite{DRLP25,CIN26} provide a useful contrast: in their respective settings, both yield Hamiltonian-conserving weak solutions. In particular, the solutions obtained above by choosing endpoint states with different Hamiltonians cannot arise as vanishing-viscosity limits in~\cite{DRLP25}.

On the rigidity side, the condition
\(\theta\in L^3([0,T]\times\mathbb T^2)\) guarantees conservation of the
Hamiltonian~\cite{IV15}. Consequently, every weak solution in
\(L^\infty(0,T;L^p(\mathbb T^2))\) conserves the Hamiltonian for
\(p\geq 3\). Our theorem is the first result establishing failure of
conservation for some \(p>4/3\), more precisely in the range
\(4/3\leq p\leq \bar p\). Extending flexibility to the full range
\(4/3\leq p<3\) remains an interesting open problem.

	The same flexibility mechanism, together with the time-locality of the
	iteration, also yields nonuniqueness. Indeed, one may use the same initial
	target and two well-separated terminal targets, and arrange that the two
	resulting solutions agree on a short initial time interval while remaining
	distinct at the final time. Their common initial datum is only required to be
	$\varepsilon$-close to the prescribed initial target, as in
	\eqref{eq:introduction-endpoint-flexibility}. Since the target and
	$\varepsilon>0$ are arbitrary, the common initial data obtained in this way
	form a dense subset of the mean-zero subspace of $L^{\bar p}$. 

    \begin{theorem}[Nonuniqueness]\label{thm:SQG-nonuniqueness}
	Let $\bar p$ be given by
		\eqref{eq:explicit-pbar-introduction}.
	Then there exists a dense subset
    $\mathcal D\subset L^{\bar{p}}(\mathbb{T}^2)$ with zero mean,
	such that, for every
	\(\theta_{\mathrm{start}}\in\mathcal D\), \eqref{eq:SQG-introduction}
	admits two distinct weak solutions $\theta\in C([0,1];L^{\bar{p}})$
	with initial data $\theta(\cdot,0)=\theta_{\mathrm{start}}$.
\end{theorem}

To the best of the authors' knowledge, this is the first nonuniqueness
result for the SQG equation in class
$\theta\in L^\infty(0,T;L^p)$, for some $p\geq 4/3$, above the
critical threshold. This should be compared with the available uniqueness
theory, which requires substantially stronger regularity. Classical
solutions are locally well posed for initial data in
$H^s$, $s>2$~\cite{CMT94}, whereas Jeong and
Kim~\cite{JK21} proved strong ill-posedness at the borderline Sobolev
space $H^2$. At the level of distributional solutions,
Azzam and Bedrossian~\cite{AB15} proved uniqueness under the additional
assumption that the scalar gradients belong to
$L^2(0,T;BMO\cap L^{p_0})$, for some $1\leq p_0<\infty$. Such an
assumption is not implied by a uniform $L^p$ bound on the scalar.
	
\subsection{Convex integration and previous results}	

The convex integration method has its origins in differential geometry, in the work of Nash~\cite{Nash54} and Gromov~\cite{Gromov86}. It was brought into incompressible fluid dynamics by De Lellis and Sz\'ekelyhidi~\cite{DLS09,DLS13}, in connection with Onsager's conjecture. Following the developments in~\cite{BDIS15,DS17}, Isett~\cite{Isett18} used convex integration to prove the flexible direction of Onsager's conjecture. Admissible solutions below the Onsager threshold were later constructed by Buckmaster, De Lellis, Sz\'ekelyhidi, and Vicol~\cite{BDSV19}.
	
A subsequent breakthrough was achieved by Buckmaster and Vicol~\cite{BV19}, who introduced intermittency into the convex integration scheme and constructed nonunique finite-energy weak solutions to the three-dimensional Navier--Stokes equations. Cheskidov and Luo~\cite{CL22} later established sharp nonuniqueness below the endpoint Ladyzhenskaya--Prodi--Serrin class. Using a different approach, based on self-similar instability, Albritton, Bru\`e, and Colombo~\cite{ABC22,ABC23} proved nonuniqueness of Leray solutions for the forced equation.

For the three-dimensional Euler equations, Buckmaster, Masmoudi,
Novack, and Vicol~\cite{BMNV23} developed a spatially intermittent
scheme based on pipe flows and established flexibility in
\(C_t^0H_x^\beta\) for every \(\beta<1/2\). Building on this scheme,
Novack and Vicol~\cite{NV23} proved an intermittent version of the
\(L^3\)-based Onsager theorem, while Giri, Kwon, and
Novack~\cite{GKN24} subsequently introduced a new wavelet-inspired
\(L^3\)-based framework built around intermittent Mikado bundles.
For the two-dimensional Euler equations,  
Bru\`e and Colombo~\cite{BC23} used bundles of intermittent
jets to construct nonunique solutions with vorticity in
\(L^{1,\infty}\).  
Buck and Modena~\cite{BM24b,BM24a} proved nonuniqueness with vorticity in \(H^p\), for the full range \(0<p<1\).

A parallel line of work concerns transport and continuity
equations driven by rough velocity fields. Beyond the classical
renormalization theories of DiPerna--Lions~\cite{DPL89} and
Ambrosio~\cite{Ambrosio04}, Modena and Sz\'ekelyhidi
\cite{MS18,MS19}, Modena and Sattig~\cite{MS20} developed
convex-integration constructions producing nonunique and
non-renormalized solutions. Bru\`e,
Colombo, and De Lellis~\cite{BCD21} subsequently constructed
nonunique nonnegative solutions, related Eulerian nonuniqueness to
classical nonuniqueness of characteristic curves, and obtained an
improved uniqueness range for nonnegative densities. The sharpness of
this range was established by Bru\`e, Colombo, and
Kumar~\cite{BCKTransport24} through a new locally self-similar and
intermittently localized flow mechanism. In a complementary direction, Cheskidov and Luo~\cite{CL21}
introduced temporal intermittency into the convex-integration scheme
and obtained nonuniqueness at the critical spatial regularity, under
weaker time-integrability assumptions.

The use of convex integration for active scalar equations goes back
to Shvydkoy~\cite{Shv11}, who constructed nonunique bounded weak
solutions for a broad class of active scalar models. His framework,
however, does not apply to the classical SQG equation, whose drift
operator has an odd Fourier symbol. Isett and Vicol~\cite{IV15}
subsequently developed a H\"older convex-integration theory for active
scalar equations and made this odd-multiplier obstruction precise;
in the two-dimensional odd-multiplier setting, they also established
a Hamiltonian-conservation criterion that applies, in particular, to
SQG. Buckmaster, Shkoller, and Vicol~\cite{BSV19} later overcame this
obstruction through the potential-velocity formulation and constructed
the first nonunique weak solutions of the unforced SQG equation with
a prescribed nonconstant Hamiltonian. Their method also applies to a
range of fractionally dissipative SQG equations. Isett and
Ma~\cite{IM21} subsequently developed an alternative convex-integration scheme acting directly at the scalar level and established a corresponding failure of compactness.

Following these foundational works, the theory developed along two
complementary directions: further flexibility constructions in
stationary, forced, and generalized settings, and sharper rigidity
results for Hamiltonian conservation. On the flexibility side,
Cheng, Kwon, and Li~\cite{CKL21} constructed nontrivial stationary
weak solutions of the unforced SQG equation. Further nonuniqueness
results for the forced SQG and the forced generalized SQG family were
obtained in
\cite{DaiPeng23,BulutHuynhPalasek,
CastroFaracoMengualSolera,MS26}. On the rigidity side, Isett and Ma~\cite{IM24} established
Hamiltonian conservation under conjecturally optimal critical Besov
regularity assumptions and introduced a new formulation revealing
additional conservation laws and structural cancellations for SQG
and modified SQG. The optimality of the resulting threshold for
standard SQG was subsequently confirmed by the flexible
constructions of Isett and Looi~\cite{IL24} and Dai, Giri, and
Radu~\cite{DGR24}.

\medskip

More recently, Gismondi and Radu~\cite{GR25} constructed nontrivial stationary weak paraproduct solutions to the dissipative SQG equation with dissipation 
\(\Lambda^\gamma\) on \(\mathbb T^2\), 
for every
\(0<\gamma\leq2\). Their construction is based on genuinely
two-dimensional intermittent blobs and, in the highly intermittent regime, yields
$\theta\in L^p(\mathbb T^2)$ for every $1\leq p<\frac43$.
Because the constructed solutions lie below the natural
\(\dot H^{-1/2}\) regularity, 
the equation is interpreted in the paraproduct sense. 
In terms of scalar integrability and solution
concept, 
Theorem~\ref{thm:introduction-flexibility} goes beyond
\cite{GR25} by crossing the \(L^{4/3}\) threshold and producing time-dependent inviscid solutions in the standard momentum formulation, 
together with flexibility between prescribed endpoint
states. 
The two results otherwise concern different regimes:
stationary dissipative equations in~\cite{GR25}, and time-dependent inviscid dynamics here.

	\subsection{Overview of the convex-integration scheme}
    
	Our argument is inspired by the moving-dipole convex integration
	scheme introduced by Bru\`e, Colombo, and Kumar
	\cite{BCK26} for the two-dimensional Euler equations. Their scheme
	uses moving Lamb--Chaplygin dipoles and cancels the error through time
	averages of nonperiodic, spatially anisotropic perturbations.
	Adapting this mechanism to SQG requires new ingredients
 because the transport velocity is one derivative more singular than  the potential velocity:
	\[
	u=\Lambda  v.
	\]
	
	\subsubsection*{Traveling SQG building blocks}
	
	The principal building block is the profile corresponding to the
	standard SQG equation within the family of traveling counter-rotating
	circular pairs constructed in \cite{CQZZ23}. On $\mathbb R^2$, its compactly supported Lipschitz
	active scalar satisfies
	\begin{equation*}
		W\partial_1\theta
		+
		\operatorname{div}(u\theta)
		=0,
		\qquad
		u=\nabla^\perp\Lambda^{-1}\theta,
	\end{equation*}
	and has zero total mass. Here $W>0$ is the traveling speed. After
	rescaling the amplitude of $(\theta,u)$, we normalize $W=1$.
	
	For a direction $\xi\in\mathbb S^1$ and a small physical scale $r>0$, we
	rotate and rescale the profile by
	\[
	\theta_{r,\xi}(x)
	=
	r^{-3/2}
		\theta\left(\frac{O_\xi^\top x}{r}\right),
	\]
	where $O_\xi$ is the rotation satisfying $O_\xi e_1=-\xi$.
	The corresponding potential velocity scales as
	\[
	v_{r,\xi}(x)
	=
	r^{-1/2}
		O_\xi v\left(\frac{O_\xi^\top x}{r}\right).
	\]
	In particular,
	\[
	\|v_{r,\xi}\|_{L^4(\mathbb R^2)}
	=\|v\|_{L^4(\mathbb R^2)},
	\qquad
	\|\theta_{r,\xi}\|_{L^{4/3}(\mathbb R^2)}
	=\|\theta\|_{L^{4/3}(\mathbb R^2)}.
	\]
	
	Because the active scalar has zero mass, the potential velocity is a
	gradient outside the support of $\theta_{r,\xi}$. Subtracting this far
	field gives a compactly supported field $V_{r,\xi}^{\mathrm{loc}}$, which can
	be periodized on $\mathbb T^2$, although it is not divergence-free. Adding
	the periodic gradient corrector produces a divergence-free field
	$V_{r,\xi}$ satisfying
	\[
	\operatorname{curl}V_{r,\xi}=-\theta_{r,\xi},
	\qquad
	U_{r,\xi}
	=
	\Lambda \mathbb P_{\neq0}V_{r,\xi}.
	\]
	The physical and localization scales, and the distinction between localized and nonlocal fields,
	are illustrated in Figure~\ref{fig:sqg-traveling-block-geometry}.
	\begin{figure}[t]
		\centering
		\begin{tikzpicture}[
			x=1cm,y=1cm,>=Latex,
			motion/.style={->,cyan!60!black,very thick},
			farfield/.style={->,gray!72,semithick},
			paneltitle/.style={font=\footnotesize\bfseries,align=center}
		]
			\begin{scope}[shift={(0,0)}]
				\fill[gray!3] (0,0) rectangle (5,4);
				\draw[gray!65] (0,0) rectangle (5,4);
				\draw[dashed,gray!75] (2.05,2) circle (1.34);
				\shade[inner color=red!76,outer color=red!12]
					(2.05,2.48) ellipse (.72 and .42);
				\shade[inner color=blue!76,outer color=blue!10]
					(2.05,1.52) ellipse (.72 and .42);
				\node[font=\bfseries,text=red!75!black] at (2.05,2.48) {$+$};
				\node[font=\bfseries,text=blue!85!black] at (2.05,1.52) {$-$};
				\draw[->,red!70!black,semithick]
					(1.48,2.50) arc[start angle=170,end angle=20,radius=.56];
				\draw[->,blue!80!black,semithick]
					(2.62,1.50) arc[start angle=-10,end angle=-160,radius=.56];
				\draw[motion] (3.35,2)--(4.55,2)
					node[midway,above,font=\scriptsize] {$\xi$};
				\node[font=\scriptsize,anchor=west] at (3.26,1.56)
					{translation};
				\node[font=\scriptsize,anchor=south west] at (2.89,2.86)
					{$B_r$};
				\node[font=\scriptsize,align=center] at (2.05,.35)
					{$\operatorname{supp}\theta_{r,\xi}\subset B_r$};
				\node[font=\bfseries] at (.27,3.73) {(a)};
				\node[paneltitle] at (2.5,4.40)
					{Traveling counter-rotating\\scalar pair};
			\end{scope}

			\begin{scope}[shift={(6.15,0)}]
				\fill[gray!3] (0,0) rectangle (5,4);
				\draw[gray!65] (0,0) rectangle (5,4);
				\draw[farfield] (.28,3.18) .. controls (.68,3.62) and (1.15,3.65)
					.. (1.50,3.30);
				\draw[farfield] (4.68,3.05) .. controls (4.26,3.56) and (3.78,3.60)
					.. (3.48,3.27);
				\draw[farfield] (.30,.78) .. controls (.72,.32) and (1.20,.32)
					.. (1.52,.67);
				\draw[farfield] (4.68,.86) .. controls (4.25,.36) and (3.78,.34)
					.. (3.48,.70);
				\fill[cyan!9] (2.5,2) circle (1.55);
				\fill[orange!24,even odd rule]
					(2.5,2) circle (1.55) (2.5,2) circle (.91);
				\fill[white] (2.5,2) circle (.91);
				\fill[cyan!5] (2.5,2) circle (.43);
				\draw[dashed,cyan!55!black,semithick] (2.5,2) circle (1.55);
				\draw[dashed,orange!85!black,semithick] (2.5,2) circle (.91);
				\draw[black!75,semithick] (2.5,2) circle (.43);
				\shade[inner color=red!75,outer color=red!12]
					(2.5,2.14) ellipse (.25 and .12);
				\shade[inner color=blue!75,outer color=blue!10]
					(2.5,1.86) ellipse (.25 and .12);
				\draw[->,thin] (3.12,2.69)--(3.72,3.12);
				\node[font=\scriptsize,anchor=west,align=left] at (3.73,3.13)
					{cutoff annulus\\$r^\alpha<|x|<2r^\alpha$};
				\draw[->,thin] (2.78,2.33)--(3.63,2.48);
				\node[font=\scriptsize,anchor=west] at (3.65,2.48) {$B_r$};
				\draw[->,thin] (3.15,1.36)--(3.72,1.15);
				\node[font=\scriptsize,anchor=west] at (3.73,1.15)
					{$B_{r^\alpha}$};
				\node[font=\scriptsize,anchor=south west] at (.67,3.36)
					{$B_{2r^\alpha}$};
				\node[font=\tiny,align=center,text=gray!75!black]
					at (2.5,.22)
					{nonlocal fields:\ $V_{r,\xi}^{c}=\nabla\phi_{r,\xi}$
					and $U_{r,\xi}$};
				\node[font=\scriptsize,align=center,fill=white,
					fill opacity=.88,text opacity=1,rounded corners=1pt]
					at (2.5,.72)
					{$\operatorname{supp}V_{r,\xi}^{\mathrm{loc}}
					\subset B_{2r^\alpha}$};
				\node[font=\bfseries] at (.27,3.73) {(b)};
				\node[paneltitle] at (2.5,4.40)
					{Localized block and\\nonlocal correction};
			\end{scope}
		\end{tikzpicture}
		\caption{Schematic geometry of an SQG building block (not to scale).
		Panel~\textup{(a)} shows the positive and negative lobes of the active
		scalar in its physical core $B_r$, translating in direction $\xi$.
		Panel~\textup{(b)} separates the physical core $B_r$ from the cutoff
		scales $B_{r^\alpha}\subset B_{2r^\alpha}$ ($0<\alpha<1$).
		The cutoff varies in $B_{2r^\alpha}\setminus B_{r^\alpha}$, and
		subtracting the exterior gradient field localizes
		$V_{r,\xi}^{\mathrm{loc}}$ to $B_{2r^\alpha}$; the periodic
		gradient correction $V_{r,\xi}^{c}$ and the SQG velocity $U_{r,\xi}$
		are nonlocal and are represented by the exterior arrows.}
		\label{fig:sqg-traveling-block-geometry}
	\end{figure}
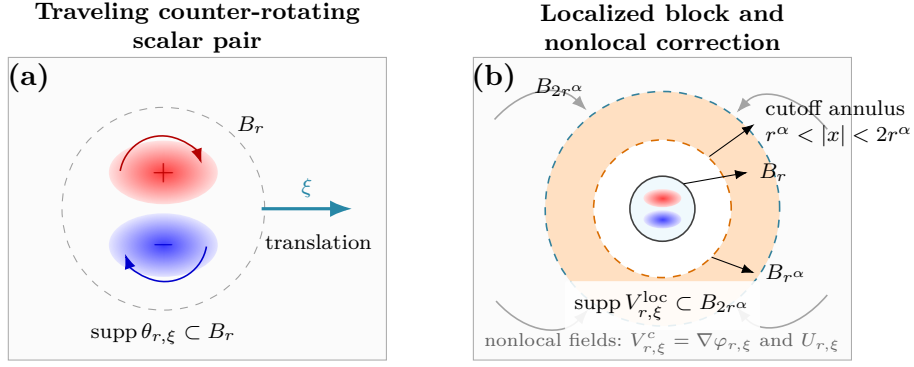
	Let
	$c_1:=\int_{\mathbb R^2}x_2\theta(x)\,dx>0$ be the first-moment constant
	of the fixed profile. A crucial feature is the nonzero mean
	\begin{equation}\label{eq:source-mean-introduction}
		\int_{\mathbb T^2}V_{r,\xi}(x)\,dx
		=
		c_1r^{3/2}\xi.
	\end{equation}
	The block is an approximate solution of the SQG momentum equation:
	\begin{align}\label{eq:constant-speed-block-introduction}
		-r^{-3/2}\xi\cdot\nabla V_{r,\xi}
		+
		U_{r,\xi}\cdot\nabla V_{r,\xi}
		-
		(\nabla V_{r,\xi})^TU_{r,\xi}
		+
		\nabla P_1
		=
		\operatorname{div}R_1.
	\end{align}
	It also satisfies a useful derivative identity with respect to its
	physical scale:
	\begin{equation}\label{eq:r-derivative-block-introduction}
		\partial_rV_{r,\xi}
		=
		\frac{3}{2r}V_{r,\xi}
		-
		\operatorname{div}R_2
		-
		\nabla P_2.
	\end{equation}
	
	\subsubsection*{Variable speed and space-dependent radius}
	
	At the $q$-th stage, the mollified Reynolds stress is decomposed into
	four rank-one directions:
	\begin{equation*}
		-\operatorname{div}R_\ell
		=
		\operatorname{div}
		\left(
		\sum_{i=1}^4a_i(x,t)\xi_i\otimes\xi_i
		\right)
		+\nabla P^d,
	\end{equation*}
	with the sign convention used in the iteration. The geometric
	decomposition is chosen so that $a_i(x,t)\geq\delta_{q+1}>0$, where $\delta_{q+1}$ is roughly the size of $R_{\ell}$ in $L^1(\mathbb{T}^2)$.
     Our principal perturbation \(v_{q+1}\) consists of only one building block moving in
     one of the \(\xi_i\) directions at any given time. More precisely, it will be
     \(\tau_{q+1}\)-periodic in time, and each interval of the form
     $[k\tau_{q+1},(k+1)\tau_{q+1}], \, k\in\mathbb{Z}$,
     will be equally divided into four subintervals and each associated with
        a different direction \(\xi_j\). See Figure~\ref{fig:sqg-four-direction-time-partition}.

	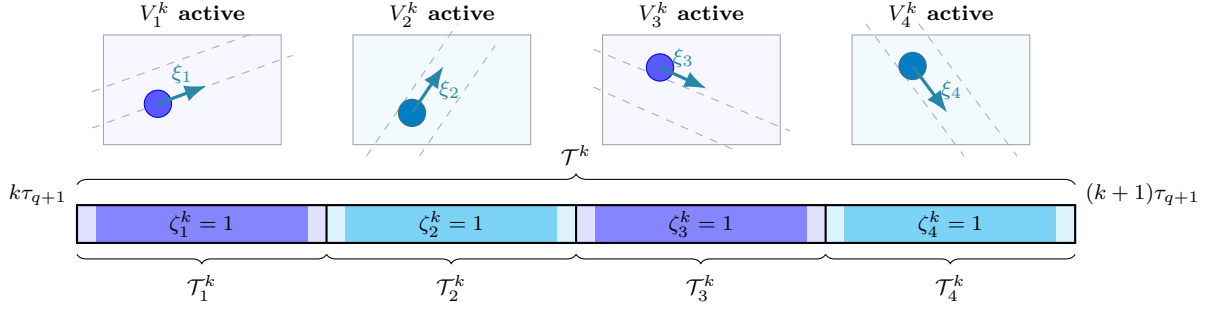
\begin{figure}[t]
		\centering
		\begin{tikzpicture}[
			x=1cm,y=1cm,>=Latex,
			motion/.style={->,cyan!60!black,very thick},
			orbit/.style={gray!65,dashed,thin},
			paneltitle/.style={font=\scriptsize\bfseries,align=center}
		]
			\begin{scope}[shift={(0.35,1.65)}]
				\fill[blue!3] (0,0) rectangle (2.35,1.45);
				\draw[gray!70] (0,0) rectangle (2.35,1.45);
				\draw[orbit] (-.15,.22)--(2.5,1.18);
				\draw[orbit] (-.15,.70)--(1.95,1.46);
				\filldraw[fill=blue!65,draw=blue!90!black]
					(.72,.54) circle (.18);
				\draw[motion] (.72,.54)--++(.65,.24)
					node[midway,above,font=\scriptsize] {$\xi_1$};
				\node[paneltitle] at (1.175,1.72) {$V_1^k$ active};
			\end{scope}
			\begin{scope}[shift={(3.65,1.65)}]
				\fill[cyan!4] (0,0) rectangle (2.35,1.45);
				\draw[gray!70] (0,0) rectangle (2.35,1.45);
				\draw[orbit] (.18,-.15)--(1.38,1.60);
				\draw[orbit] (.88,-.15)--(1.88,1.31);
				\filldraw[fill=cyan!60!blue,draw=cyan!35!black]
					(.78,.42) circle (.18);
				\draw[motion] (.78,.42)--++(.42,.62)
					node[midway,right,font=\scriptsize] {$\xi_2$};
				\node[paneltitle] at (1.175,1.72) {$V_2^k$ active};
			\end{scope}
			\begin{scope}[shift={(6.95,1.65)}]
				\fill[blue!3] (0,0) rectangle (2.35,1.45);
				\draw[gray!70] (0,0) rectangle (2.35,1.45);
				\draw[orbit] (-.12,1.26)--(2.48,.18);
				\draw[orbit] (-.12,.76)--(1.72,-.06);
				\filldraw[fill=blue!65,draw=blue!90!black]
					(.76,1.02) circle (.18);
				\draw[motion] (.76,1.02)--++(.62,-.28)
					node[midway,above,font=\scriptsize] {$\xi_3$};
				\node[paneltitle] at (1.175,1.72) {$V_3^k$ active};
			\end{scope}
			\begin{scope}[shift={(10.25,1.65)}]
				\fill[cyan!4] (0,0) rectangle (2.35,1.45);
				\draw[gray!70] (0,0) rectangle (2.35,1.45);
				\draw[orbit] (.18,1.60)--(1.48,-.15);
				\draw[orbit] (.88,1.60)--(2.18,-.15);
				\filldraw[fill=cyan!60!blue,draw=cyan!35!black]
					(.80,1.04) circle (.18);
				\draw[motion] (.80,1.04)--++(.46,-.62)
					node[midway,right,font=\scriptsize] {$\xi_4$};
				\node[paneltitle] at (1.175,1.72) {$V_4^k$ active};
			\end{scope}

			\foreach \a/\b/\c in {
				0/3.3/blue!12,
				3.3/6.6/cyan!12,
				6.6/9.9/blue!12,
				9.9/13.2/cyan!12}
				\fill[\c] (\a,.35) rectangle (\b,.85);
			\foreach \a/\b/\c in {
				0.25/3.05/blue!48,
				3.55/6.35/cyan!45,
				6.85/9.65/blue!48,
				10.15/12.95/cyan!45}
				\fill[\c] (\a,.35) rectangle (\b,.85);
			\draw[thick] (0,.35) rectangle (13.2,.85);
			\foreach \x in {3.3,6.6,9.9}
				\draw[thick] (\x,.35)--(\x,.85);
			\foreach \x/\i in {1.65/1,4.95/2,8.25/3,11.55/4}{
				\node[font=\scriptsize] at (\x,.60) {$\zeta_{\i}^k=1$};
				\draw[decorate,decoration={brace,mirror,amplitude=3.5pt}]
					({\x-1.65},.20)--({\x+1.65},.20)
					node[midway,below=4pt,font=\scriptsize]
					{$\mathcal T_{\i}^k$};
			}
			\node[font=\scriptsize,anchor=east] at (0,.99)
				{$k\tau_{q+1}$};
			\node[font=\scriptsize,anchor=west] at (13.2,.99)
				{$(k+1)\tau_{q+1}$};
			\draw[decorate,decoration={brace,amplitude=4pt}]
				(0,1.12)--(13.2,1.12)
				node[midway,above=4pt,font=\scriptsize]
				{$\mathcal T^k$};
		\end{tikzpicture}
		\caption{Time partition for the four rational directions. The coarse
		interval $\mathcal T^k$ is divided into four active subintervals
		$\mathcal T_i^k$.  The cutoff $\zeta_i^k$ equals one on the shortened
		central interval $\bar{\mathcal T}_i^k$ and switches smoothly near its
		ends. Since the supports of the four cutoffs are disjoint, at most one
		principal SQG block is active at any time, eliminating interactions
		between distinct principal blocks.}
		\label{fig:sqg-four-direction-time-partition}
	\end{figure}
	
	We freeze the coefficients $a_i$ on short time intervals and denote
	their time averages by $a_i^k(x)$. The physical radius of the
	corresponding SQG block is chosen according to
	\begin{equation*}
		\bigl(r_i^k(x)\bigr)^{3/2}
		=
		r_{q+1}^{3/2}a_i^k(x).
	\end{equation*}
	The exponent $3/2$ is forced by the scaling
	\eqref{eq:source-mean-introduction} and by the source term produced
	when the radius varies.
	
	The center $x_i^k(t)$ of the block travels along a rational direction
	$\xi_i$ according to
	\begin{equation}\label{eq:trajectory-introduction}
		\frac{d}{dt}x_i^k(t)
		=
		\frac{\eta_i^k\zeta_i^k(t)}
		{\bigl(r_i^k(x_i^k(t))\bigr)^{3/2}}\xi_i.
	\end{equation}
	Here $\eta_i^k>0$ is a constant amplitude fixed by the spatial average
	of $a_i^k$, and $\zeta_i^k$ is a smooth temporal cutoff supported in the
	$i$th active subinterval.
	The rational directions are chosen with long period
	$L_i\sim\lambda_{q+1}$. The temporal cutoffs have disjoint supports, so at
	most one direction is active at each time. Consequently, distinct principal
	blocks do not interact; see
	Figure~\ref{fig:sqg-four-direction-time-partition}.
	
	Combining \eqref{eq:constant-speed-block-introduction},
	\eqref{eq:r-derivative-block-introduction}, and
	\eqref{eq:trajectory-introduction}, one obtains
	\begin{align}\label{eq:variable-block-introduction}
		\partial_tV_i^k
		+
		U_i^k\cdot\nabla V_i^k
		-
		(\nabla V_i^k)^TU_i^k
		+
		\nabla P_i^k
		&=
		S_i^k
		\frac{d}{dt}
		\left[
		\eta_i^k\zeta_i^k(t)
		\bigl(r_i^k(x_i^k(t))\bigr)^{3/2}
		\right]
		+
		\operatorname{div}R_i^k,
	\end{align}
	where, after the fixed source normalization used in Section~2,
	\[
	\int_{\mathbb T^2}S_i^k(x,t)\,dx=\xi_i.
	\]
	
	\subsubsection*{Reconstruction of the stress by time averages}
	
	The source in \eqref{eq:variable-block-introduction} is not small
	pointwise in time. Instead, it is designed so that its time average
	reconstructs the old Reynolds stress.
	
	During each active time interval, the trajectory completes many periods.
	Since its speed is inversely proportional to $(r_i^k)^{3/2}$, and hence to
	$a_i^k$, the time spent in a given region is weighted by $a_i^k$.
	Integration by parts in time gives
	\begin{equation*}
		P_{\tau_{q+1}}\mathcal U_{q+1}(x,t)
		=
		\operatorname{div}
		\left(
		\sum_{i=1}^4
		a_i^k(x)\xi_i\otimes\xi_i
		+
		\sum_{i=1}^4G_i^k(x)
		\right),\qquad t\in\mathcal T^k,
	\end{equation*}
	Here $P_{\tau_{q+1}}$ denotes the intervalwise time average,
	$\mathcal U_{q+1}$ is the global auxiliary vector source, and each $G_i^k$
	is a symmetric tensor satisfying
	\[
	\|G_i^k\|_{L^1}
	\lesssim
	\lambda_{q+1}^{-1}
	\|a_i^k\|_{C^1}.
	\]
	Thus the spatial error is reconstructed through temporal averaging
	rather than through the low-frequency interaction of
	periodic plane waves.
	
	An auxiliary localized profile and a time corrector replace the exact source
	by one whose orbit average can be computed explicitly. The corrector is
	chosen to vanish at the endpoints of each coarse time interval.

	Figure~\ref{fig:sqg-source-averaging-schematic} summarizes the geometry of
	the principal localized source and of its auxiliary comparison profile
	during one active subinterval.
	\begin{figure}[t]
		\centering
		\begin{tikzpicture}[
			x=1cm,
			y=1cm,
			>=Latex,
			orbit/.style={gray!70,thin},
			gridline/.style={gray!55,dashed,very thin},
			motion/.style={->,cyan!65!black,semithick},
			paneltitle/.style={font=\footnotesize\bfseries,align=center},
			panelnote/.style={font=\scriptsize,align=center,fill=white,
				fill opacity=.84,text opacity=1,rounded corners=1pt,
				inner sep=2pt}
		]
			\begin{scope}[shift={(0,5.05)}]
				\clip (0,0) rectangle (4.4,3.5);
				\fill[red!5] (0,0) rectangle (4.4,3.5);
				\shade[inner color=red!60,outer color=red!4]
					(0.85,2.55) circle (1.45);
				\shade[inner color=red!45,outer color=red!3]
					(3.55,0.65) circle (1.25);
				\shade[inner color=red!35,outer color=red!2]
					(3.15,2.75) circle (0.95);
				\foreach \x in {1.1,2.2,3.3}
					\draw[gridline] (\x,0)--(\x,3.5);
				\foreach \y in {0.875,1.75,2.625}
					\draw[gridline] (0,\y)--(4.4,\y);
				\foreach \b in {-0.45,0.55,1.55,2.55}
					\draw[orbit] (-0.5,\b)--(4.9,{\b+1.65});
				\filldraw[fill=blue!70,draw=blue!90!black]
					(0.55,0.72) circle (0.43);
				\filldraw[fill=blue!70,draw=blue!90!black]
					(1.55,1.02) circle (0.22);
				\filldraw[fill=blue!70,draw=blue!90!black]
					(2.75,1.39) circle (0.53);
				\filldraw[fill=blue!70,draw=blue!90!black]
					(3.78,1.70) circle (0.30);
				\foreach \x/\y in {0.55/0.72,1.55/1.02,2.75/1.39,3.78/1.70}
					\draw[motion] (\x,\y)--++(0.58,0.18);
				\node[font=\scriptsize,anchor=west,fill=white,
					fill opacity=.8,text opacity=1,rounded corners=1pt]
					at (0.16,3.15) {$a_i^k(x)$ intensity};
			\draw[<->,thin] (2.14,0.86)--(2.14,1.92);
			\node[font=\scriptsize,anchor=east,align=right] at (2.08,1.39)
				{$3(r_i^k(t))^\alpha$};
			\node[panelnote,anchor=east] at (4.22,0.24)
				{radius adapts to $a_i^k$};
			\end{scope}
			\draw[black] (0,5.05) rectangle (4.4,8.55);
			\node[font=\bfseries] at (-0.28,8.35) {(a)};
			\node[paneltitle] at (2.2,8.88)
				{Variable localized source $S_i^k$};

			\begin{scope}[shift={(5.65,5.05)}]
				\clip (0,0) rectangle (4.4,3.5);
				\fill[blue!2] (0,0) rectangle (4.4,3.5);
				\foreach \x in {1.1,2.2,3.3}
					\draw[gridline] (\x,0)--(\x,3.5);
				\foreach \y in {0.875,1.75,2.625}
					\draw[gridline] (0,\y)--(4.4,\y);
				\foreach \b in {-0.45,0.60,1.65,2.70}{
					\draw[blue!58,opacity=.72,line width=5pt,line cap=round]
						(-0.35,\b)--(1.15,{\b+0.46});
					\draw[blue!58,opacity=.72,line width=13pt,line cap=round]
						(1.08,{\b+0.44})--(2.75,{\b+0.95});
				\draw[blue!58,opacity=.72,line width=7pt,line cap=round]
					(2.68,{\b+0.93})--(4.75,{\b+1.56});
			}
			\node[panelnote] at (2.2,0.24)
				{nonuniform spatial coverage};
			\end{scope}
			\draw[black] (5.65,5.05) rectangle (10.05,8.55);
			\node[font=\bfseries] at (5.37,8.35) {(b)};
			\node[paneltitle] at (7.85,8.88)
				{Nonuniform swept region};

			\begin{scope}[shift={(0,0)}]
				\clip (0,0) rectangle (4.4,3.5);
				\fill[red!5] (0,0) rectangle (4.4,3.5);
				\shade[inner color=red!60,outer color=red!4]
					(0.85,2.55) circle (1.45);
				\shade[inner color=red!45,outer color=red!3]
					(3.55,0.65) circle (1.25);
				\foreach \x in {1.1,2.2,3.3}
					\draw[gridline] (\x,0)--(\x,3.5);
				\foreach \y in {0.875,1.75,2.625}
					\draw[gridline] (0,\y)--(4.4,\y);
				\foreach \b in {-0.45,0.55,1.55,2.55}
					\draw[orbit] (-0.5,\b)--(4.9,{\b+1.65});
				\foreach \x/\y in {0.55/0.72,1.55/1.02,2.75/1.39,3.78/1.70}{
					\filldraw[fill=cyan!68!blue,draw=cyan!35!black]
						(\x,\y) circle (0.43);
					\draw[motion] (\x,\y)--++(0.58,0.18);
				}
			\draw[<->,thin] (2.16,0.96)--(2.16,1.82);
			\node[font=\scriptsize,anchor=east,align=right] at (2.10,1.39)
				{$O(\lambda_{q+1}^{-1})$};
			\node[panelnote,anchor=east] at (4.22,0.24)
				{radius fixed in time};
			\end{scope}
			\draw[black] (0,0) rectangle (4.4,3.5);
			\node[font=\bfseries] at (-0.28,3.30) {(c)};
			\node[paneltitle] at (2.2,3.83)
				{Fixed auxiliary profile $\widetilde\Omega_i^k$};

			\begin{scope}[shift={(5.65,0)}]
				\clip (0,0) rectangle (4.4,3.5);
				\fill[blue!32] (0,0) rectangle (4.4,3.5);
				\foreach \b in {-0.70,0.10,0.90,1.70,2.50,3.30}
					\draw[blue!50,opacity=.55,line width=12pt]
						(-0.5,\b)--(4.9,{\b+1.65});
				\foreach \x in {1.1,2.2,3.3}
					\draw[gridline] (\x,0)--(\x,3.5);
				\foreach \y in {0.875,1.75,2.625}
					\draw[gridline] (0,\y)--(4.4,\y);
				\node[font=\scriptsize,align=center,fill=white,
					fill opacity=.88,text opacity=1,rounded corners=2pt,
					inner sep=4pt] at (2.2,1.75)
				{$\displaystyle\frac{1}{c_i\lambda_{q+1}}
				\int_0^{c_i\lambda_{q+1}}$
				\\[-1pt]
				$\widetilde\Omega_i^k(x-(x_0+s\xi_i))\,ds=1$};
			\node[panelnote] at (2.2,0.24)
				{uniform coverage of $\mathbb T^2$};
			\end{scope}
			\draw[black] (5.65,0) rectangle (10.05,3.5);
			\node[font=\bfseries] at (5.37,3.30) {(d)};
			\node[paneltitle] at (7.85,3.83)
				{Uniform auxiliary orbit average};
		\end{tikzpicture}
		\caption{Schematic comparison of the principal and auxiliary sources
		during one active subinterval. Panel~\textup{(a)} shows the localized
		source $S_i^k(\cdot,t)$ following the rational trajectory $x_i^k(t)$;
		its support radius varies with $a_i^k$, represented by the red
		background. Panel~\textup{(b)} shows its resulting nonuniform swept
		region. Panel~\textup{(c)} shows the fixed-width comparison profile
		$\widetilde\Omega_i^k(\cdot-x_i^k(t))$, and panel~\textup{(d)} depicts
		the uniform orbit-average property
		\eqref{eq:Omega-orbit-average}. The disks represent localized source
		profiles only; the full fields $V_i^k$ and $U_i^k$ are nonlocal.}
		\label{fig:sqg-source-averaging-schematic}
	\end{figure}
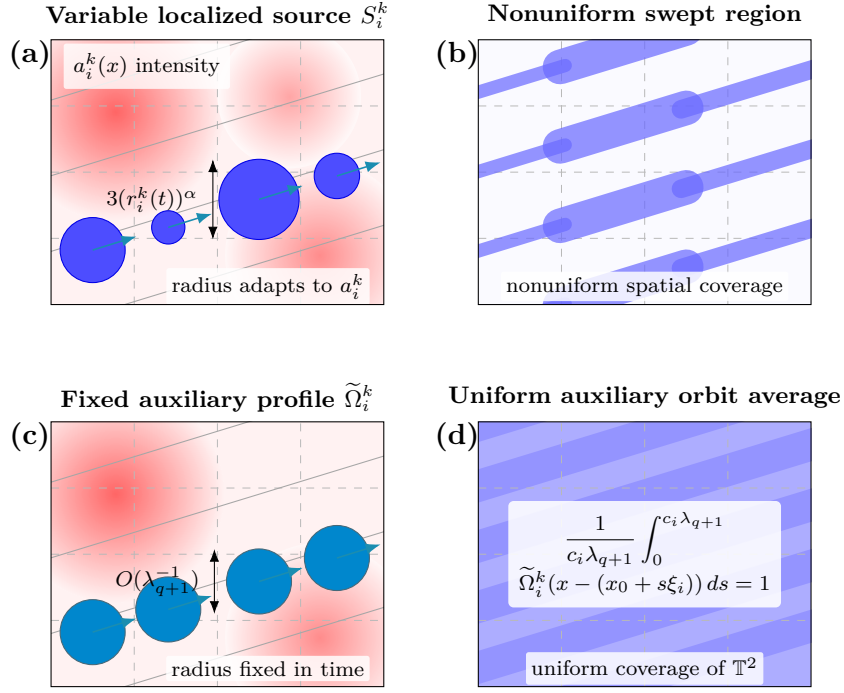
	
	\subsubsection*{The main new difficulty}
	
	Compared with the Euler construction of
	\cite{BCK26}, the principal additional difficulty is the
	nonlocal relation
	\[
	U_i^k=\Lambda \mathbb{P}_{\neq0}V_i^k.
	\]
	In particular, the linear interaction between the new block and the
	mollified background contains one derivative of the concentrated
	profile. A direct estimate leads to the unfavorable quantity
	\[
	\|u_\ell\|_{L^4}
	\|DV_i^k\|_{L^{\bar p}},
	\]
	which incurs a negative power of the concentration radius.
	
	To avoid this loss, the scheme exploits the two-dimensional identity
	\[
	N(U,v)
	:=
	U\cdot\nabla v-(\nabla v)^TU
	=
	(\operatorname{curl}v)U^\perp
	\]
	together with the bilinear null-form estimate for smooth mean-zero
	scalars $f$ and $g$
	\[
	\left\|
	\mathcal R_0\!\left(g\mathcal R^\perp f+
	f\mathcal R^\perp g\right)
	\right\|_{L^\rho}
	\lesssim
	\|f\|_{L^4}\|\Lambda^{-1}g\|_{L^{p_0}},
	\qquad
	\frac1\rho=\frac14+\frac1{p_0}<1,
	\qquad p_0>\frac43.
	\]
	Here $\mathcal R_0$ is the symmetric anti-divergence and
	$\mathcal R^\perp=\nabla^\perp\Lambda^{-1}$. The exponent $p_0$ here is
	an auxiliary exponent, not the target exponent $\bar p$; for example,
	the linear-stress estimate below uses $p_0=2$. The null form replaces an
	unfavorable derivative norm of the concentrated perturbation by a norm
	of its potential velocity. It is essential for closing the linear error,
	but it does not by itself explain why the target $\bar p$ is close to
	$4/3$.

	\paragraph{Why the gain above $4/3$ is small.}
	For $p>4/3$, introduce the concentration exponent
	\begin{equation*}
		s_p
		:=
		\frac32-\frac2p
		=
		\frac{3p-4}{2p}.
	\end{equation*}
	It vanishes at $p=4/3$ and is positive for every gain above the
	endpoint. The differentiated principal block satisfies schematically
	\begin{equation*}
		\|DV_i^k\|_{L_t^\infty L_x^p}
		\lesssim
		\delta_{q+1}^{\frac12-\frac23s_p}
		r_{q+1}^{-s_p}.
	\end{equation*}
	Since $r_{q+1}=\lambda_{q+1}^{-\mu}$ and, up to a fixed prefactor,
	$\delta_{q+1}=\lambda_{q+1}^{-\beta}$, the loss relative to the
	$L^{4/3}$ scaling has frequency exponent
	\[
		\left(\mu+\frac23\beta\right)s_p.
	\]
	To fit the $\delta_{q+1}^{1/100}$ allowance in the displayed
	$C_t^\gamma L_x^{\bar p}$ bookkeeping, the principal term already
	requires
	\begin{equation*}
		\left(\mu+\frac23\beta\right)s_{\bar p}
		<
		\frac{49}{100}\beta.
	\end{equation*}
	The time-H\"older estimates impose slightly stronger versions of this
	inequality. Thus the null form removes a fatal derivative loss, whereas
	the small size of $\bar p-4/3$ is dictated by the amount of profile
	concentration that the remaining parameter hierarchy can absorb.
	More explicitly, if $\varepsilon=p-4/3$, then
	\begin{equation*}
		s_p
		=
		\frac{9\varepsilon}{2(4+3\varepsilon)},
		\qquad
		\varepsilon
		=
		\frac{8s_p}{3(3-2s_p)}.
	\end{equation*}
	For the exponent used in the theorem,
	$p=4/3+10^{-5}$, one has $s_p=9/800006$.

		\subsection{Further directions and open problems}
	\label{subsec:further-directions}
	
	The combination of moving localized profiles, temporal averaging, and the
	SQG null structure suggests the following open problems.
	\begin{enumerate}

		\item \emph{The full subcritical integrability range.}
		Can Theorem~\ref{thm:introduction-flexibility} be extended to every
		exponent \(4/3\leq p<3\), with
		\(\theta\in C([0,1];L^p(\mathbb T^2))\) and endpoint flexibility in
		\(L^p\)? In particular, can one construct Hamiltonian-nonconserving
		weak solutions with integrability arbitrarily close to \(L^3\)? The
		exponent \(p=3\) is the natural upper threshold, since
		\(\theta\in L^3([0,T]\times\mathbb T^2)\) guarantees Hamiltonian
		conservation~\cite{IV15}.
		
		\item \emph{Generalized and dissipative SQG.}
		For the generalized family
		\[
		u=\nabla^\perp\Lambda^{-2+\alpha}\theta,
		\]
		the source-moment scaling, the critical Lebesgue exponent, and the
		derivative loss depend on \(\alpha\). The traveling circular pairs
		of \cite{CQZZ23} suggest that moving localized profiles remain
		available in a range of these models. The main questions are whether
		the temporal reconstruction of the stress survives the modified
		scaling and whether the SQG null form has an appropriate fractional
		analogue. For dissipative equations, even a small fixed dissipation
		produces a high-frequency error not absorbed by the present
		iteration; a frequency-dependent vanishing-dissipation regime may be
		a more accessible first problem.
		
		\item \emph{The whole-space problem.}
		The underlying traveling profiles are defined on \(\mathbb R^2\),
		which makes a nonperiodic construction plausible. The present
		averaging argument, however, relies on closed rational trajectories
		on \(\mathbb T^2\). On the whole space a block escapes instead of
		repeatedly sampling the coefficient it is meant to reconstruct.
		A successful extension would require finite sweeping trajectories or
		a large-box limit, together with uniform control of moments and of
		the nonlocal tails of the potential velocity and the divergence
		correction.
	\end{enumerate}
	\subsection{Organization of the paper}
	
	Section~2 constructs the constant-speed and variable-speed traveling
	SQG blocks and derives their localization and scaling properties. Section~3
	formulates the SQG--Reynolds iteration and introduces the trajectories,
	temporal averaging, and perturbations. Section~4
	estimates the perturbations and the new Reynolds stress and verifies the
	parameter inequalities displayed. The Appendix collects the symmetric
	anti-divergence and null-form estimates used in the construction.

\subsection*{Acknowledgment}
EB's research is funded by the European Research Council (ERC) project MIND, Grant Agreement No.~101219635.  Q-H Nguyen's research is supported by the 
CAS Project for Young Scientists in Basic Research, Grant No.~YSBR-031; and the NSFC under Grant Nos.~1251101538 and 12595282. 
RJ's research is supported by the China Scholarship Council, Grant No. 202506230112.

\subsection*{AI use disclosure}
This paper was written by the authors and was not generated by artificial
intelligence. ChatGPT 5.6 was used only to check spelling and to assist the
authors in checking the correctness of the mathematical calculations and
results. All mathematical ideas, arguments, proofs, and conclusions are those
of the authors, who take full responsibility for the contents of the paper.
    
	\section{Traveling SQG building blocks}
	This section turns the whole-space traveling pair into periodic
	potential-velocity blocks suitable for the convex-integration iteration.
	We first localize and periodize the rescaled profile, and then derive the
	identities needed when its center, amplitude, and radius vary in time.
	
	\subsection{Whole-space traveling profile}
	We begin by recalling the compactly supported traveling active scalar
	that serves as the seed for the construction. Cao, Qin, Zhan, and Zou~\cite{CQZZ23} construct a nontrivial function $\Psi$ satisfying
	\begin{equation*}
		\left\{ \begin{array}{ll}
			&\Lambda \Psi=\lambda (\Psi-c_0x_2)_+\quad\text{in }\mathbb{R}^2_{+}:=\mathbb{R}^2\cap \{x_2>0\},\\&
			\displaystyle\lim_{|x|\to \infty}\Psi(x)=0,\\&
			\operatorname{supp}\Lambda\Psi\subset \overline{B(0,1)},\\&
			\Psi(x_1,0)=0,\\&
			\Psi(x_1,x_2)=-\Psi(x_1,-x_2),
		\end{array}\right.
	\end{equation*}
	for some $\lambda>0$ and $c_0>0$. Moreover,
	$\Psi\in C^{2-}:=\bigcap_{\varepsilon>0}C^{2-\varepsilon}$ and
	$\Lambda\Psi\in C^{0,1}$, with
	\begin{equation*}
		\langle x\rangle |\Psi(x)|
		+\langle x\rangle^2 |\nabla\Psi(x)|
		+\langle x\rangle^{2+\gamma_0}
		\sup_{0<|z|\leq 1}|z|^{-\gamma_0}
		|\nabla\Psi(x)-\nabla\Psi(x-z)|
		\lesssim_{\gamma_0}1
	\end{equation*}
	for every $\gamma_0\in(0,1)$. Define
	$\theta=\lambda(\Psi-c_0x_2)_+$ in the upper half-plane and extend it
	oddly across the $x_1$-axis.

	\subsection{Rotation, scaling, and moments}
	
	We fix a direction \(\xi\in S^1\). Let \(O_\xi\) be a rotation such that
	\(O_\xi e_1=-\xi\). All the objects below are first rotated in \(\mathbb R^2\)
	and then periodized on \(\mathbb T^2\). For simplicity, we keep the notation
	\[
	\theta_r,\quad u_r,\quad v_r,\quad V_r,\quad U_r,
	\]
	and suppress the dependence on \(\xi\).
	
	Let $\theta: \mathbb{R}^2\to\mathbb{R}$ be the traveling SQG profile satisfying
	\[
	W\partial_1\theta+\operatorname{div}(u\theta)=0,
	\qquad
	u=\nabla^\perp\Lambda_{\mathbb R^2}^{-1}\theta ,
	\]
	where $\operatorname{supp}\theta\subset B_1(0)$ and $W>0$ denotes the
	traveling speed. We normalize $W=1$ throughout the paper.
	
	The odd symmetry
	$\theta(x_1,x_2)=-\theta(x_1,-x_2)$ implies that
	\[
	\int_{\mathbb R^2}\theta\,dx=0,\qquad \int_{\mathbb R^2}x_1\theta(x)\,dx=0.
	\]
	Since $\theta(x_1,x_2)\geq0$ for $x_2>0$, define
	\[
		c_1:=\int_{\mathbb R^2}x_2\theta(x)\,dx>0.
	\]
	For \(0<r\ll1\), define the rotated and rescaled functions
	\[
	\theta_r(x)
	:=
	r^{-3/2}\theta\!\left(\frac{O_\xi^Tx}{r}\right),
	\]
	\[
	u_r(x)
	:=
	r^{-3/2}O_\xi u\!\left(\frac{O_\xi^Tx}{r}\right),
	\]
	and
	\[
	v_r(x)
	:=
	\nabla^\perp(-\Delta_{\mathbb R^2})^{-1}\theta_r(x)
	=
	r^{-1/2}O_\xi v\!\left(\frac{O_\xi^Tx}{r}\right).
	\]
	Then
	\begin{equation}
		\label{eq:rescaled-sqg-scalar}
		-r^{-3/2}\xi\cdot\nabla\theta_r
		+
		\operatorname{div}(u_r\theta_r)
		=
		0
		\qquad\text{in }\mathbb R^2 .
	\end{equation}
	
	Let
	\[
	\psi_r:=(-\Delta_{\mathbb R^2})^{-1}\theta_r,
	\qquad
	v_r=\nabla^\perp\psi_r .
	\]
	Since \(\int_{\mathbb R^2}\theta_r\,dx=0\) and \(v_r\) is
	irrotational outside \(B_r(0)\), Stokes' theorem yields a harmonic potential
	\(\Phi_r\) on \(\mathbb R^2\setminus B_r(0)\) such that
	\[
	v_r=\nabla\Phi_r
	\qquad\text{on }\mathbb R^2\setminus B_r(0).
	\]
	Moreover,
	\[
	\Phi_r(x)
	=
	-\frac{c_1}{2\pi}
	\frac{r^{3/2}(x\cdot\xi)}{|x|^2}
	+
	O\left(\frac{r^{5/2}}{|x|^2}\right).
	\]
	Indeed, 
	\[
	\begin{aligned}
		\psi_r(x)
		&=-\frac{1}{2\pi}\int_{\mathbb R^2}\log|x-y|\,\theta_r(y)\,dy\\
		&=-\frac{1}{2\pi}\int_{\mathbb R^2}(\log|x|+\log\Big|\frac{x}{|x|}-\frac{y}{|x|}\Big|)
		\theta_r(y)\,dy\\
		&=-\frac{1}{2\pi}\int_{\mathbb R^2}(\log|x|-\frac{x\cdot y}{|x|^2}+O\!\left(\frac{|y|^2}{|x|^2}\right))
		\theta_r(y)\,dy\\
		&=-\frac{1}{2\pi}\log|x|\int \theta_r(y)\,dy
		+\frac{1}{2\pi}\frac{x\cdot \int y\theta_r(y)\,dy}{|x|^2}
		+\frac{1}{2\pi}\int O\!\left(\frac{|y|^2}{|x|^2}\right)\theta_r(y)\,dy\\
		&=\frac{1}{2\pi}\frac{x\cdot \int y\theta_r(y)\,dy}{|x|^2}
		+\frac{1}{2\pi}\int O\!\left(\frac{|y|^2}{|x|^2}\right)\theta_r(y)\,dy.
	\end{aligned}
	\]
	Since
	\[
	\theta_r(x)
	=
	r^{-3/2}\theta\left(\frac{O_\xi^Tx}{r}\right),
	\]
	the change of variables \(x=rO_\xi y\) gives
	\[
	\begin{aligned}
		\int_{\mathbb R^2}x\theta_r(x)\,dx
		&=
		\int_{\mathbb R^2}
		rO_\xi y\,
		r^{-3/2}\theta(y)\,
		r^2\,dy  \\
		&=
		r^{3/2}O_\xi
		\int_{\mathbb R^2}y\theta(y)\,dy .
	\end{aligned}
	\]
	By the symmetry of the profile,
	\[
	\int_{\mathbb R^2}y\theta(y)\,dy
	=
	c_1e_2 .
	\]
	Since \(O_\xi e_2=-\xi^\perp\), we have
	\begin{equation}\label{eq:xthetar}
		\int_{\mathbb R^2}x\theta_r(x)\,dx
		=
		-c_1r^{3/2}\xi^\perp,
	\end{equation}
	and consequently,
	\begin{equation*}
		\psi_r(x)=-\frac{1}{2\pi} \frac{c_1 r^{\frac32} x\cdot \xi^\perp }{|x|^2}+ O(\frac{r^{\frac52}}{|x|^2}).
	\end{equation*}
	
	\subsection{Localization and periodization}
	The whole-space potential velocity has a noncompact far field, so it
	cannot be periodized directly as a localized block. We first remove its
	exterior gradient part and then restore incompressibility by a periodic
	gradient corrector.

	Let \(\chi\in C^\infty([0,\infty))\) satisfy
	\[
	\chi(s)=0\quad\text{for }s\le1,
	\qquad
	\chi(s)=1\quad\text{for }s\ge2,
	\]
	and set
	\[
	\chi_r(x):=\chi\!\left(\frac{|x|}{r^\alpha}\right),
	\qquad 0<\alpha<1.
	\]
	Define the compactly supported local potential-velocity block
	\[
	V_r^{\mathrm{loc}}
	:=
	v_r-\nabla(\chi_r\Phi_r),
	\]
	which is no longer divergence free:
	\[
	\operatorname{div}V_r^{\mathrm{loc}}
	=
	-\Delta(\chi_r\Phi_r).
	\]
	
	\begin{lemma}\label{lemma-2.1}
		The local potential-velocity block defined above has the following properties:
		\begin{enumerate}
			\item [(i)]
			$\operatorname{supp}V_r^{\mathrm{loc}}\subset B_{2r^\alpha}(0),
			\qquad
			\operatorname{curl}V_r^{\mathrm{loc}}=-\theta_r;$
			
			\vspace{0.6em}
			
			\item [(ii)]
			$\int_{\mathbb R^2}V_r^{\mathrm{loc}}\,dx
			=
			c_1r^{3/2}\xi.$
		\end{enumerate}
	\end{lemma}

	\begin{proof}
		Statement~(i) follows directly from the definition of
		$V_r^{\mathrm{loc}}$. To prove~(ii), note that
		\(V_r^{\mathrm{loc}}\) is compactly supported and
		\[
		\operatorname{curl}V_r^{\mathrm{loc}}=-\theta_r,
		\]
		we can recover the spatial average of \(V_r^{\mathrm{loc}}\) from the first moment of
		\(\operatorname{curl}V_r^{\mathrm{loc}}\). Indeed, by integration by parts,
		\[
		\begin{aligned}
			\int_{\mathbb R^2}x_2\operatorname{curl}V_r^{\mathrm{loc}}\,dx
			&=
			\int_{\mathbb R^2}x_2
			\bigl(\partial_1V_{r,2}^{\mathrm{loc}}
			-\partial_2V_{r,1}^{\mathrm{loc}}\bigr)\,dx =
			\int_{\mathbb R^2}V_{r,1}^{\mathrm{loc}}\,dx,
		\end{aligned}
		\]
		and similarly
		\[
		-\int_{\mathbb R^2}x_1\operatorname{curl}V_r^{\mathrm{loc}}\,dx
		=
		\int_{\mathbb R^2}V_{r,2}^{\mathrm{loc}}\,dx .
		\]
		Therefore, by~\eqref{eq:xthetar},
		\[
		\begin{aligned}
			\int_{\mathbb R^2}V_r^{\mathrm{loc}}\,dx
			&=
			\left(
			\int_{\mathbb R^2}x_2\operatorname{curl}V_r^{\mathrm{loc}}\,dx,
			-\int_{\mathbb R^2}x_1\operatorname{curl}V_r^{\mathrm{loc}}\,dx
			\right)  \\
			&=
			\left(
			-\int_{\mathbb R^2}x_2\theta_r(x)\,dx,
			\int_{\mathbb R^2}x_1\theta_r(x)\,dx
			\right)  \\
			&=
			\left(
			\int_{\mathbb R^2}x\theta_r(x)\,dx
			\right)^\perp \\
			&=
			\left(-c_1r^{3/2}\xi^\perp\right)^\perp
			=
			c_1r^{3/2}\xi,
		\end{aligned}
		\]
		which proves the claim.
	\end{proof}
	
	We now periodize \(V_r^{\mathrm{loc}}\) on \(\mathbb T^2\), and still denote the
	periodized vector field by \(V_r^{\mathrm{loc}}\). Since
	\(\operatorname{supp}V_r^{\mathrm{loc}}\subset B_{2r^\alpha}(0)\) and \(r>0\) is
	sufficiently small, there is no overlap between different periodic copies.
	
	Define the divergence corrector
	\[
	V_r^c
	:=
	\nabla(-\Delta_{\mathbb T^2})^{-1}
	\operatorname{div}V_r^{\mathrm{loc}},
	\]
	where \((-\Delta_{\mathbb T^2})^{-1}\) is taken on mean-zero functions. Set
	\[
	V_r:=V_r^{\mathrm{loc}}+V_r^c .
	\]
	Then
	\[
	\operatorname{div}V_r
	=
	\operatorname{div}V_r^{\mathrm{loc}}
	+
	\Delta(-\Delta_{\mathbb T^2})^{-1}
	\operatorname{div}V_r^{\mathrm{loc}}
	=
	0.
	\]
	Moreover, since \(V_r^c\) is a gradient field, it has zero curl. Hence
	\[
	\operatorname{curl}V_r
	=
	\operatorname{curl}V_r^{\mathrm{loc}}
	=
	-\theta_r .
	\]
	Since \(V_r^c\) is periodic and is a gradient, its spatial integral vanishes:
	\[
	\int_{\mathbb T^2}V_r^c\,dx=0.
	\]
	Therefore, using the identity for \(V_r^{\mathrm{loc}}\) on \(\mathbb R^2\), we obtain
	\begin{equation}
		\label{eq:mean-Vr}
		\int_{\mathbb T^2}V_r\,dx
		=
		\int_{\mathbb T^2}V_r^{\mathrm{loc}}\,dx
		=
		\int_{\mathbb R^2}V_r^{\mathrm{loc}}\,dx
		=
		c_1r^{3/2}\xi .
	\end{equation}
	
	We define the mean-zero part of \(V_r\) by
	\[
	\widetilde V_r
	:=
	V_r-\fint_{\mathbb T^2}V_r\,dx .
	\]
	Equivalently, if the measure on \(\mathbb T^2\) is normalized so that
	\(|\mathbb T^2|=1\), then
	\[
	\widetilde V_r
	=
	V_r-\int_{\mathbb T^2}V_r\,dx
	=
	V_r-c_1r^{3/2}\xi .
	\]
	The vector field \(\widetilde V_r\) is divergence-free, mean-free, and satisfies
	\[
	\operatorname{curl}\widetilde V_r
	=
	-\theta_r .
	\]
	By uniqueness of the mean-zero Biot--Savart potential on the torus, we therefore have
	\[
	\widetilde V_r
	=
	\nabla^\perp(-\Delta_{\mathbb T^2})^{-1}\theta_r .
	\]
	Finally, define the corresponding SQG velocity by
	\[
	U_r
	:=
	\nabla^\perp\Lambda_{\mathbb T^2}^{-1}\theta_r .
	\]
	So,
	\[
	U_r
	=
	\Lambda_{\mathbb T^2}\widetilde V_r .
	\]
	Because \(\Lambda_{\mathbb T^2}\) annihilates constants, this can also be written as
	\[
	U_r
	=
	\Lambda_{\mathbb T^2}V_r,
	\]
	where the last identity is understood as
	\[
	\Lambda_{\mathbb T^2}V_r
	=
	\Lambda_{\mathbb T^2}\mathbb{P}_{\neq0}V_r .
	\]
	In particular,
	\[
	\widetilde V_r
	=
	\Lambda_{\mathbb T^2}^{-1}U_r,
	\qquad
	\widetilde V_r\neq V_r,
	\]
	because \(V_r\) has the nonzero spatial mean given in \eqref{eq:mean-Vr}.
	
	\subsection{Constant-speed block and estimates}
	We now record the two structural identities used later to move a block
	with variable amplitude and radius. The first expresses the traveling
	equation in momentum form, while the second isolates differentiation with
	respect to the physical scale.

	\begin{proposition}[Constant-speed SQG building block]
		\label{prop:sqg-constant-speed}
		Let \(V_r\), \(\widetilde V_r\), and \(U_r\) be defined as above. Then there exist
		pressures \(P_1\), \(P_2\) and symmetric tensors \(R_1\),  \(R_2\) such that
		\begin{equation}
			\label{eq:sqg-constant-speed-equation}
			-r^{-3/2}\xi\cdot\nabla V_r
			+
			U_r\cdot\nabla V_r
			-
			(\nabla V_r)^TU_r
			+
			\nabla P_1
			=
			\operatorname{div}R_1
			\qquad\text{on }\mathbb T^2,
		\end{equation}
		and
		\begin{equation}
			\label{eq:sqg-r-derivative-equation}
			\partial_rV_r
			=
			\frac{3}{2r}V_r
			-
			\operatorname{div}R_2
			-
			\nabla P_2
			\qquad\text{on }\mathbb T^2.
		\end{equation}
		Moreover, for every \(1<q<\infty\),
		\begin{equation}
			\label{eq:Vr-Lq-estimates}
			\|V_r\|_{L^q}
			\lesssim_q r^{\frac2q-\frac12},
			\qquad
			\|DV_r\|_{L^q}
			\lesssim_q r^{\frac2q-\frac32},
			\qquad
			\|D^2V_r\|_{L^q}
			\lesssim_q r^{\frac2q-\frac52},
		\end{equation}
		and
		\begin{equation*}
			\|\partial_rV_r\|_{L^q}
			\lesssim_q r^{2/q-3/2},
			\qquad
			\|D\partial_rV_r\|_{L^q}
			\lesssim_q r^{2/q-5/2}.
		\end{equation*}
		In addition, for \(1<p<\infty\),
		\begin{equation}
			\label{eq:R1-L1}
			\|R_1\|_{L^p(\mathbb T^2)}
			\lesssim_p r^{\frac{2}{p}+1},
		\end{equation}
		and
		\begin{equation}
			\label{eq:R2-L1}
			\|R_2\|_{L^p(\mathbb T^2)}
			\lesssim_p r^{\frac{2}{p}-\frac12}.
		\end{equation}
		Finally,
		\begin{equation*}
			\int_{\mathbb T^2}r^{-3/2}V_r\,dx
			=
			c_1\xi .
		\end{equation*}
	\end{proposition}
	
	\begin{proof}
		We first prove \eqref{eq:sqg-constant-speed-equation}. Since
		\(\theta_r\) and \(u_r\theta_r\) are compactly supported in \(B_r(0)\), the scalar
		equation \eqref{eq:rescaled-sqg-scalar} can be rewritten on \(\mathbb T^2\) as
		\begin{equation}
			\label{eq:periodic-scalar-equation}
			-r^{-3/2}\xi\cdot\nabla\theta_r
			+
			\operatorname{div}(U_r\theta_r)
			=
			\operatorname{div}F_1 ,
		\end{equation}
		where
		\[
		F_1
		:=
		U_r\theta_r
		-
		\operatorname{Per}_{\mathbb T^2}(u_r\theta_r).
		\]
		We claim that \(F_1\) has zero spatial mean. Indeed,
		\[
		\int_{\mathbb T^2}
		\operatorname{Per}_{\mathbb T^2}(u_{r,j}\theta_r)\,dx
		=
		\int_{\mathbb R^2}u_{r,j}\theta_r\,dx
		=
		0
		\]
		by skew-adjointness of the whole-space operator
		\(\nabla^\perp\Lambda_{\mathbb R^2}^{-1}\). Similarly,
		\[
		\int_{\mathbb T^2}U_{r,j}\theta_r\,dx=0
		\]
		by skew-adjointness of the torus operator
		\((\nabla^\perp\Lambda_{\mathbb T^2}^{-1})_j\).
		
		Let
		\[
		\mathcal A:=\nabla^\perp(-\Delta_{\mathbb T^2})^{-1}.
		\]
		Applying \(\mathcal A\) to \eqref{eq:periodic-scalar-equation} and using the fact
		that \(\mathcal A\) commutes with derivatives, gives
		\begin{equation}
			\label{eq:B-applied}
			-r^{-3/2}\xi\cdot\nabla\widetilde V_r
			+
			\mathcal A\operatorname{div}(U_r\theta_r)
			=
			\mathcal A\operatorname{div}F_1 .
		\end{equation}
		Since \(V_r-\widetilde V_r\) is a constant vector, we have
		\[
		\xi\cdot\nabla\widetilde V_r=\xi\cdot\nabla V_r .
		\]
		
		Next, we identify the nonlinear term. Define
		\[
		W_r
		:=
		U_r\cdot\nabla V_r-(\nabla V_r)^TU_r.
		\]
		In components,
		\[
		(W_r)_i
		=
		U_{r,j}\partial_j(V_r)_i
		-
		U_{r,j}\partial_i(V_r)_j
		=
		U_{r,j}\bigl(\partial_j(V_r)_i-\partial_i(V_r)_j\bigr).
		\]
		In two dimensions this gives
		\[
		W_r=(\operatorname{curl}V_r)U_r^\perp=-\theta_r U_r^\perp.
		\]
		Therefore
		\[
		\operatorname{curl}W_r
		=
		-\operatorname{div}(\theta_rU_r)
		=
		-U_r\cdot\nabla\theta_r,
		\]
		because \(\operatorname{div}U_r=0\). On the other hand,
		\[
		\operatorname{curl}\bigl(\mathcal A\operatorname{div}(U_r\theta_r)\bigr)
		=
		-\operatorname{div}(U_r\theta_r).
		\]
		Hence
		\[
		\mathcal A\operatorname{div}(U_r\theta_r)-W_r
		\]
		is curl-free. Also, it is mean-free. Thus there exists a scalar \(P_1\) such that
		\begin{equation}
			\label{eq:B-nonlinear}
			\mathcal A\operatorname{div}(U_r\theta_r)
			=
			U_r\cdot\nabla V_r
			-
			(\nabla V_r)^TU_r
			+
			\nabla P_1 .
		\end{equation}
		Since $\mathcal A\operatorname{div}F_1$ has zero spatial mean on
		$\mathbb{T}^2$, we may apply the symmetric anti-divergence operator
		\(\mathcal R_0\) and set
		\[
		R_1:=\mathcal R_0\mathcal A\operatorname{div}F_1.
		\]
		Then \(R_1\) is symmetric and satisfies
		\[
		\operatorname{div}R_1
		=
		\mathcal A\operatorname{div}F_1 .
		\]
		Combining this identity with \eqref{eq:B-applied} and
		\eqref{eq:B-nonlinear} gives
		\[
		-r^{-3/2}\xi\cdot\nabla V_r
		+
		U_r\cdot\nabla V_r
		-
		(\nabla V_r)^TU_r
		+
		\nabla P_1
		=
		\operatorname{div}R_1.
		\]
		This proves \eqref{eq:sqg-constant-speed-equation}.
		
		We now prove the \(r\)-derivative identity. First,
		\[
		v_r(x)=r^{-1/2}O_\xi v(O_\xi^Tx/r),
		\]
		so
		\[
		\partial_rv_r
		=
		-\frac{1}{2r}v_r-\frac{x}{r}\cdot\nabla v_r.
		\]
		Since \(\operatorname{div}v_r=0\),
		\[
		\operatorname{div}\left(v_r\otimes\frac{x}{r}\right)
		=
		\frac{x}{r}\cdot\nabla v_r+\frac{2}{r}v_r.
		\]
		Thus
		\[
		\partial_rv_r
		=
		\frac{3}{2r}v_r
		-
		\operatorname{div}\left(v_r\otimes\frac{x}{r}\right).
		\]
		We next compute the $r$-derivative of the cutoff term in
		$V_r^{\mathrm{loc}}=v_r-\nabla(\chi_r\Phi_r)$. Since
		\[
		\chi_r(x)=\chi\left(\frac{|x|}{r^\alpha}\right),
		\]
		we have
		\[
		\partial_r\chi_r
		=
		-\frac{\alpha}{r}x\cdot\nabla\chi_r .
		\]
		Moreover, with the normalization of \(\Phi_r\) induced by scaling,
		\[
		\partial_r\Phi_r
		=
		\frac{1}{2r}\Phi_r
		-
		\frac{x}{r}\cdot\nabla\Phi_r .
		\]
		Therefore,
		\[
		\begin{aligned}
			\partial_r(\chi_r\Phi_r)
			&=
			(\partial_r\chi_r)\Phi_r+\chi_r\partial_r\Phi_r  \\
			&=
			-\frac{\alpha}{r}(x\cdot\nabla\chi_r)\Phi_r
			+
			\frac{1}{2r}\chi_r\Phi_r
			-
			\frac{\chi_r}{r}x\cdot\nabla\Phi_r  \\
			&=
			\frac{1}{2r}\chi_r\Phi_r
			-
			\frac{x}{r}\cdot\nabla(\chi_r\Phi_r)
			-
			\frac{\alpha-1}{r}(x\cdot\nabla\chi_r)\Phi_r .
		\end{aligned}
		\]
		Set
		\[
		G_r
		:=
		\frac{\alpha-1}{r}(x\cdot\nabla\chi_r)\Phi_r .
		\]
		Then
		\[
		\partial_r(\chi_r\Phi_r)
		=
		\frac{1}{2r}\chi_r\Phi_r
		-
		\frac{x}{r}\cdot\nabla(\chi_r\Phi_r)
		-
		G_r .
		\]
		Taking a gradient and using that \(\nabla(\chi_r\Phi_r)\) is itself a gradient
		field, we obtain
		\[
		\partial_r\nabla(\chi_r\Phi_r)
		=
		\frac{3}{2r}\nabla(\chi_r\Phi_r)
		-
		\operatorname{div}\left(
		\nabla(\chi_r\Phi_r)\otimes\frac{x}{r}
		\right)
		-
		\nabla G_r .
		\]
		Combining this with the identity for \(\partial_r v_r\), we conclude that
		\[
		\partial_rV_r^{\mathrm{loc}}
		=
		\frac{3}{2r}V_r^{\mathrm{loc}}
		+
		\nabla G_r
		-
		\operatorname{div}\left(
		V_r^{\mathrm{loc}}\otimes\frac{x}{r}
		\right).
		\]
		Finally, since \(x\cdot\nabla\chi_r\) is supported where
		\[
		r^\alpha\le |x|\le 2r^\alpha,
		\]
		the scalar \(G_r\) is supported in the cutoff annulus
		\[
		B_{2r^\alpha}\setminus B_{r^\alpha}.
		\]
		Since
		\[
		V_r^c
		=
		\nabla(-\Delta_{\mathbb T^2})^{-1}
		\operatorname{div}V_r^{\mathrm{loc}},
		\]
		we obtain, using the identity for \(\partial_rV_r^{\mathrm{loc}}\),
		\[
		\begin{aligned}
			\partial_rV_r^c
			&=
			\nabla(-\Delta_{\mathbb T^2})^{-1}
			\operatorname{div}(\partial_rV_r^{\mathrm{loc}}) \\
			&=
			\nabla(-\Delta_{\mathbb T^2})^{-1}
			\operatorname{div}
			\left(
			\frac{3}{2r}V_r^{\mathrm{loc}}
			+\nabla G_r
			-\operatorname{div}
			\left(
			V_r^{\mathrm{loc}}\otimes\frac{x}{r}
			\right)
			\right)  \\
			&=
			\frac{3}{2r}V_r^c
			-\nabla G_r
			-\nabla(-\Delta_{\mathbb T^2})^{-1}
			\operatorname{div}\operatorname{div}
			\left(
			V_r^{\mathrm{loc}}\otimes\frac{x}{r}
			\right).
		\end{aligned}
		\]
		Here we used
		\[
		\nabla(-\Delta_{\mathbb T^2})^{-1}\Delta G_r=-\nabla G_r,
		\]
		up to an irrelevant additive constant.
		
		Adding this identity to the identity for \(\partial_rV_r^{\mathrm{loc}}\), the
		\(\nabla G_r\) terms cancel and we get
		\[
		\partial_rV_r
		=
		\frac{3}{2r}V_r
		-
		\operatorname{div}
		\left(
		V_r^{\mathrm{loc}}\otimes\frac{x}{r}
		\right)
		-
		\nabla(-\Delta_{\mathbb T^2})^{-1}
		\operatorname{div}\operatorname{div}
		\left(
		V_r^{\mathrm{loc}}\otimes\frac{x}{r}
		\right).
		\]
		
		It remains to separate the genuinely non-gradient part from the exterior
		gradient part. Set
		\[
		A_r
		:=
		V_r^{\mathrm{loc}}\otimes\frac{x}{r}.
		\]
		In the exterior region, where
		\[
		V_r^{\mathrm{loc}}
		=
		\nabla\bigl((1-\chi_r)\Phi_r\bigr),
		\]
		we have the identity
		\[
		\operatorname{div}A_r
		=
		\nabla\left[
		\frac{
			x\cdot\nabla\bigl((1-\chi_r)\Phi_r\bigr)
			+
			(1-\chi_r)\Phi_r
		}{r}
		\right].
		\]
		Indeed, if \(f=(1-\chi_r)\Phi_r\), then
		\[
		\operatorname{div}\left(\nabla f\otimes\frac{x}{r}\right)
		=
		\frac{1}{r}
		\left(
		x\cdot\nabla\nabla f+2\nabla f
		\right)
		=
		\nabla\left(
		\frac{x\cdot\nabla f+f}{r}
		\right).
		\]
		
		Choose a smooth radial cutoff \(\zeta_r\in C^\infty_c(B_{3r})\) such that
		\[
		0\le \zeta_r\le 1,
		\qquad
		\zeta_r\equiv1 \ \text{on } B_{2r}.
		\]
		Since \(r\ll r^\alpha\), the region \(\{|x|\ge 2r\}\cap B_{2r^\alpha}\) lies in the
		exterior region where the above gradient identity is valid. Define
		\[
		f_r:=(1-\chi_r)\Phi_r
		\]
		and
		\[
		\Pi_r^{\mathrm{ext}}
		:=
		(1-\zeta_r)
		\frac{x\cdot\nabla f_r+f_r}{r}.
		\]
		Then
		\[
		F_r
		:=
		\operatorname{div}A_r-\nabla\Pi_r^{\mathrm{ext}}
		\]
		is smooth, mean-free, and supported in \(B_{3r}\). Therefore we may define
		\[
		R_2:=\mathcal R_0 F_r,
		\]
		where \(\mathcal R_0\) is the symmetric anti-divergence operator. Then
		\[
		\operatorname{div}R_2=F_r.
		\]
		Consequently,
		\[
		\partial_rV_r
		=
		\frac{3}{2r}V_r
		-
		\operatorname{div}R_2
		-
		\nabla P_2,
		\]
		where
		\[
		P_2
		:=
		\Pi_r^{\mathrm{ext}}
		+
		(-\Delta_{\mathbb T^2})^{-1}
		\operatorname{div}\operatorname{div}
		\left(
		V_r^{\mathrm{loc}}\otimes\frac{x}{r}
		\right).
		\]
		This gives \eqref{eq:sqg-r-derivative-equation}.

		We now prove the estimates for \(V_r\). It is enough to estimate
		\(V_r^{\mathrm{loc}}\), since
		\[
		V_r^c=\nabla(-\Delta_{\mathbb T^2})^{-1}\operatorname{div}V_r^{\mathrm{loc}}
		\]
		is obtained from \(V_r^{\mathrm{loc}}\) by a zero-order Calderón--Zygmund operator on
		\(\mathbb T^2\).
		
		We split the support of \(V_r^{\mathrm{loc}}\) into three regions. Since \(0<\alpha<1\),
		for \(r\ll1\) we have \(r<r^\alpha\).
		First, in the near-field region \(|x|\le 2r\), the cutoff vanishes for
		all sufficiently small $r$, hence
		\[
		V_r^{\mathrm{loc}}=v_r .
		\]
		Recall that
		\[
		v=\nabla^\perp(-\Delta_{\mathbb R^2})^{-1}\theta .
		\]
		Since $\theta\in C^{0,1}$ and $\theta$ is compactly supported, we have
		\[
		\nabla\theta\in L^q(\mathbb R^2)
		\qquad\text{for every }1\leq q\leq \infty.
		\]
		Moreover,
		\[
		D^2v
		=
		T(\nabla\theta),
		\]
		where \(T\) is a zero-order Calderón--Zygmund operator. Hence
		\begin{equation*}
			\|D^2v\|_{L^q(\mathbb R^2)}
			\lesssim_q
			\|\nabla\theta\|_{L^q(\mathbb R^2)}
			<\infty ,\qquad 1<q<\infty.
		\end{equation*}
		Also, since \(\theta\in C_c^{0,1}\subset C_c^{0,\gamma}\) for every
		\(0<\gamma<1\), Calderón--Zygmund Schauder estimates imply
		\[
		Dv\in C^{0,\gamma}_{\mathrm{loc}}(\mathbb R^2),
		\]
		and in particular
		\[
		\|v\|_{L^\infty(B_2(0))}
		+
		\|Dv\|_{L^\infty(B_2(0))}
		<\infty.
		\]
		For $k=0,1$, the local $C^1$ bound on $v$ and pointwise scaling give
		\[
		|D^kv_r(x)|
		\lesssim
		r^{-\frac12-k},
		\qquad |x|\le 2r.
		\]
		Therefore
		\begin{equation}
			\label{eq:core-k01}
			\|D^kv_r\|_{L^q(B_{2r})}
			\lesssim_q
			r^{\frac2q-\frac12-k},
			\qquad k=0,1.
		\end{equation}
		For \(k=2\),
		\begin{equation}
			\label{eq:core-D2}
			\|D^2v_r\|_{L^q(B_{2r})}
			\leq
			r^{-5/2}r^{2/q}
			\|D^2v\|_{L^q(\mathbb R^2)}
			\lesssim_q
			r^{\frac2q-\frac52}.
		\end{equation}
		Combining \eqref{eq:core-k01} and \eqref{eq:core-D2}, we have
		\begin{equation*}
			\|D^kV_r^{\mathrm{loc}}\|_{L^q(B_{2r})}
			\lesssim_q
			r^{\frac2q-\frac12-k},
			\qquad k=0,1,2.
		\end{equation*}

		Second, in the region \(2r<|x|<r^\alpha\), the cutoff is still zero and
		\(V_r^{\mathrm{loc}}=v_r=\nabla\Phi_r\). From the exterior expansion of the harmonic
		potential,
		\[
		|D^kv_r(x)|\lesssim \frac{r^{3/2}}{|x|^{2+k}},
		\qquad 2r<|x|<r^\alpha.
		\]
		Here the separation $|x|\geq2r$ makes the multipole expansion uniform;
		no pointwise second-derivative estimate is used in the near field.
		Thus
		\[
		\|D^kV_r^{\mathrm{loc}}\|_{L^q(2r<|x|<r^\alpha)}
		\lesssim
		r^{3/2}\left(\int_{2r}^{r^\alpha}\rho^{1-(2+k)q}\,d\rho\right)^{1/q}
		\lesssim
		r^{2/q-1/2-k}.
		\]
		
		Finally, in the cutoff annulus \(r^\alpha<|x|<2r^\alpha\),
		\[
		V_r^{\mathrm{loc}}=(1-\chi_r)\nabla\Phi_r-\Phi_r\nabla\chi_r.
		\]
		Using
		\[
		|\Phi_r(x)|\lesssim \frac{r^{3/2}}{|x|},
		\qquad
		|D^m\chi_r|\lesssim r^{-\alpha m},
		\]
		we obtain
		\[
		|D^kV_r^{\mathrm{loc}}(x)|
		\lesssim
		r^{3/2-\alpha(2+k)}.
		\]
		Consequently,
		\[
		\|D^kV_r^{\mathrm{loc}}\|_{L^q(r^\alpha<|x|<2r^\alpha)}
		\lesssim
		r^{3/2-\alpha(2+k)+2\alpha/q}.
		\]
		Since
		\[
		3/2-\alpha(2+k)+2\alpha/q
		-
		(2/q-1/2-k)
		=
		(1-\alpha)\left(2+k-\frac2q\right)>0,
		\]
		this contribution is smaller than \(r^{2/q-1/2-k}\).
		
		Combining the three regions gives
		\[
		\|D^kV_r^{\mathrm{loc}}\|_{L^q(\mathbb T^2)}
		\lesssim_q
		r^{2/q-1/2-k},
		\qquad k=0,1,2.
		\]
		
		Next, we prove the estimates for $\partial_r V_r$.
		From the expression for $\partial_rV_r$ and the Calderón--Zygmund estimates,
		we obtain
		\begin{equation*}
			\|\partial_rV_r\|_{L^q}
			\lesssim \frac1r \|V_r\|_{L^q}+ \|
			\operatorname{div}(V_r^{\mathrm{loc}}\otimes\frac{x}{r})
			\|_{L^q}\lesssim 
			r^{2/q-3/2},
		\end{equation*}
		\[
		\|D\partial_rV_r\|_{L^q}
		\lesssim \frac1r \|D V_r\|_{L^q}+ \|
		D\operatorname{div}(V_r^{\mathrm{loc}}\otimes\frac{x}{r})
		\|_{L^q}\lesssim 
		r^{2/q-5/2},
		\]
		where the estimates of
		$\|\operatorname{div}(V_r^{\mathrm{loc}}\otimes x/r)\|_{L^q}$ and
		$\|D\operatorname{div}(V_r^{\mathrm{loc}}\otimes x/r)\|_{L^q}$ follow
		from the same regional decomposition used for
		$\|DV_r^{\mathrm{loc}}\|_{L^q(\mathbb T^2)}$.

		Next, we estimate the tensor. For \(R_1\), we use the zero mean of \(\theta_r\). The difference between the
		torus SQG kernel and the whole-space SQG kernel is smooth near the origin.
		Let $K_{\mathbb R^2}$ and $K_{\mathbb T^2}$ denote the kernels of
		$\nabla^\perp(-\Delta_{\mathbb R^2})^{-1/2}$ and
		$\nabla^\perp(-\Delta_{\mathbb T^2})^{-1/2}$, respectively. On
		$B_{1/2}(0)$ one may write
		\[
		K_{\mathbb T^2}=K_{\mathbb R^2}+H,
		\qquad H\in C^\infty(B_{1/2}(0)).
		\]
		In particular, $\|\nabla H\|_{L^\infty(B_{1/2})}<\infty$.
		Hence, writing
		\[
		R_1=\mathcal{R}_0\mathcal{A}\div F_1=\mathcal{R}_0\mathcal{A}\div\div \mathcal{R}_0 F_1,
		\]
		for every \(1<p<\infty\), we have
		\begin{equation*}
			\begin{aligned}
				\|R_1\|_{L_x^p}
				&\lesssim_p r\|F_1\|_{L_x^p}                                      \\
				&\lesssim_p r\|U_r-u_r\|_{L^\infty(\operatorname{supp}\theta_r)}
				\|\theta_r\|_{L_x^p}     
				\\
				&\lesssim_p
				r\left\|
				\int H(x-y)\theta_r(y)\,dy
				\right\|_{L^\infty(\operatorname{supp}\theta_r)}
				\|\theta_r\|_{L_x^p}                                                 \\
				&\lesssim_p
				r\left\|
				\int \bigl(H(x-y)-H(x)\bigr)\theta_r(y)\,dy
				\right\|_{L^\infty(\operatorname{supp}\theta_r)}
				\|\theta_r\|_{L_x^p}                                                 \\
				&\lesssim_p
				r\|\nabla H\|_{L^\infty}
				\left(\int_{\mathbb T^2}|y|\,|\theta_r(y)|\,dy\right)
				\|\theta_r\|_{L_x^p}                                                 \\
				&\lesssim_p
				r r^{3/2} r^{\frac{2}{p}-\frac{3}{2}}
				\lesssim_p r^{\frac{2}{p}+1} .
			\end{aligned}
		\end{equation*}
		The first inequality follows from the $L^p$-boundedness of the
		Calderón--Zygmund operator
		$\mathcal R_0\mathcal A\operatorname{div}\operatorname{div}$ and the
		localized anti-divergence estimate in
		Proposition~\ref{prop:antidiv-compact}. The fourth line uses
		$\int_{\mathbb T^2}\theta_r\,dx=0$. This proves \eqref{eq:R1-L1}.
		Finally, we estimate \(R_2\). Since \(F_r\) is smooth, mean-free, and supported
		in a ball of radius \(O(r)\), the localized anti-divergence estimate gives
		\[
		\|R_2\|_{L^p}
		=
		\|\mathcal R_0F_r\|_{L^p}
		\lesssim_p
		r\|F_r\|_{L^p}.
		\]
		On \(B_{3r}\), the core estimates yield
		\[
		|V_r^{\mathrm{loc}}|\lesssim r^{-1/2},
		\qquad
		|DV_r^{\mathrm{loc}}|\lesssim r^{-3/2}.
		\]
		Therefore
		\[
		\left|
		\operatorname{div}
		\left(
		V_r^{\mathrm{loc}}\otimes\frac{x}{r}
		\right)
		\right|
		\lesssim
		\frac{|x|}{r}|DV_r^{\mathrm{loc}}|
		+
		\frac1r|V_r^{\mathrm{loc}}|
		\lesssim
		r^{-3/2}
		\qquad\text{on }B_{3r}.
		\]
		Similarly,
		\[
		|\nabla\Pi_r^{\mathrm{ext}}|\lesssim r^{-3/2}
		\qquad\text{on }B_{3r}.
		\]
		Thus
		\[
		\|F_r\|_{L^p}
		\lesssim_p
		r^{2/p-3/2}.
		\]
		Consequently,
		\[
		\|R_2\|_{L^p}
		\lesssim_p
		r\|F_r\|_{L^p}
		\lesssim_p
		r^{2/p-1/2}.
		\]
		This proves \eqref{eq:R2-L1}.
		
	\end{proof}
	\begin{remark}[The divergence corrector as a pressure gauge]
		\label{rem:divergence-corrector-pressure-gauge}
		Although \(V_r^c\) is nonlocal, no separate estimate of its spatial
		tail is needed in the iteration. Since
		\(\operatorname{div}V_r^{\mathrm{loc}}\) has zero spatial mean, we may set
		\[
			\phi_r
			:=
			(-\Delta_{\mathbb T^2})^{-1}
			\operatorname{div}V_r^{\mathrm{loc}},
			\qquad
			V_r^c=\nabla\phi_r.
		\]
		The operator
		\[
			\nabla(-\Delta_{\mathbb T^2})^{-1}\operatorname{div}
		\]
		is a Calder\'on--Zygmund operator of order zero and commutes with spatial
		derivatives and with \(\partial_r\). Consequently, for \(1<q<\infty\),
		\[
			\|D^mV_r^c\|_{L^q}
			\lesssim_q
			\|D^mV_r^{\mathrm{loc}}\|_{L^q},
			\qquad 0\leq m\leq2,
		\]
		and
		\[
			\|D^m\partial_rV_r^c\|_{L^q}
			\lesssim_q
			\|D^m\partial_rV_r^{\mathrm{loc}}\|_{L^q},
			\qquad 0\leq m\leq1.
		\]
		Moreover, \(D^2\phi_r\) is symmetric. Hence, for every sufficiently
		regular vector field \(A\),
		\[
			A\cdot\nabla V_r^c-(\nabla V_r^c)^TA
			=
			D^2\phi_r A-(D^2\phi_r)^TA
			=
			0.
		\]
		Thus the corrector makes no contribution through the second slot of the
		momentum nonlinearity:
		\[
			A\cdot\nabla V_r-(\nabla V_r)^TA
			=
			A\cdot\nabla V_r^{\mathrm{loc}}
			-(\nabla V_r^{\mathrm{loc}})^TA.
		\]
		Although the associated velocity \(U_r\) remains nonlocal, the full
		self-interaction satisfies
		\[
			U_r\cdot\nabla V_r-(\nabla V_r)^TU_r
			=
			-\theta_rU_r^\perp,
		\]
		and is therefore supported in \(\operatorname{supp}\theta_r\). All other
		terms involving the full profiles \(V_r\) and \(U_r\) are controlled using
		the global \(L^q\) bounds above.

		The gradient structure of \(V_r^c\) is preserved under translations,
		time-dependent changes of radius, multiplication by time-dependent scalar
		coefficients, and time differentiation. More precisely, if \(a=a(t)\),
		\(r=r(t)\), and \(X=X(t)\), then
		\[
			\partial_t\!\left[
				a(t)V_{r(t)}^c(x-X(t))
			\right]
			=
			\nabla_x\partial_t\!\left[
				a(t)\phi_{r(t)}(x-X(t))
			\right].
		\]
		Thus every term obtained by differentiating this expression in time is a
		spatial gradient and may be incorporated into the pressure. In particular,
		\[
			\begin{aligned}
				&r(t)^{-3/2}V_{r(t)}^c(x-X(t))
				\frac{d}{dt}\bigl(\eta(t)r(t)^{3/2}\bigr)
				\\
				&\qquad =
				\nabla_x\!\left[
					r(t)^{-3/2}\phi_{r(t)}(x-X(t))
					\frac{d}{dt}\bigl(\eta(t)r(t)^{3/2}\bigr)
				\right].
			\end{aligned}
		\]
		Here \(r(t)\) and \(\eta(t)\), after evaluation along a fixed trajectory,
		are independent of the spatial variable. Consequently, with \(X(t)=x(t)\),
		the compactly supported source in
		Proposition~\ref{prop:sqg-nonconstant-speed} is
		\[
			S(x,t)
			=
			r(t)^{-3/2}
			V_{r(t)}^{\mathrm{loc}}(x-x(t)).
		\]
		In the source-replacement estimate, compact support is used only for \(S\)
		and its auxiliary local comparison profile, never for the full profiles
		\(V_r\) or \(U_r\). Finally, the spatial smoothing introduced below
		preserves the gradient structure because
		\[
			\rho_\ell*_xV_r^c
			=
			\nabla(\rho_\ell*_x\phi_r).
		\]
	\end{remark}

	\paragraph{Space--time smoothing of the profile family.}
	The iteration needs a profile family that is smooth jointly in space and
	in the radius parameter, while preserving the structural identities above
	up to an absorbable error.
	
	We adapt the regularization device of \cite[Remark~2.2]{BCK26}. Fix a
	compact interval $I\Subset(0,r_0)$ containing the radii of one
	variable-speed block, and let $\rho_\ell$ be a smooth, even spatial
	mollifier on $\mathbb T^2$. The scale $\ell>0$ is fixed with respect to
	$r\in I$. With
	\[
	\mathcal N(U,V):=U\cdot\nabla V-(\nabla V)^TU,
	\]
	define
	\[
	\begin{gathered}
	V_{r,\ell}^{\mathrm{loc}}:=\rho_\ell*_xV_r^{\mathrm{loc}},
	\qquad
	V_{r,\ell}^{c}:=\rho_\ell*_xV_r^c
	=\nabla(-\Delta_{\mathbb T^2})^{-1}
	\operatorname{div}V_{r,\ell}^{\mathrm{loc}},\\
	V_{r,\ell}:=V_{r,\ell}^{\mathrm{loc}}+V_{r,\ell}^{c}
	=\rho_\ell*_xV_r,
	\qquad
	U_{r,\ell}:=\rho_\ell*_xU_r
	=\Lambda_{\mathbb T^2}\mathbb P_{\neq0}V_{r,\ell}.
	\end{gathered}
	\]
	Set
	\[
	\mathcal C_{r,\ell}
	:=\mathcal N(U_{r,\ell},V_{r,\ell})
	-\rho_\ell*_x\mathcal N(U_r,V_r),
	\]
	and, for $j=1,2$, define
	\[
	P_{j,\ell}:=\rho_\ell*_xP_j,
	\qquad
	R_{2,\ell}:=\rho_\ell*_xR_2,
	\qquad
	R_{1,\ell}:=\rho_\ell*_xR_1
	+\mathcal R_0\mathcal C_{r,\ell}.
	\]
	For every linked pair
	$U=\Lambda_{\mathbb T^2}\mathbb P_{\neq0}V$, the skew-adjointness of
	the Riesz transforms gives
	$\int_{\mathbb T^2}\mathcal N(U,V)\,dx=0$. Hence
	$\mathcal C_{r,\ell}$ has zero spatial mean, so the last definition is
	legitimate. Convolving the two identities of
	Proposition~\ref{prop:sqg-constant-speed} and using the definition of
	$\mathcal C_{r,\ell}$ gives the exact equations
	\begin{equation*}
	\begin{aligned}
	-r^{-3/2}\xi\cdot\nabla V_{r,\ell}
	+\mathcal N(U_{r,\ell},V_{r,\ell})
	+\nabla P_{1,\ell}
	&=\operatorname{div}R_{1,\ell},\\
	\partial_rV_{r,\ell}
	&=\frac{3}{2r}V_{r,\ell}
	-\operatorname{div}R_{2,\ell}-\nabla P_{2,\ell}.
	\end{aligned}
	\end{equation*}
	It is essential here that $\ell$ is independent of $r$, so that
	$\partial_r$ commutes with convolution. Spatial convolution with this
	fixed kernel makes the scaled family $C^\infty$ jointly in $(x,r)$ on
	$\mathbb T^2\times I$; consequently, its composition with smooth
	$r(t)$, $x(t)$, and $\eta(t)$ is smooth in space--time.

	Young's inequality preserves all the displayed $L^q$ bounds for
	$V_r$, its spatial derivatives through order two, its $r$-derivative and
	the spatial derivative thereof, as well as the bound for $R_2$. Moreover,
	\[
	\mathcal R_0\mathcal C_{r,\ell}\longrightarrow0
	\quad\text{in }L^p(\mathbb T^2)
	\quad\text{uniformly for }r\in I
	\]
	as $\ell\downarrow0$. Thus, at each iteration stage, one chooses a
	single stage-dependent scale $\ell_q$, independent of $r$ and $t$ and
	common to the finitely many blocks at that stage, sufficiently small to
	absorb this commutator into the stated $R_1$ (and hence $R_3$) budget for
	the exponents used at that stage.
	Convolution also preserves the source normalization exactly:
	\[
	\int_{\mathbb T^2}r^{-3/2}V_{r,\ell}^{\mathrm{loc}}\,dx=c_1\xi.
	\]
	Finally, choosing $\ell_q\leq\inf_{r\in I}r^\alpha$ gives
	\[
	\operatorname{supp}V_{r,\ell}^{\mathrm{loc}}
	\subset B_{2r^\alpha+\ell_q}(0)
	\subset B_{3r^\alpha}(0).
	\]
	The harmless change from the support constant $2$ to $3$ is incorporated
	below. After making this choice, we suppress the subscript $\ell_q$.

	\subsection{Variable-speed and variable-radius blocks}
	We next define the trajectory of the center of a building block along
	a straight line on the two-dimensional torus. In the iteration below, the
	direction $\xi$ will be chosen rational so that this line is periodic.
	Let \(r:\mathbb T^2\to(0,\infty)\) be smooth with \(0<r(x)\ll1\), and let
	\(\eta:\mathbb R\to[0,\infty)\) be smooth. Let the center \(x(t)\) solve the ODE
	\begin{equation}
		\label{eq:sqg-center-ode}
		x'(t)
		=
		\frac{\eta(t)}{r(x(t))^{3/2}}\xi,
		\qquad
		x(t_0)=x_0.
	\end{equation}
	We write
	\[
	r(t):=r(x(t)).
	\]
	Define
	\[
	V^p(x,t)
	:=
	\eta(t)V_{r(t)}(x-x(t)),
	\]
	and
	\[
	U^p(x,t)
	:=
	\eta(t)U_{r(t)}(x-x(t)).
	\]
	
	\begin{proposition}[Variable-speed SQG building block]
		\label{prop:sqg-nonconstant-speed}
		There exist a vector field \(S\), a pressure \(P\), and a symmetric tensor \(R_3\) such that
		\begin{equation}
			\label{eq:sqg-nonconstant-equation}
			\partial_tV^p
			+
			U^p\cdot\nabla V^p
			-
			(\nabla V^p)^TU^p
			+
			\nabla P
			=
			S\frac{d}{dt}\bigl(\eta(t)r(t)^{3/2}\bigr)
			+
			\operatorname{div}R_3,
		\end{equation}
		where 
		\[
		S(x,t)
		:=
		\frac{1}{r(t)^{3/2}}V_{r(t)}^{\mathrm{loc}}(x-x(t)).
		\]
		Moreover,
		\begin{equation*}
			\int_{\mathbb T^2}S(x,t)\,dx
			=
			c_1\xi.
		\end{equation*}
		For every \(1<q<\infty\),
		\begin{equation}
			\label{eq:Vp-Lq-final}
			\|V^p\|_{L_t^\infty L_x^q}
			\lesssim_q
			\|\eta\|_{L_t^\infty}
			\left\|r^{\frac2q-\frac12}\right\|_{L_x^\infty},
		\end{equation}
		\begin{equation}
			\label{eq:DVp-Lq-final}
			\|DV^p\|_{L_t^\infty L_x^q}
			\lesssim_q
			\|\eta\|_{L_t^\infty}
			\left\|r^{\frac2q-\frac32}\right\|_{L_x^\infty},
		\end{equation}
		\begin{equation}
			\label{eq:DDVp-Lq-final}
			\|D^2 V^p\|_{L_t^\infty L_x^q}
			\lesssim_q
			\|\eta\|_{L_t^\infty}
			\left\|r^{\frac2q-\frac52}\right\|_{L_x^\infty},
		\end{equation}
		\begin{equation}
			\label{eq:dtVp-Lq-final}
			\begin{aligned}
				\|\partial_tV^p\|_{L_t^\infty L_x^q}
				\lesssim_q&
				\|\eta'\|_{L_t^\infty}
				\left\|r^{\frac2q-\frac12}\right\|_{L_x^\infty}     \\
				&+
				\|\eta\|_{L_t^\infty}^2
				\left[
				\left\|r^{\frac2q-3}\right\|_{L_x^\infty}
				+
				\|\nabla\log r\|_{L_x^\infty}
				\left\|r^{\frac2q-2}\right\|_{L_x^\infty}
				\right],
			\end{aligned}
		\end{equation}
		and
		\begin{equation}
			\label{eq:dtDVp-Lq-final}
			\begin{aligned}
				\|\partial_tDV^p\|_{L_t^\infty L_x^q}
				\lesssim_q&
				\|\eta'\|_{L_t^\infty}
				\left\|r^{\frac2q-\frac32}\right\|_{L_x^\infty}     \\
				&+
				\|\eta\|_{L_t^\infty}^2
				\left[
				\left\|r^{\frac2q-4}\right\|_{L_x^\infty}
				+
				\|\nabla\log r\|_{L_x^\infty}
				\left\|r^{\frac2q-3}\right\|_{L_x^\infty}
				\right].
			\end{aligned}
		\end{equation}
		Finally, for every $1<q<2$,
		\begin{equation}
			\label{eq:R3-L1-final}
			\|R_3\|_{L_t^\infty L_x^1}
			\le C_q
			\|\eta\|_{L_t^\infty}^2
			\left(
			\|r\|_{L_x^\infty}^{\frac{2}{q}+1}
			+
			\|\nabla\log r\|_{L_x^\infty}
			\|r\|_{L_x^\infty}^{\frac{2}{q}-1}
			\right).
		\end{equation}
	\end{proposition}
	
	\begin{proof}
		We compute
		\[
		\begin{aligned}
			\partial_tV^p
			&=
			\eta'V_r
			+
			\eta r'\partial_rV_r
			-
			\eta x'(t)\cdot\nabla V_r .
		\end{aligned}
		\]
		Here and below, all profile families on the right-hand side are evaluated
		at the radius \(r=r(t)\) and the spatial point \(x-x(t)\).
		By the ODE \eqref{eq:sqg-center-ode},
		\[
		x'(t)=\eta r^{-3/2}\xi.
		\]
		Therefore
		\[
		-\eta x'(t)\cdot\nabla V_r
		=
		-\eta^2r^{-3/2}\xi\cdot\nabla V_r.
		\]
		Using the constant-speed identity
		\eqref{eq:sqg-constant-speed-equation}, we get
		\[
		\begin{aligned}
			-r^{-3/2}\xi\cdot\nabla V_r
			=
			-U_r\cdot\nabla V_r
			+
			(\nabla V_r)^TU_r
			-
			\nabla P_1
			+
			\operatorname{div}R_1 .
		\end{aligned}
		\]
		Thus
		\[
		\begin{aligned}
			\partial_tV^p
			&=
			-U^p\cdot\nabla V^p
			+
			(\nabla V^p)^TU^p
			-
			\nabla(\eta^2P_1)
			+
			\operatorname{div}(\eta^2R_1)      \\
			&\quad
			+
			\eta'V_r
			+
			\eta r'\partial_rV_r .
		\end{aligned}
		\]
		Using
		\[
		\partial_rV_r
		=
		\frac{3}{2r}V_r-\operatorname{div}R_2-\nabla P_2,
		\]
		we obtain
		\[
		\eta r'\partial_rV_r
		=
		\frac{3}{2r}\eta r'V_r
		-
		\eta r'\operatorname{div}R_2
		-
		\eta r'\nabla P_2.
		\]
		Since
		\[
		\frac{d}{dt}(\eta r^{3/2})
		=
		\eta'r^{3/2}
		+
		\frac32\eta r^{1/2}r',
		\]
		we have
		\[
		\eta'V_r+\frac{3}{2r}\eta r'V_r
		=
		r^{-3/2}V_r
		\frac{d}{dt}(\eta r^{3/2}).
		\]
		Therefore
		\[
		\begin{aligned}
			\partial_tV^p
			+
			U^p\cdot\nabla V^p
			-
			(\nabla V^p)^TU^p
			+
			\nabla(\eta^2P_1+\eta r'P_2)
			\\
			=
			r^{-3/2}V_r
			\frac{d}{dt}(\eta r^{3/2})
			+
			\operatorname{div}(\eta^2R_1-\eta r'R_2).
		\end{aligned}
		\]
		Moreover, for the source/sink term,
		\[
		\begin{aligned}
			r^{-3/2}V_r
			\frac{d}{dt}(\eta r^{3/2})
			=&
			r^{-3/2}\left(V_r^{\mathrm{loc}}+V_r^c\right)
			\frac{d}{dt}(\eta r^{3/2})
			\\
			=&
			r^{-3/2} V_r^{\mathrm{loc}}
			\frac{d}{dt}(\eta r^{3/2}) 
			+\nabla\left( r^{-3/2} (-\Delta_{\mathbb T^2})^{-1}\operatorname{div}V_r^{\mathrm{loc}} \frac{d}{dt}(\eta r^{3/2})  \right).
		\end{aligned}
		\]
		
		This proves \eqref{eq:sqg-nonconstant-equation} by setting
		\[
		\begin{aligned}
			P(x,t)
			:=&
			\eta(t)^2P_1(x-x(t))
			+
			\eta(t)r'(t)P_2(x-x(t))
			\\
			&-
			r^{-3/2} (-\Delta_{\mathbb T^2})^{-1}\operatorname{div}V_r^{\mathrm{loc}}(x-x(t)) \frac{d}{dt}(\eta(t) r^{3/2}),
		\end{aligned}
		\]
		and 
		\[
		R_3(x,t)
		:=
		\eta(t)^2R_1(x-x(t))
		-
		\eta(t)r'(t)R_2(x-x(t)).
		\]
		Since \(R_1\) and \(R_2\) are symmetric, the tensor \(R_3\) is symmetric
		as well.
		
		The mean identity follows from \eqref{eq:mean-Vr}:
		\[
		\int_{\mathbb T^2}S(x,t)\,dx
		=
		r(t)^{-3/2}\int_{\mathbb T^2}V_{r(t)}^{\mathrm{loc}}\,dx
		=
		c_1\xi.
		\]
		
		The estimates \eqref{eq:Vp-Lq-final}, \eqref{eq:DVp-Lq-final}, and
		\eqref{eq:DDVp-Lq-final} follow directly from
		\eqref{eq:Vr-Lq-estimates}.
		
		For the time derivative, we use
		\[
		\partial_tV^p
		=
		\eta'V_r+\eta r'\partial_rV_r-\eta x'(t)\cdot\nabla V_r.
		\]
		Furthermore,
		\[
		|x'(t)|\lesssim |\eta|r^{-3/2},
		\]
		and
		\[
		r'(t)
		=
		\nabla r(x(t))\cdot x'(t)
		=
		\eta r^{-3/2}\xi\cdot\nabla r(x(t))
		=
		\eta r^{-1/2}\xi\cdot\nabla\log r(x(t)).
		\]
		Using
		\[
		\|V_r\|_{L^q}\lesssim_q r^{\frac2q-\frac12},
		\qquad
		\|\partial_rV_r\|_{L^q}\lesssim_q r^{\frac2q-\frac32},
		\]
		and
		\[
		\|DV_r\|_{L^q}\lesssim_q r^{\frac2q-\frac32},
		\]
		we get
		\[
		\begin{aligned}
			\|\partial_tV^p\|_{L^q}
			\lesssim_q&
			|\eta'|r^{\frac2q-\frac12}
			+
			|\eta|^2
			\|\nabla\log r\|_{L^\infty}
			r^{\frac2q-2}
			+
			|\eta|^2r^{\frac2q-3}.
		\end{aligned}
		\]
		Taking the supremum in \(t\) gives \eqref{eq:dtVp-Lq-final}.
		
		Similarly,
		\[
		\partial_tDV^p
		=
		\eta'DV_r+\eta r'D\partial_rV_r-\eta x'(t)\cdot D^2V_r.
		\]
		Using
		\[
		\|DV_r\|_{L^q}\lesssim_q r^{\frac2q-\frac32},
		\]
		\[
		\|D\partial_rV_r\|_{L^q}\lesssim_q r^{\frac2q-\frac52},
		\]
		and
		\[
		\|D^2V_r\|_{L^q}\lesssim_q r^{\frac2q-\frac52},
		\]
		we obtain
		\[
		\begin{aligned}
			\|\partial_tDV^p\|_{L^q}
			\lesssim_q&
			|\eta'|r^{\frac2q-\frac32}
			+
			|\eta|^2
			\|\nabla\log r\|_{L^\infty}
			r^{\frac2q-3}
			+
			|\eta|^2r^{\frac2q-4}.
		\end{aligned}
		\]
		Taking the supremum in \(t\) proves \eqref{eq:dtDVp-Lq-final}.
		
		Finally, since
		\[
		R_3=\eta^2R_1-\eta r'R_2,
		\]
		we have
		\[
		\begin{aligned}
			\|R_3\|_{L^q}
			&\le
			|\eta|^2\|R_1\|_{L^q}
			+
			|\eta|\,|r'|\,\|R_2\|_{L^q}          \\
			&\leq
			C_q|\eta|^2 r^{\frac{2}{q}+1}
			+
			C_q|\eta|^2
			\|\nabla\log r\|_{L^\infty}
			r^{-1/2}r^{\frac{2}{q}-\frac12}           \\
			&=
			C_q |\eta|^2
			\left(
			r^{\frac{2}{q}+1}+
			\|\nabla\log r\|_{L^\infty}r^{\frac{2}{q}-1} 
			\right).
		\end{aligned}
		\]
		Since $|\mathbb T^2|=1$, Hölder's inequality gives
		$\|R_3\|_{L^1}\leq\|R_3\|_{L^q}$. Taking the supremum in time gives
		\eqref{eq:R3-L1-final}.
	\end{proof}
	\begin{remark}
		For the remainder of the paper, we normalize the source by
		\begin{equation*}
			\int_{\mathbb T^2} S(x,t)\,dx=\xi.
		\end{equation*}
		This convention is obtained by replacing $S$ with $c_1^{-1}S$ and
		absorbing the fixed factor $c_1$ into the scalar coefficient multiplying
		$S$ in~\eqref{eq:sqg-nonconstant-equation}. Because $c_1$ depends only on
		the fixed traveling profile, this changes only the implicit constants in
		Proposition~\ref{prop:sqg-nonconstant-speed}. We retain the same notation
		after this normalization.
	\end{remark}

	\section{Convex-integration iteration}
	This section assembles the building blocks of Section~2 into one step of
	the SQG--Reynolds iteration. We regularize the background, decompose its
	stress, prescribe the radii and trajectories, and then introduce the
	auxiliary source and time corrector used to define the next stress.
	
	\subsection{The SQG--Reynolds system}
	
	We work with the SQG equation in momentum form. At the \(q\)-th step we
	construct quadruples
	\[
	(v_q,u_q,p_q,R_q)
	\]
	solving the SQG--Reynolds system
	\begin{equation*}
		\left\{
		\begin{aligned}
			\partial_t v_q
			+ u_q\cdot\nabla v_q
			-(\nabla v_q)^T u_q
			+\nabla p_q
			&= \operatorname{div} R_q,\\
			\operatorname{div}v_q&=0,\\
			u_q&=\Lambda \mathbb P_{\neq 0}v_q.
		\end{aligned}
		\right.
	\end{equation*}
	Here \(\mathbb P_{\neq 0}\) denotes the projection onto mean-zero functions on
	\(\mathbb T^2\). We also define the associated active scalar
	\begin{equation*}
		\theta_q:=-\operatorname{curl}v_q.
	\end{equation*}
	Then
	\[
	u_q=\mathcal R^\perp\theta_q
	=(-\mathcal R_2\theta_q,\mathcal R_1\theta_q),
	\]
	and, conversely, by the boundedness of Riesz transforms on \(L^p(\mathbb T^2)\),
	for every \(1<p<\infty\),
	\begin{equation*}
		\|u_q\|_{L^p}
		\sim_p
		\|\theta_q\|_{L^p}.
	\end{equation*}
	Consequently, we have for every \(1<p<\infty\),
	
	\begin{equation*}
		\|u_q\|_{L^p}
		\sim
		\|\theta_q\|_{L^p}
		\sim
		\|\operatorname{curl}v_q\|_{L^p}
		\sim
		\|\nabla v_q\|_{L^p},
	\end{equation*}
	The implicit constants depend only on $p$. We will use this
	equivalence for $p\in[4/3,4]$; over this compact interval the constants may
	be bounded uniformly by a constant $C>0$.

	\subsection{Parameters and iteration statement}
	
	We first specify the scale hierarchy and the inductive estimates that the
	construction is designed to propagate.
	
	Let
	\begin{equation}\label{eq:parameters}
		\lambda_{q+1}=\lambda_q^\sigma,\qquad
		\delta_q=\lambda_1^{2\beta}\lambda_q^{-\beta},\qquad
		r_{q+1}=\lambda_{q+1}^{-\mu},\qquad
		\tau_{q+1}=\lambda_{q+1}^{-\kappa}.
	\end{equation}
	Here $\lambda_0\gg1$ is chosen sufficiently large, $\sigma\in\mathbb N$ is large, and
	$\beta,\mu>0$, $\kappa\in \mathbb{N}$ will be fixed later after all error estimates have been derived. We also assume $\beta\sigma<1$, which will be frequently used in the following sections.

	\begin{proposition}[Iteration step]\label{prop:iteration}
		Let $\bar{p}=\frac{4}{3}+10^{-5}$, $\gamma=10^{-5}$, $n\in\mathbb N$,
		and let the parameters be as in \eqref{eq:parameters}. There
		exists a constant $M\geq 1$
		such that the following holds for every
		$\lambda_0\ge\lambda_0(M)$.
		
		Let
		\[
		(v_q,u_q,p_q, R_q)
		\]
		be a weak solution to the SQG--Reynolds system
		\[
		\begin{cases}
			\partial_t v_q+u_q\cdot\nabla v_q-(\nabla v_q)^Tu_q+\nabla p_q
			=\operatorname{div} R_q,\\
			\operatorname{div}v_q=0,\\
			u_q=\Lambda \mathbb{P}_{\neq0}v_q,
		\end{cases}
		\]
		satisfying
		\begin{align}
			\| R_q\|_{L_t^\infty L_x^1}
			&\le \delta_{q+1},\notag \\
			\|v_q\|_{L_t^\infty L_x^4}
			&\le 2\delta_0^{1/2}-\delta_q^{1/2},\notag \\
			\|D_{t,x}v_q\|_{L_t^\infty L_x^4}+
			\|D_{t,x}u_q\|_{L_t^\infty L_x^4}
			&\le \lambda_q^n. \label{eq:ind-A}
		\end{align}	
		Then there exists another weak solution
		\[
		(v_{q+1},u_{q+1},p_{q+1}, R_{q+1})
		\]
		of the SQG--Reynolds system such that
		\begin{align}
			\| R_{q+1}\|_{L_t^\infty L_x^1}
			&\le \delta_{q+2},\notag \\
			\|v_{q+1}\|_{L_t^\infty L_x^4}
			&\le 2\delta_0^{1/2}-\delta_{q+1}^{1/2},\notag \\
			\|D_{t,x}v_{q+1}\|_{L_t^\infty L_x^4}+
			\|D_{t,x}u_{q+1}\|_{L_t^\infty L_x^4}
			&\le \lambda_{q+1}^n,\notag \\
			\|v_{q+1}-v_q\|_{L_t^\infty L_x^4}
			&\le M\delta_{q+1}^{1/2}, \label{eq:ind-B}\\
			\|D v_{q+1}\|_{C_t^\gamma L_x^{\bar{p}}}
			&\le \|D v_q\|_{C_t^\gamma L_x^{\bar{p}}}
			+\lambda_0\delta_{q+1}^{\frac{1}{100}},\notag\\
            \|Dv_{q+1}-D v_q\|_{C_t^\gamma L_x^{\bar p}} 
            &\le 2 \lambda_0\delta_{q+1}^{\frac{1}{100}}\label{eq:Dv-increment}.
		\end{align}
		with endpoint estimates
		\begin{equation}\label{eq:est-endpoint}
			\|u_{q+1}(\cdot,0)-u_{q}(\cdot,0)\|_{L_x^4}\leq \lambda_q^{-1},\;
			\|u_{q+1}(\cdot,1)-u_{q}(\cdot,1)\|_{L_x^4}\leq \lambda_q^{-1}.
		\end{equation}
	\end{proposition}
	
	\subsection{Proof of Theorem \ref{thm:introduction-flexibility}}
	We now give a complete proof of Theorem \ref{thm:introduction-flexibility} based on Proposition \ref{prop:iteration}.

		Fix mean-zero functions $\theta_{\mathrm{start}},\,\theta_{\mathrm{end}} \in L^{\bar{p}}(\mathbb{T}^2)$ and $\varepsilon>0$. Let $\rho_\ell$ be the standard smooth convolution kernel, and denote
		\[
		\theta_{\mathrm{start}}^\ell:=\theta_{\mathrm{start}}*\rho_{\ell},
		\qquad 
		\theta_{\mathrm{end}}^\ell:=\theta_{\mathrm{end}}*\rho_{\ell}.
		\]
		Then set $\ell$ small enough such that
		\begin{equation}\label{eq:theorem-smooth-approximation}
			\|\theta_{\mathrm{start}}^\ell
			-\theta_{\mathrm{start}}\|_{L^{\bar p}}
			+
			\|\theta_{\mathrm{end}}^\ell
			-\theta_{\mathrm{end}}\|_{L^{\bar p}}
			\leq \frac{\varepsilon}{2}.
		\end{equation}
			Define the associated smoothen potential velocities by
		\[
		v_{\mathrm{start}}^\ell
		:=
		\nabla^\perp(-\Delta_{\mathbb{T}^2})^{-1}
		\theta_{\mathrm{start}}^\ell,
		\qquad
		v_{\mathrm{end}}^\ell
		:=
		\nabla^\perp(-\Delta_{\mathbb{T}^2})^{-1}
		\theta_{\mathrm{end}}^\ell.
		\]
		We now set the $0$-step iteration values by
		\begin{equation}\label{def-0st-SQGR}
			\begin{aligned}
			&v_0(x,t):=
				\chi(t)v_{\mathrm{start}}^\ell (x)
				+
				\bigl(1-\chi(t)\bigr)
				v_{\mathrm{end}}^\ell  (x),\\
			&u_0(x,t)= \bigl(\Lambda_{\mathbb T^2} v_0\bigr)(x,t),\quad \theta_0(x,t)=-\operatorname{curl} v_0(x,t),\\
			&p_0(x,t):=0,\\
			&R_0(x,t):=\mathcal{R}_0
			\left(
			\partial_t v_0+u_0\cdot\nabla v_0-(\nabla v_0)^Tu_0
			\right)(x,t),
			\end{aligned}
		\end{equation} 
		where $\chi\in C^\infty([0,1];[0,1])$ is a smooth time cut-off such that $\chi(t)=1$ for $t\leq 1/4$, and  $\chi(t)=0$ for $t\geq 1/2$. Then $v_0$ and $u_0$ are smooth, divergence-free, and have zero spatial mean. The error $R_0$ is a well-defined symmetric tensor since 
		\[
		\int_{\mathbb T^2}N(u_0,v_0)dx=\int_{\mathbb T^2} \theta_0 \mathcal{R} \theta_0=0,
		\]
        by the skew-adjointness of the Riesz transforms.
		 Notice that both 
		\[
		\delta_0
		=
		\lambda_1^{2\beta}\lambda_0^{-\beta},
		\qquad
		\delta_1
		=
		\lambda_1^\beta,
		\]
		tend to infinity as
		$\lambda_0\to\infty$. Thus, after increasing $\lambda_0$ if necessary, we may arrange that
		\[
		\|R_0\|_{L^\infty_tL^1_x}\leq\delta_1,
		\qquad
		\|v_0\|_{L^\infty_tL^4_x}\leq\delta_0^{1/2},
        \qquad 4C\lambda_0^{-1}\leq \varepsilon,
		\]
		and
		\[
		\|D_{t,x}v_0\|_{L^\infty_tL^4_x}
		+
		\|D_{t,x}u_0\|_{L^\infty_tL^4_x}
		\leq\lambda_0^{n},
		\]
        where $C$ denotes the constant in the $L^{\bar{p}}\to L^{\bar{p}}$ estimate for the Riesz transform.
		Hence Proposition~\ref{prop:iteration} can be applied starting from
		$q=0$, and it produces a sequence
		\[
		(v_q,u_q,p_q,R_q)_{q\geq0}
		\]
		of solutions to the SQG--Reynolds system.
		
		The estimates \eqref{eq:ind-B} and \eqref{eq:Dv-increment} 
        show that 
        \begin{align*}
            \sum_{q\geq k} \|v_{q+1}-v_q\|_{L^\infty_t L^4_x} 
            &\leq M\sum_{q\geq k} \delta_{q+1}^{\frac12} <\infty,\\
           \sum_{q\geq k} \|D v_{q+1}-D v_q\|_{C^\gamma_t L^{\bar{p}}_x} 
           &\leq 2\lambda_0 \sum_{q\geq k}\delta_{q+1}^{\frac{1}{100}}<\infty,
        \end{align*}
        which imply that $(v_q)_{q\in\mathbb{N}}$ is a Cauchy sequence and converges in $C_t(L^4_x \cap W_x^{1,\bar{p}})$ to the limit
		\[
		v(x,t)=v_0(x,t)+ \sum_{q=0}^{\infty} (v_{q+1}(x,t)-v_q(x,t))\in C_t(L^4_x \cap W_x^{1,\bar{p}}).
		\]
        Since there exists a uniform constant $C>0$ such that the following estimates holds for all $q\geq0$ 
        \[
           \|u_{q+1}-u_q\|_{C_t^\gamma L^{\bar{p}}} 
           + \|\theta_{q+1}-\theta_q\|_{C_t^\gamma L^{\bar{p}}}
           \leq
           C \|Dv_{q+1}-D v_q\|_{C_t^\gamma L_x^{\bar p}},
        \]
        where $\theta_q:=-\operatorname{curl}v_q $
        we also obtain its associated active scalar and transport velocity:
        \begin{align}
		\theta(x,t)&=\theta_0(x,t)+ \sum_{q=0}^{\infty} (\theta_{q+1}(x,t)-\theta_q(x,t))\in C_t L^{\bar{p}},\label{eq:theta-sum}\\
        u(x,t)& =u_0(x,t)+ \sum_{q=0}^{\infty} (u_{q+1}(x,t)-u_q(x,t))\in C_t L^{\bar{p}}.\notag
        \end{align}
    Therefore, by \eqref{eq:est-endpoint} and \eqref{eq:theta-sum},
    \[
        \begin{aligned}
            \|\theta(\cdot,0)- \theta_{\mathrm{start}}^\ell\|_{L^{\bar{p}}}
            &=
            \|\theta(\cdot,0)- \theta_0(\cdot,0))\|_{L^{\bar{p}}}\\
            &\le 
            \sum_{q=0}^{\infty} \|\theta_{q+1}(\cdot,0)- \theta_{q}(\cdot,0)\|_{L^{\bar{p}}}\\
            &\le 
            \sum_{q=0}^{\infty}C \|u_{q+1}(\cdot,0)- u_{q}(\cdot,0)\|_{L^{\bar{p}}}\\
             &\le 
            \sum_{q=0}^{\infty}C \|u_{q+1}(\cdot,0)- u_{q}(\cdot,0)\|_{L^4}\\
            &\le
             C\sum_{q=0}^{\infty} \lambda_{q}^{-1} \leq 2C\lambda_0^{-1}\leq \frac{\varepsilon}{2}.
        \end{aligned}
    \]
    In view of \eqref{eq:theorem-smooth-approximation}, we conclude that
    \[
    \|\theta(\cdot,0)-\theta_{\mathrm{start}}\|_{L^{\bar p}}
    \leq\varepsilon.
    \]
    Similarly,
    \[
    \|\theta(\cdot,1)-\theta_{\mathrm{end}}\|_{L^{\bar p}}
    \leq\varepsilon.
    \]
	   We now pass the limit for the momentum weak formulation.
		Since $\bar p>4/3$, the Sobolev embedding 
		$
		W^{1,\bar p}(\mathbb T^2)
		\hookrightarrow
		H^{1/2}(\mathbb T^2)
		$
	   gives
		\begin{equation}\label{eq:theorem-Hhalf-convergence}
			v_q\longrightarrow v
			\quad\text{in}\quad
			C([0,1];H^{1/2}(\mathbb T^2)).
		\end{equation}
		For smooth and divergence-free test function $\phi$,
		we have
		\begin{align}\label{eq:theorem-approximate-weak-form}
			\Bigg|
			\int_0^1
			\langle (v_q)_i,\partial_t\phi_i\rangle
			+
			\langle\Lambda(v_q)_j,(v_q)_i\partial_j\phi_i\rangle
			-
			\frac12
			\langle\partial_i(v_q)_j,
			[\Lambda,\phi_i](v_q)_j\rangle
			\,dt
			\Bigg|
			\leq
			\|R_q\|_{L^\infty_tL^1_x}
			\|\nabla\phi\|_{L^1_tL^\infty_x},
		\end{align}
		and the right-hand side tends to zero since $\|R_q\|_{L^\infty_tL^1_x}
		\leq\delta_{q+1}
		\to 0$.
		Using \eqref{eq:theorem-Hhalf-convergence}, we can pass to
		the limit as $q\to\infty$ in
		\eqref{eq:theorem-approximate-weak-form}, and consequently
		\[
		\int_0^1
		\Big(
		\langle v_i,\partial_t\phi_i\rangle
		+
		\langle\Lambda v_j,v_i\partial_j\phi_i\rangle
		-
		\frac12
		\langle\partial_i v_j,[\Lambda,\phi_i]v_j\rangle
		\Big)\,dt
		=0.
		\]
		Thus $\theta=-\operatorname{curl}v$ is a weak solution of SQG in the
		momentum sense of
		\eqref{eq:weak-momentum-introduction}.

	\subsection{Proof of Theorem \ref{thm:SQG-nonuniqueness}}
	Fix an arbitrary mean-zero scalar
	\[
		\theta_{\mathrm{start}}
		\in L^{\bar p}(\mathbb T^2)
	\]
	and an arbitrary \(\varepsilon>0\). Choose two mean-zero scalars
	\[
		\theta_{\mathrm{end},1},\theta_{\mathrm{end},2}
		\in L^{\bar p}(\mathbb T^2)
	\]
	such that
	\[
		\|\theta_{\mathrm{end},1}
		-\theta_{\mathrm{end},2}\|_{L^{\bar p}}
		>2\varepsilon.
	\]
	Choose one data-mollification scale \(\ell>0\), common to both
	constructions, sufficiently small that
	\[
		\|\theta_{\mathrm{start}}^\ell-\theta_{\mathrm{start}}\|_{L^{\bar p}}
		+
		\sum_{j=1}^2
		\|\theta_{\mathrm{end},j}^\ell
		-\theta_{\mathrm{end},j}\|_{L^{\bar p}}
		\leq \frac{\varepsilon}{2}.
	\]
	For each \(j\in\{1,2\}\), construct a zeroth-level solution
	\[
		(v_0^{(j)},u_0^{(j)},p_0^{(j)},R_0^{(j)})
	\]
	as in \eqref{def-0st-SQGR}, using the same cutoff \(\chi\) and the common
	scale \(\ell\) to mollify the initial target
	\(\theta_{\mathrm{start}}\) and the respective terminal target
	\(\theta_{\mathrm{end},j}\).
	Fix the admissible parameters from \eqref{eq:admissible-parameters} and
	choose one integer \(\lambda_0\), common to both constructions, sufficiently
	large that Proposition~\ref{prop:iteration} applies to both zeroth-level
	solutions and
	\[
		4C\lambda_0^{-1}\leq\varepsilon,
		\qquad
		4\lambda_0^{-1}\leq\frac1{100},
	\]
	where \(C\) is the constant in the endpoint estimate used in the proof of
	Theorem~\ref{thm:introduction-flexibility}.

By the explicit definition in \eqref{def-0st-SQGR}, the two zeroth-level
solutions coincide on \([0,1/4]\). Therefore, by the time-locality
property in Remark~\ref{rem:SQG-time-locality}, the two associated
solutions at level \(q+1\) coincide on the time interval
\[
\Big[
0,\frac14-2\sum_{q'=0}^{q}\lambda_{q'}^{-1}
\Big], \quad\text{with} \quad 2\sum_{q'=0}^{q}\lambda_{q'}^{-1}
\leq 4\lambda_0^{-1}
\leq \frac1{100}.
\]
Consequently, the two limiting solutions, denoted by
\(\widetilde\theta^{(1)}\) and \(\widetilde\theta^{(2)}\), coincide on a
nontrivial time interval containing \(t=0\). In particular,
\[
\widetilde\theta^{(1)}(\cdot,0)
=
\widetilde\theta^{(2)}(\cdot,0)
=:\widetilde\theta_{\mathrm{start}},
\]
where
\[
\|\widetilde\theta_{\mathrm{start}}
-\theta_{\mathrm{start}}\|_{L^{\bar p}}
\leq \varepsilon.
\]

On the other hand, the endpoint estimates yield
\[
\|\widetilde\theta^{(j)}(\cdot,1)
-\theta_{\mathrm{end},j}\|_{L^{\bar p}}
\leq \varepsilon,
\qquad j=1,2.
\]
Hence, by the triangle inequality,
\[
\begin{aligned}
\|\widetilde\theta^{(1)}(\cdot,1)
-\widetilde\theta^{(2)}(\cdot,1)\|_{L^{\bar p}}
&\geq
\|\theta_{\mathrm{end},1}
-\theta_{\mathrm{end},2}\|_{L^{\bar p}}
\\
&\quad
-\|\widetilde\theta^{(1)}(\cdot,1)
-\theta_{\mathrm{end},1}\|_{L^{\bar p}}
-\|\widetilde\theta^{(2)}(\cdot,1)
-\theta_{\mathrm{end},2}\|_{L^{\bar p}}
\\
&>
2\varepsilon-\varepsilon-\varepsilon
=0.
\end{aligned}
\]
Therefore,
$
\widetilde\theta^{(1)}(\cdot,1)
\neq
\widetilde\theta^{(2)}(\cdot,1).
$
Consequently, the two limiting solutions are distinct.

Let \(\mathcal D\) denote the set of all mean-zero initial data in
\(L^{\bar p}(\mathbb T^2)\) that admit at least two distinct momentum weak
solutions in \(C([0,1];L^{\bar p}(\mathbb T^2))\). The construction above
shows that, for every mean-zero
\(\theta_{\mathrm{start}}\in L^{\bar p}(\mathbb T^2)\) and every
\(\varepsilon>0\), there exists
\(\widetilde\theta_{\mathrm{start}}\in\mathcal D\) such that
\[
\|\widetilde\theta_{\mathrm{start}}
-\theta_{\mathrm{start}}\|_{L^{\bar p}}
\leq\varepsilon.
\]
Therefore, \(\mathcal D\) is dense in the mean-zero subspace of
\(L^{\bar p}(\mathbb T^2)\), which proves the theorem.
	
	\subsection{Mollification}
	\leavevmode\par\noindent
	We begin the construction by smoothing the current approximate solution,
	thereby obtaining regular coefficients for the geometric decomposition and
	the trajectories. The nonlinear commutator generated by this operation is
	incorporated into the mollified Reynolds stress.
	
	We work with the SQG--Reynolds system in momentum form
	\begin{equation}\label{eq:SQG-q}
		\partial_t v_q
		+
		u_q\cdot\nabla v_q
		-
		(\nabla v_q)^T u_q
		+
		\nabla p_q
		=
		\operatorname{div} R_q,
	\end{equation}
	with
	\begin{equation*}
		\operatorname{div}v_q=0,
		\qquad
		u_q=\Lambda \mathbb P_{\neq 0} v_q=\mathcal R^\perp\theta_q
		=
		(-\mathcal R_2\theta_q,\mathcal R_1\theta_q)
	\end{equation*}
	
	Note that \(v_q\) may have nonzero spatial average, while \(u_q\) is determined
	only by the mean-zero part of \(v_q\). Thus the notation \(u_q=\Lambda v_q\)
	should always be understood as \(u_q=\Lambda\mathbb P_{\neq0}v_q\).
	
	Define the SQG momentum nonlinearity by
	\begin{equation*}
		N(u,v):=
		u\cdot\nabla v-(\nabla v)^T u=(\operatorname{curl} v)u^{\perp}.
	\end{equation*}
	Then \eqref{eq:SQG-q} can be written as
	\begin{equation}\label{eq:SQG-q-rewrite}
		\partial_t v_q+ N(u_q,v_q)+\nabla p_q
		=
		\operatorname{div} R_q.
	\end{equation}
	
	Let \(\varphi_\ell\) be a standard space--time mollifier at scale
	\(\ell\). After fixing a temporal extension when needed, define
	\begin{equation*}
		v_\ell:=v_q*\varphi_\ell,
		\qquad
		p_\ell:=p_q*\varphi_\ell,
		\qquad
		\theta_\ell:=-\operatorname{curl}v_\ell=\theta_q*\varphi_\ell,
		\qquad
		u_\ell:=u_q*\varphi_\ell=\mathcal R^\perp\theta_\ell.
	\end{equation*}
	Convolving \eqref{eq:SQG-q-rewrite} gives
	\begin{equation*}
		\partial_t v_\ell+
		\big( N(u_q,v_q)\big)_\ell
		+
		\nabla p_\ell
		=
		\operatorname{div}( R_q)_\ell.
	\end{equation*}
	We add and subtract \( N(u_\ell,v_\ell)\). Set
	\begin{equation*}
		\mathcal C_\ell
		:=
		N(u_\ell,v_\ell)
		-
		\big( N(u_q,v_q)\big)_\ell
		= (\operatorname{curl} v_\ell)u_{\ell}^{\perp}-((\operatorname{curl} v_q)u_q^{\perp})_\ell.
	\end{equation*}
	Then
	\begin{equation*}
		\partial_t v_\ell+
		N(u_\ell,v_\ell)
		+
		\nabla p_\ell
		=
		\operatorname{div}( R_q)_\ell
		+
		\mathcal C_\ell.
	\end{equation*}
	By the skew-adjointness of the Riesz transforms and the properties
	of mollification,
	\[
	\int_{\mathbb T^2}\mathcal C_{\ell,i}\,dx
	=
	\int_{\mathbb T^2}\theta_\ell \mathcal R_i\theta_\ell\,dx
	-
	\int_{\mathbb T^2}(\theta_q \mathcal R_i\theta_q)_\ell\,dx
	=
	-\int_{\mathbb R}\varphi_\ell^t(t-s)
	\int_{\mathbb T^2}
	\theta_q(y,s)\mathcal R_i\theta_q(y,s)\,dy\,ds
	=0.
	\]
	
	We now define the mollified Reynolds stress by
	\begin{equation}\label{def-Rl}
		R_\ell
		:=
		(R_q)_\ell
		+
		\mathcal R_0\mathcal C_\ell.
	\end{equation}
	Then we obtain
	\begin{equation}\label{eq:vl}
		\partial_t v_\ell
		+
		u_\ell\cdot\nabla v_\ell
		-
		(\nabla v_\ell)^T u_\ell
		+
		\nabla p_\ell
		=
		\operatorname{div} R_\ell.
	\end{equation}
	Moreover,
	\begin{equation*}
		\operatorname{div}v_\ell=0,
		\qquad
		u_\ell=\mathcal R^\perp\theta_\ell
		=
		\Lambda\mathbb P_{\neq0}v_\ell.
	\end{equation*}
	Thus \((v_\ell,u_\ell,p_\ell, R_\ell)\) solves the SQG-Reynolds system.
	
	\subsubsection{Mollification estimates}
	We first estimate the Reynolds stress $R_\ell$. By the decomposition
	\[
	\begin{aligned}
		\mathcal C_\ell
		&=
		(\theta_\ell-\theta_q)\mathcal R\theta_\ell
		+\theta_q\mathcal R(\theta_\ell-\theta_q)
		+\theta_q\mathcal R\theta_q-(\theta_q\mathcal R\theta_q)_\ell,
	\end{aligned}
	\]
	we have
	\begin{equation}\label{RlL1}
		\begin{aligned}
			\|R_\ell\|_{L^\infty_tL^1_x}
			\leq& \|(R_q)_\ell\|_{L^\infty_tL^1_x}
			+ \|(\theta_\ell-\theta_q)\mathcal R\theta_\ell\|_{L^\infty_tL^2_x} \\
			&+ \|\theta_q \mathcal R(\theta_\ell-\theta_q)\|_{L^\infty_tL^2_x}
			+ \|\theta_q\mathcal R\theta_q-(\theta_q\mathcal R\theta_q)_\ell\|_{L^\infty_tL^2_x}\\
			\leq& \|R_q\|_{L^\infty_tL^1_x}
			+\|\theta_\ell-\theta_q\|_{L^\infty_tL^4_x}
			\|\mathcal R\theta_\ell\|_{L^\infty_tL^4_x}\\
			&+\|\theta_q\|_{L^\infty_tL^4_x}
			\|\mathcal R(\theta_\ell-\theta_q)\|_{L^\infty_tL^4_x}
			+C\ell
			\|D_{t,x}(\theta_q\mathcal R\theta_q)\|_{L^\infty_tL^{2}_x}\\
			\leq& \delta_{q+1}+ C \ell \|D_{t,x}\theta_q\|_{L^\infty_tL^4_x}
			\|\theta_q\|_{L^\infty_tL^4_x}\\
			\leq& \delta_{q+1}+ C\ell \lambda_q^{2n}\\
			\leq& 2\delta_{q+1},
		\end{aligned}
	\end{equation}
	provided
	\begin{equation}\label{def:ell}
	\ell=\lambda_0^{-\beta}\lambda_q^{-2n}\delta_{q+1}
	\end{equation}
	with $\lambda_0\geq \lambda_0(C)$ sufficiently large.

	For any $\rho>2$, by the Sobolev embedding $W^{1,\rho}(\mathbb T^2)\hookrightarrow C^0(\mathbb T^2)$ and the $L^\rho$-boundedness of the anti-divergence operator on mean-zero functions, we have
	\begin{equation}\label{RlC0}
		\begin{aligned}
			\|R_\ell\|_{C^0_{t,x}}
			&\le \|(R_q)_\ell\|_{C^0_{t,x}}
			+\|\mathcal R_0 \mathcal C_\ell\|_{C^0_{t,x}}  \\
			&\leq
			\ell^{-2}\|R_q\|_{L^\infty_tL^1_x}
			+C_{\rho}\|\mathcal C_\ell\|_{L^\infty_tL^\rho_x}                                      \\
			&\leq
			\ell^{-2}\delta_{q+1}
			+C_{\rho}\|\theta_\ell \mathcal R\theta_\ell\|_{L^\infty_tL^\rho_x}
			+C_{\rho}\|(\theta_q \mathcal R\theta_q)_\ell\|_{L^\infty_tL^\rho_x}                    \\
			&\leq
			\ell^{-2}\delta_{q+1}
			+C_{\rho}\ell^{-1+\frac{2}{\rho}}
			\|\theta_q\|_{L^\infty_tL^4_x}
			\|u_q\|_{L^\infty_tL^4_x}                                          \\
			&\leq
			\ell^{-2}\delta_{q+1}
			+C_{\rho} \ell^{-1+\frac{2}{\rho}}\lambda_q^{2n}\\
			&\leq
			\lambda_0^{2\beta}\lambda_1^{-2\beta}\lambda_q^{4n+\beta\sigma}
			+
			C_{\rho} (\lambda_0^{-\beta}\lambda_q^{-2n}\delta_{q+1})^{-1+\frac{2}{\rho}}\lambda_q^{2n}  \\
			&\le
			\lambda_0^{2\beta}\lambda_1^{-2\beta}\lambda_q^{4n+\beta\sigma}
			+C_{\rho}(\lambda_0^{\beta}\lambda_1^{-2\beta})^{1-\frac{2}{\rho}} \lambda_q^{(2n+\beta\sigma)(1-\frac{2}{\rho})+2n} \\
			&\le \lambda_q^{4n+\beta\sigma},
		\end{aligned}
	\end{equation}
	where in the last inequality we first fix $\rho=3$ and then choose $\lambda_0\geq \lambda_0(C_{\rho=3})\geq \lambda_0 (C)$ sufficiently large.

	Similarly, with this choice of $\lambda_0$, for the $C^1_{t,x}$ norm, we have
	\begin{equation}\label{RlC1}
		\begin{aligned}
			\|R_\ell\|_{C^1_{t,x}}
			&\le \|(R_q)_\ell\|_{C^1_{t,x}}
			+\|\mathcal R_0 \mathcal C_\ell\|_{C^1_{t,x}}  \\
			&\leq
			\ell^{-3}\|R_q\|_{L^\infty_tL^1_x}
			+C_{\rho}\|D_{t,x}\mathcal{R}_0 \mathcal C_\ell\|_{L^\infty_t W^{1,\rho}_x}                                      \\
			&\leq 
			\ell^{-3}\delta_{q+1}
			+C_{\rho} \|D_x \mathcal C_\ell\|_{L^\infty_t L^\rho_x}
			+C_{\rho} \|D_t \mathcal C_\ell\|_{L^\infty_t L^\rho_x}\\
			&\leq
			\ell^{-3}\delta_{q+1}
			+C_{\rho} \ell^{-2+\frac{2}{\rho}}
			\|\theta_q\|_{L^\infty_tL^4_x}
			\|u_q\|_{L^\infty_tL^4_x}  \\
			&\leq C \lambda_q^{6n+2\beta\sigma}.
		\end{aligned}
	\end{equation}

	Finally, for the mollified approximate solution itself
	\begin{equation}\label{eq:u-mollification-Du}
		\begin{aligned}
			\|v_\ell-v_q\|_{L_t^\infty L_x^4}
			+
			\|u_\ell-u_q\|_{L_t^\infty L_x^4}
			&\leq
			\ell\|D_{t,x}v_q\|_{L_t^\infty L_x^4}
			+
			\ell\|D_{t,x}u_q\|_{L_t^\infty L_x^4}\\
			&\leq \lambda_0^{-\beta} \lambda_q^{-n} \delta_{q+1}\\
			&= \lambda_0^{\beta\sigma-\beta} \lambda_q^{-n-\frac12 \beta\sigma}\delta_{q+1}^{\frac12},
		\end{aligned}
	\end{equation}
	where the coefficient $\lambda_0^{\beta\sigma-\beta} \lambda_q^{-n-\frac12 \beta\sigma}<1$  for all $q\geq 0$ by choosing
	\begin{equation*}
		n\geq \frac12\beta\sigma.
	\end{equation*} 
	
	\subsection{Geometric error decomposition}
	To match the mollified stress with blocks moving in prescribed directions,
	we decompose it into positive rank-one components associated with rational
	directions on the torus.
	
	The following decomposition is adapted from~\cite{BCK26}.
	\begin{lemma}[Error decomposition]\label{lem:error-decomposition}
		Given $\lambda_{q+1} \ge 8$, $\delta_{q+1} > 0$, there exist unit vectors
		$\xi_1, \xi_2, \xi_3, \xi_4$ in $\mathbb{R}^2$ with rational components
		such that the following holds.
		For every
		\[
		R_\ell \in C^\infty(\mathbb{T}^2 \times [0,1]; \mathrm{Sym}_2),
		\]
		\begin{enumerate}
			\item[(i)] The period of the closed geodesic $s \mapsto s \xi_i$ in $\mathbb{T}^2$ is
			$c_i \lambda_{q+1}$ for $c_i \in [1,3]$,
			
			\item[(ii)] The following decomposition holds:
			\begin{equation}\label{eq:error-decomposition}
				-\operatorname{div}(R_\ell)
				=
				\operatorname{div}\left(
				\sum_{i=1}^{4} a_i(x,t)\xi_i \otimes \xi_i
				\right)
				+ \nabla P^d ,
			\end{equation}
			
			\item[(iii)] The functions $a_i(x,t)$ are smooth and satisfy
			\begin{equation*}
				\begin{aligned}
					&a_i(x,t) \ge \delta_{q+1}, \qquad
					\|a_i\|_{L_t^\infty L_x^1} \le 32(\delta_{q+1}+\|R_\ell\|_{L_t^\infty L_x^1}), \\
					&\|a_i\|_{C_{x,t}^0} \le C(\delta_{q+1}+\|R_\ell\|_{C_{x,t}^0}), \qquad
				\|a_i\|_{C_{x,t}^1} \le
				C\bigl(\delta_{q+1}+\|R_\ell\|_{C_{x,t}^1}\bigr).
				\end{aligned}
			\end{equation*}
		\end{enumerate}
	\end{lemma}

	\subsection{Time partition and averaging}
	\leavevmode\par\noindent
	We now freeze the rank-one coefficients on short coarse intervals and
	assign a disjoint active subinterval to each direction. This separation
	prevents interactions between distinct principal blocks and prepares the
	orbit-averaging argument.
	
	We partition $[0,+\infty)$ into time intervals of length $\tau_{q+1}$. We define
	\begin{equation*}
		\mathcal T^k := [k\tau_{q+1},(k+1)\tau_{q+1}) .
	\end{equation*}
	
	Each interval $\mathcal T^k$ is further divided into four subintervals
	of equal length:
	\begin{equation*}
		\mathcal T_i^k :=
		\left[
		\tau_{q+1}\left(k+\frac{i-1}{4}\right),
		\tau_{q+1}\left(k+\frac{i}{4}\right)
		\right),
		\qquad i=1,2,3,4,\quad k\in\mathbb N .
	\end{equation*}
	We also define the shortened interval
	\begin{equation*}
		\bar{\mathcal T}_i^k :=
		\left[
		\tau_{q+1}\left(k+\frac{i-1}{4}+\frac{1}{\lambda_{q+1}}\right),
		\tau_{q+1}\left(k+\frac{i}{4}-\frac{1}{\lambda_{q+1}}\right)
		\right),
		\qquad i=1,2,3,4,\quad k\in\mathbb N .
	\end{equation*}
	
	\begin{definition}[Time-average operator]
		Given a scalar-, vector-, or tensor-valued field $g$ on
		$\mathbb T^2\times\mathbb R_+$, define the following operator componentwise:
		\begin{equation*}
		(P_\tau g)(x,t)
			:=
			\frac{1}{|\mathcal T^k|}\int_{\mathcal T^k} g(x,s)\,ds,
			\qquad t\in \mathcal T^k .
		\end{equation*}
	\end{definition}
	Given $R_\ell$ as in \eqref{def-Rl}, we define $a_i(x,t)$ from Lemma~\ref{lem:error-decomposition}. Next, using the definition of $P_\tau$ above, we introduce the shorthand notation
	\begin{equation*}
		a_i^k(x)
		:=
		P_{\tau_{q+1}}a_i(x,t)
		=
		\frac{1}{|\mathcal T^k|}\int_{\mathcal T^k} a_i(x,t)\,dt.
	\end{equation*}
	From Lemma~\ref{lem:error-decomposition}, \eqref{RlL1}, \eqref{RlC0} and \eqref{RlC1}, we conclude
	\begin{equation}\label{est-aik}
		\begin{aligned}
			&a_i^k(x)\geq \delta_{q+1},
			\qquad
			\|a_i^k\|_{L_x^1}\leq 96\delta_{q+1},
			\\
			&\|a_i^k\|_{C_x^0}\leq C \lambda_q^{4n+\beta\sigma},\qquad	\|a_i^k\|_{C_x^1}\leq C \lambda_q^{6n+2\beta\sigma}.
		\end{aligned}
	\end{equation}
	
	Also, note that
	\begin{equation}\label{est-time-error-a}
		\sup_{t\in\mathcal T^k}
		\|a_i^k-a_i(\cdot,t)\|_{L_x^1}
		\leq
		\tau_{q+1}\|a_i\|_{C_{x,t}^1}
		\leq
		C\tau_{q+1}\lambda_q^{6n+2\beta\sigma},
		\qquad i=1,2,3,4,\quad k\in\mathbb N.
	\end{equation}

	\subsection{Space-dependent radii}
	\label{subsec:SQG-space-dependent-radius}
	\leavevmode\par\noindent
	The spatial mean of the potential-velocity block $V_r$ is proportional to
	$r^{3/2}$. We therefore encode the frozen coefficient $a_i^k$ in the radius
	through the following choice.
	
	We define the space-dependent physical radius \(r_i^k\) by
	\begin{equation*}
		\bigl(r_i^k(x)\bigr)^{3/2}
		:=
		r_{q+1}^{3/2}a_i^k(x).
	\end{equation*}
	Equivalently,
	\begin{equation*}
		r_i^k(x)
		=
		r_{q+1}\bigl(a_i^k(x)\bigr)^{2/3}.
	\end{equation*}
	
	From \eqref{est-aik}, we obtain the following basic estimates for
	$r_i^k(x)$:
	\begin{enumerate}
		\item[(i)] \refstepcounter{equation}\label{rik-lower}
		$\displaystyle
		r_i^k(x) \ge r_{q+1}\delta_{q+1}^{2/3}.
		$
		\hfill\textup{(\theequation)}
		
		\item[(ii)] \refstepcounter{equation}\label{rik-upper}
		$\displaystyle
		\|r_i^k\|_{L_x^\infty}
		\leq
		Cr_{q+1}
		\lambda_q^{\frac23(4n+\beta\sigma)}.
		$
		\hfill\textup{(\theequation)}
		
		\item[(iii)] \refstepcounter{equation}\label{eq:grad-log-rik}
		$\displaystyle
		\|\nabla \log r_i^k\|_{L_x^\infty}
		=\|\frac23\frac{\nabla a_i^k}{a_i^k}\|_{L_x^\infty}
		\leq
		\delta_{q+1}^{-1}\|a_i^k\|_{C_x^1}
		\leq C \lambda_q^{6n+3\beta\sigma}.
		$
		\hfill\textup{(\theequation)}
	\end{enumerate}
	
	The regularized local core of the SQG building block has radius at most
	$3(r_i^k)^\alpha$. We now impose
	\begin{equation}
		\label{eq:SQG-support-condition}
		3\|r_i^k\|_{L_x^\infty}^{\alpha}
		\leq
		\lambda_{q+1}^{-1}.
	\end{equation}
	Using \eqref{rik-upper}, this follows from the stronger parameter inequality
	\begin{equation*}
		C
		r_{q+1}
		\lambda_q^{\frac23(4n+\beta\sigma)}
		\leq
		\lambda_{q+1}^{-1/\alpha},
	\end{equation*}
	provided
	\begin{equation*}
		\frac{8n}{3\sigma}+\frac23\beta <\mu-\frac{1}{\alpha}
	\end{equation*}
	
	\subsection{Core trajectories}
	\label{subsec:SQG-core-trajectory}
	\leavevmode\par\noindent
	For each frozen coefficient and rational direction, we choose the amplitude
	and trajectory so that the core completes many closed periods during its
	active subinterval and samples the coefficient with the required weight.
	
	Given $k\in\mathbb N$ and $i\in\{1,2,3,4\}$, define a smooth time-cutoff function
	\[
	\zeta_i^k:\mathbb R\to[0,1]
	\]
	associated to the intervals $\mathcal T_i^k$ satisfying
	\[
	\operatorname{supp}\zeta_i^k\subset \mathcal T_i^k,
	\qquad
	\zeta_i^k\equiv 1 \quad \text{on } \bar{\mathcal T}_i^k,
	\]
	and
	\begin{equation*}
		\left\|
		\frac{d}{dt}\zeta_i^k
		\right\|_{L_t^\infty}
		\leq
		\frac{10\lambda_{q+1}}{\tau_{q+1}} .
	\end{equation*}

	We define the amplitude function \(\eta_i^k\) by
	\begin{equation*}
		(\eta_i^k)^2
		=
		4\int_{\mathbb T^2}a_i^k(x)\,dx.
	\end{equation*}
	In particular, by \eqref{est-aik}, we have
	\begin{equation*}
		\eta_i^k \in [2\delta_{q+1}^{\frac12},\sqrt{384}\delta_{q+1}^{\frac12}].
	\end{equation*}
	The cutoff function used in
	Proposition~\ref{prop:sqg-nonconstant-speed} has the form
	\begin{equation*}
		\eta(t):=\eta_i^k\zeta_i^k(t).
	\end{equation*}
	
	Consequently, the trajectory of the center of the core is defined by the ODE
	\begin{equation}
		\label{eq:xik-ODE-SQG-physical}
		\frac{d}{dt}x_i^k(t)
		=\frac{
			\eta_i^k\zeta_i^k(t)}{r_i^k(x_i^k(t))^{3/2}
		}\xi_i,\qquad t\in \mathcal T_i^k.
	\end{equation}
	For brevity, set $r_i^k(t):=r_i^k(x_i^k(t))$.
	Let
	\[
	t_0
	:=
	\tau_{q+1}
	\left(
	k+\frac{i-1}{4}+\lambda_{q+1}^{-1}
	\right).
	\]
	We choose the point \(x_i^k(t_0)\in \mathbb T^2\) satisfying
	\begin{equation*}
		\fint_0^{c_i\lambda_{q+1}}
		a_i^k\bigl(x_i^k(t_0)+s\xi_i\bigr)\,ds
		=
		\int_{\mathbb T^2}a_i^k(x)\,dx= \frac{(\eta_i^k)^2}{4}.
	\end{equation*}
	Such a point exists by averaging \(x_i^k(t_0)\) over $\mathbb{T}^2$ and using
	Fubini's theorem.
	We solve the ODE \eqref{eq:xik-ODE-SQG-physical} both forward and backward in time starting at $t_0$ with position $x_i^k(t_0)$ to
	obtain the trajectory on the entire interval $\mathcal{T}_i^k$.
	
	On the region where \(\zeta_i^k\equiv 1\), the time needed to complete one period
	along the rational line is
	\begin{equation}
		\begin{aligned}
			\label{eq:SQG-period-physical}
			T_i^k
			&=
			\int_0^{c_i\lambda_{q+1}}
			\frac{\bigl(r_i^k(x_i^k(t_0)+s\xi_i)\bigr)^{3/2}}
			{\eta_i^k}\,ds    \\
			&=
			\frac{r_{q+1}^{3/2}}
			{\eta_i^k}
			\int_0^{c_i\lambda_{q+1}}
			a_i^k(x_i^k(t_0)+s\xi_i)\,ds   \\
			&=
			\frac{c_i\lambda_{q+1} r_{q+1}^{3/2}}
			{\eta_i^k}
			\int_{\mathbb T^2}a_i^k(x)\,dx           \\
			&=
			\frac{c_i\lambda_{q+1} r_{q+1}^{3/2}\eta_i^k}
			{4}\\
			&\leq 5 c_i\lambda_{q+1}r_{q+1}^{3/2}\delta_{q+1}^{\frac12}
			\leq \frac{\tau_{q+1}}{10\lambda_{q+1}}.
		\end{aligned}
	\end{equation}
	For the last inequality, we impose
	\begin{equation}\label{est-period}
		\lambda_{q+1}^2 r_{q+1}^{\frac32} \delta_{q+1}^{\frac12}\leq \frac{\tau_{q+1}}{200},
	\end{equation}
	which from \eqref{eq:parameters} can be guaranteed by a stronger condition
	\begin{equation*}
		\frac32 \mu -\beta-2-\kappa>0.
	\end{equation*}

	Let $M_i^k\in\mathbb N$ be the number of complete periods contained in
	$\bar{\mathcal T}_i^k$; by~\eqref{eq:SQG-period-physical},
	$M_i^k\geq\lambda_{q+1}$. Indeed, since
	\[
	|\bar{\mathcal T}_i^k|
	=
	\tau_{q+1}
	\left(\frac14-\frac{2}{\lambda_{q+1}}\right),
	\]
	and
	\[
	T_i^k\leq\frac{\tau_{q+1}}{10\lambda_{q+1}},
	\]
	the number of complete periods satisfies
	\[
	\begin{aligned}
	M_i^k=
	\left\lfloor
	\frac{|\bar{\mathcal T}_i^k|}{T_i^k}
	\right\rfloor\geq
	\left\lfloor
	10\lambda_{q+1}
	\left(\frac14-\frac{2}{\lambda_{q+1}}\right)
	\right\rfloor
	=
	\left\lfloor
	\frac52\lambda_{q+1}-20
	\right\rfloor
	\geq\lambda_{q+1}
	\end{aligned}
	\]
	for sufficiently large $\lambda_0$. It also satisfies
	\begin{enumerate}
		\item[(i)] $x_i^k\left(t_0+m T_i^k\right)=x_i^k\left(t_0\right)$ for every $m \in \mathbb{N}, 0 \leq m \leq M_i^k$,
		\item[(ii)] $t_0+M_i^k T_i^k \leq \tau_{q+1}\left(k+\frac{i}{4}-\lambda_{q+1}^{-1}\right)$ and $t_0+\left(M_i^k+1\right) T_i^k>\tau_{q+1}\left(k+\frac{i}{4}-\lambda_{q+1}^{-1}\right)$.
	\end{enumerate}

	\subsection{Principal perturbation blocks}
	\leavevmode\par\noindent
	We now define the principal SQG building blocks $V_i^k$ and $U_i^k$ adapted to the
	space-dependent scale \(r_i^k\), the time cut-off \(\eta_i^k\zeta_i^k(t)\), and the
	trajectory \(x_i^k(t)\) constructed in the previous subsections.
	More precisely, by
	Proposition~\ref{prop:sqg-nonconstant-speed}, $V_i^k$ and $U_i^k$ solve
	
	\begin{equation}\label{eq:Vik}
		\partial_tV_i^k
		+U_i^k\cdot\nabla V_i^k
		-(\nabla V_i^k)^T U_i^k
		+\nabla P_i^k
		=
		S_i^k
		\frac{d}{dt}
		\left[
		\eta_i^k\zeta_i^k(t)
		\bigl(r_i^k(x_i^k(t))\bigr)^{3/2}
		\right]
		+\operatorname{div} R_i^k,\qquad t\in \mathcal{T}_i^k
	\end{equation}
	We record the basic estimates for the building blocks, source term $S_i^k$ and the error $R_i^k$.
	For every $1<p<\infty$,
	\[
	\|V_i^k\|_{L_t^\infty L_x^p}
	\lesssim_p
	\eta_i^k
	\left\|
	(r_i^k)^{\frac2p-\frac12}
	\right\|_{L_x^\infty},
	\]
	and
	\[
	\|DV_i^k\|_{L_t^\infty L_x^p}
	\lesssim_p
	\eta_i^k
	\left\|
	(r_i^k)^{\frac2p-\frac32}
	\right\|_{L_x^\infty}.
	\]
	For $4/3<p<4$, set $s_p:=3/2-2/p>0$; in particular,
	$s_{\bar p}=3/2-2/\bar p$. For $p=\bar p$ and $p=4$, we have
	\begin{equation}\label{est-Vik}
		\begin{aligned}
		\|V_i^k\|_{L_t^\infty L_x^{\bar p}}
		&\lesssim \delta_{q+1}^{1/2}
		\lambda_{q+1}^{-(1-s_{\bar p})/\alpha},\\
		\|V_i^k\|_{L_t^\infty L_x^4}
		&\lesssim \delta_{q+1}^{1/2},\\
		\|DV_i^k\|_{L_t^\infty L_x^{\bar p}}
		&\lesssim \delta_{q+1}^{1/2-(2/3)s_{\bar p}}
		r_{q+1}^{-s_{\bar p}},\\
		\|DV_i^k\|_{L_t^\infty L_x^4}
		&\lesssim \delta_{q+1}^{-1/6}r_{q+1}^{-1}.
		\end{aligned}
	\end{equation}

	Since \(U_i^k=\Lambda_{\mathbb T^2}V_i^k\), and \(U_r\) scales like the SQG
	active velocity, we also have
	\[
	\|U_i^k\|_{L_t^\infty L_x^p}
	\lesssim_p
	\eta_i^k
	\left\|
	(r_i^k)^{\frac2p-\frac32}
	\right\|_{L_x^\infty},
	\qquad
	\|DU_i^k\|_{L_t^\infty L_x^p}
	\lesssim_p
	\eta_i^k
	\left\|
	(r_i^k)^{\frac2p-\frac52}
	\right\|_{L_x^\infty}.
	\]
	In particular, for $p=\bar p$ and $p=4$, we have
	\begin{equation}\label{est-Uik}
		\|U_i^k\|_{L_t^\infty L_x^4}
		\lesssim
		\delta_{q+1}^{-\frac16} r_{q+1}^{-1},
		\quad
		\|DU_i^k\|_{L_t^\infty L_x^{\bar{p}}}
		\lesssim
		\delta_{q+1}^{-\frac16-\frac23 s_{\bar{p}}} r_{q+1}^{-1-s_{\bar{p}}},
		\quad
		\|DU_i^k\|_{L_t^\infty L_x^{4}}
		\lesssim
		\delta_{q+1}^{-\frac{5}{6}} r_{q+1}^{-2}.
	\end{equation}
	
	For the source term
	\[
	S_i^k\frac{d}{dt}
	\left[
	\eta_i^k\zeta_i^k(t)
	\bigl(r_i^k(t)\bigr)^{3/2}
	\right],
	\]
	we have
	\begin{equation*}
		\|S_i^k(\cdot,t)\|_{L_x^p}
		\lesssim_p
		\bigl(r_i^k\bigr)^{\frac2p-2}
		\lesssim
		r_{q+1}^{\frac{2}{p}-2}\delta_{q+1}^{\frac{4}{3p}-\frac43},
	\end{equation*} 
	\begin{equation*}
		\supp S_i^k(\cdot,t)
		\subseteq B_{3(r_i^k(t))^\alpha}(x_i^k(t))
		\subseteq B_{\lambda_{q+1}^{-1}}(x_i^k(t)),
		\qquad
		\int_{\mathbb{T}^2} S_i^k(x,t)\,dx=\xi_i.
	\end{equation*}
	Next, we also have
	\begin{equation}\label{est-ddt}
		\begin{aligned}
			\left|\frac{d}{d t}\left(\eta_i^k \zeta_i^k(t) (r_i^k(t))^{3/2}\right)\right| & \leq\left|\eta_i^k (r_i^k(t))^{3/2} \frac{d \zeta_i^k}{d t}\right|+C\left|\left(\eta_i^k \zeta_i^k\right)^2 \frac{\nabla r_i^k}{r_i^k}\right| \\
			& \leq C \delta_{q+1}^{\frac{1}{2}} r_{q+1}^{3/2}\left\|a_i^k\right\|_{C_{x}^{0}} \frac{\lambda_{q+1}}{\tau_{q+1}}+C \delta_{q+1}\left\|\frac{\nabla a_i^k}{a_i^k}\right\|_{L_x^{\infty}} \\
			& \leq C\delta_{q+1}^{\frac{1}{2}} r_{q+1}^{3/2}\|a_i^k\|_{C_x^0}\frac{\lambda_{q+1}}{\tau_{q+1}}+ C\left\|a_i^k\right\|_{C_x^1}\\
			& \leq C \left\|a_i^k\right\|_{C_x^1} 
			\leq C \lambda_q^{6n+2\beta\sigma},
		\end{aligned}
	\end{equation} 
	where the last line is guaranteed by \eqref{est-period}.
	
	For the time derivatives, 
	we have
	\[
	\begin{aligned}
		\|\partial_tV_i^k\|_{L_t^\infty L_x^p}
		&\lesssim_p
		\eta_i^k\frac{\lambda_{q+1}}{\tau_{q+1}}
		\left\|
		(r_i^k)^{\frac2p-\frac12}
		\right\|_{L_x^\infty}
		\\
		&\quad
		+
		(\eta_i^k)^2
		\left[
		\left\|
		(r_i^k)^{\frac2p-3}
		\right\|_{L_x^\infty}
		+
		\|\nabla\log r_i^k\|_{L_x^\infty}
		\left\|
		(r_i^k)^{\frac2p-2}
		\right\|_{L_x^\infty}
		\right].
	\end{aligned}
	\]
	Similarly,
	\[
	\begin{aligned}
		\|\partial_tDV_i^k\|_{L_t^\infty L_x^p}
		&\lesssim_{p}
		\eta_i^k\frac{\lambda_{q+1}}{\tau_{q+1}}
		\left\|
		(r_i^k)^{\frac2{p}-\frac32}
		\right\|_{L_x^\infty}
		\\
		&\quad
		+
		(\eta_i^k)^2
		\left[
		\left\|
		(r_i^k)^{\frac2{p}-4}
		\right\|_{L_x^\infty}
		+
		\|\nabla\log r_i^k\|_{L_x^\infty}
		\left\|
		(r_i^k)^{\frac2{p}-3}
		\right\|_{L_x^\infty}
		\right].
	\end{aligned}
	\]
	Since $\|\nabla\log r_i^k\|_{L_x^\infty}$ can be controlled by $O(\lambda_{q+1}^{\frac{6n}{\sigma}+3\beta})$, much smaller than $O(\lambda_{q+1}^{\mu+\frac23\beta})$, which is the bound of $(r_i^k)^{-1}$,
	when $p=\bar p$ and $p=4$, we have
	\begin{equation}\label{est:ptpxV-barp}
		\|\partial_tDV_i^k\|_{L_t^\infty L_x^{\bar{p}}}
		\leq C_{\bar{p}}\delta_{q+1}^{\frac12-\frac23 s_{\bar{p}}}
		\lambda_{q+1}^{1+\kappa+\mu s_{\bar{p}}}
		+C_{\bar p}\delta_{q+1}^{-\frac23-\frac23 s_{\bar{p}}}
		\lambda_{q+1}^{\mu(s_{\bar{p}}+\frac52)},
	\end{equation}
	and
	\begin{equation}\label{est-ptVik}
		\|\partial_tV_i^k\|_{L_t^\infty L_x^4}+ (r_{q+1}\delta_{q+1}^\frac23)\|\partial_t DV_i^k\|_{L_t^\infty L_x^4}
		\leq
		C\delta_{q+1}^{\frac12}\lambda_{q+1}^{1+\kappa}
		+C \delta_{q+1}^{-\frac23} r_{q+1}^{-\frac52},
	\end{equation}
	where we assume
	\[
	\mu>\frac{6n}{\sigma}+3\beta.
	\]
	Finally, the error tensor satisfies
	\begin{equation*}
		\begin{aligned}
			\| R_i^k\|_{L_t^\infty L_x^1}
			&\leq C_p
			(\eta_i^k)^2
			\left[
			\left\|
			(r_i^k)^{\frac{2}{p}+1}
			\right\|_{L_x^\infty}
			+
			\|\nabla\log r_i^k\|_{L_x^\infty}
			\left\|
			(r_i^k)^{\frac{2}{p}-1}
			\right\|_{L_x^\infty}
			\right]\\
			&\leq C_p \delta_{q+1} \lambda_q^{6n+3\beta\sigma}(r_{q+1} \lambda_q^{\frac23(4n+\beta\sigma)})^{\frac{2}{p}-1}\\
			&\leq C_p\delta_{q+1}r_{q+1}^{\frac2p-1}
			\lambda_q^{6n+3\beta\sigma+
				\left(\frac2p-1\right)
				\left(\frac{8n}{3}
				+
				\frac{2\beta\sigma}{3}
				\right)}.
		\end{aligned}
	\end{equation*}

	\subsection{Auxiliary source profile}
	\label{subsec:auxiliary-sqg-building-block}
	\leavevmode\par\noindent
	The exact source $S_i^k$ has a radius depending on $a_i^k$, so its orbit
	average is not explicit. We therefore replace it by a fixed-width profile
	with the same spatial mean and an explicit uniform orbit average; this
	replacement also yields a small factor in the cancellation error.
	Introduce the auxiliary source terms
	\begin{equation}
		\label{eq:auxiliary-U-definition}
		\mathcal U_i^k(x,t)
		:=
		\frac{d}{dt}\left[
		\eta_i^k\zeta_i^k(t)
		\bigl(r_i^k(x_i^k(t))\bigr)^{3/2}
		\right]\,
		\widetilde\Omega_i^k(x-x_i^k(t))\xi_i,
		\qquad
		t\in \mathcal T_i^k,
	\end{equation}
	\begin{equation}\label{eq:auxiliary-U-global}
		\mathcal U_{q+1}:=
		\sum_{k\in\mathbb N}\sum_{i=1}^4\mathcal U_i^k.
	\end{equation}
	Here each $\mathcal U_i^k$ is extended by zero outside
	$\mathcal T_i^k$, and the scalar function
	$\widetilde\Omega_i^k(x-x_i^k(t))$ is required to satisfy the following properties:
	\begin{enumerate}
		\item[(i)] $\widetilde\Omega_i^k\in C^\infty(\mathbb T^2),
		\qquad
		\widetilde\Omega_i^k\geq 0,
		\qquad
		\operatorname{supp}\widetilde\Omega_i^k
		\subset B_{10\lambda_{q+1}^{-1}}(0),$
		\item[(ii)] Its spatial average equals $1$:
		\begin{equation*}
			\int_{\mathbb T^2}\widetilde\Omega_i^k(x)\,dx= 1,
		\end{equation*}
		
		\item[(iii)] For every $x\in \mathbb{T}^2$, one has
		\begin{equation}
			\label{eq:Omega-orbit-average}
			\frac1{c_i \lambda_{q+1}}\int_0^{c_i \lambda_{q+1}}\widetilde\Omega_i^k (x-(x_i^k(t_0)+s\xi_i))\,ds=1.
		\end{equation}
		
		\item[(iv)] For every \(1\le p\le \infty\),
		\begin{equation}
			\label{eq:Omega-Lp-bounds}
			\|\widetilde\Omega_i^k\|_{L^p}
			\le C_p\lambda_{q+1}^{2-\frac2p},
			\qquad
			\|D\widetilde\Omega_i^k\|_{L^p}
			\le C_p\lambda_{q+1}^{3-\frac2p},
			\qquad
			\|D^2\widetilde\Omega_i^k\|_{L^p}
			\le C_p\lambda_{q+1}^{4-\frac2p}.
		\end{equation}
	\end{enumerate}
	Such a function exists by normalization and periodization; see~\cite{BCK26}
	for a detailed proof. Property~(ii) ensures that the auxiliary
	source has the same spatial average as $S_i^k$ and thereby yields an
	additional factor of $\lambda_{q+1}^{-1}$ when estimating the
	anti-divergence of
	$S_i^k-\widetilde\Omega_i^k(x-x_i^k(t))\xi_i$.

	\begin{proposition}
		\label{prop:auxiliary-sqg-building-block}
		Let $\mathcal{U}_{q+1}$ be defined by
		\eqref{eq:auxiliary-U-definition}--\eqref{eq:auxiliary-U-global}. Then,
		for every \(1<p<\infty\),
		\begin{equation}
			\label{eq:auxiliary-U-estimate}
			\|\mathcal U_{q+1}\|_{L_t^\infty L_x^p}
			+
			\lambda_{q+1}^{-1}
			\|D\mathcal U_{q+1}\|_{L_t^\infty L_x^p}
			+
			\lambda_{q+1}^{-2}
			\|D^2\mathcal U_{q+1}\|_{L_t^\infty L_x^p}
			\le
			C_p\lambda_{q+1}^{2-\frac2p}
			\lambda_q^{6n+2\beta\sigma}.
		\end{equation}
		Moreover, for every $k\in\mathbb N$ there exist smooth symmetric
		tensors $G_1^k,\ldots,G_4^k$ such that, for $t\in\mathcal T^k$,
		\begin{equation}
			\label{eq:auxiliary-time-average-sqg}
			P_{\tau_{q+1}}\mathcal U_{q+1}(x,t)
			=
			\operatorname{div}
			\left(
			\sum_{i=1}^4 a_i^k(x)\xi_i\otimes\xi_i
			+
			\sum_{i=1}^4G_i^k(x)
			\right),
		\end{equation}
		\begin{equation}
			\label{eq:Gik-L1-estimate-sqg}
			\|G_i^k\|_{L^1_x}
			\le
			C\frac{\|a_i^k\|_{C^1_x}}{\lambda_{q+1}},
			\qquad i=1,2,3,4.
		\end{equation}
	\end{proposition}
	
	\begin{proof}
		First, from the definition 
		\eqref{eq:auxiliary-U-definition}, the estimates \eqref{est-ddt} and \eqref{eq:Omega-Lp-bounds}, we have
		\begin{equation*}
			\begin{aligned}
				&\|\mathcal U_{q+1}\|_{L_x^p}
				+
				\lambda_{q+1}^{-1}
				\|D\mathcal U_{q+1}\|_{L_x^p}
				+
				\lambda_{q+1}^{-2}
				\|D^2\mathcal U_{q+1}\|_{L_t^\infty L_x^p}\\
				\le &
				C\lambda_{q+1}^{2-\frac2p}
				\sup_{t\in \mathcal{T}_i^k}\left|\frac{d}{dt}\left[
				\eta_i^k\zeta_i^k(t)
				\bigl(r_i^k(x_i^k(t))\bigr)^{3/2}
				\right]\right|\\
				\le & 
				C \lambda_{q+1}^{2-\frac2p} \lambda_q^{6n+2\beta\sigma}.
			\end{aligned}
		\end{equation*}
		To prove \eqref{eq:auxiliary-time-average-sqg}, we calculate the integral over $\mathcal{T}_i^k$ by integration by parts:
		\begin{equation}\label{Uq-decomp-1}
			\begin{aligned}
				\int_{\mathcal{T}_i^k}\mathcal U_{q+1}(x,t)\,dt
				&=\int_{\mathcal{T}_i^k}\frac{d}{dt}\left[
				\eta_i^k\zeta_i^k(t)
				\bigl(r_i^k(t)\bigr)^{3/2}
				\right]
				\widetilde\Omega_i^k(x-x_i^k(t))\xi_i\,dt\\		
				&= -\int_{\mathcal{T}_i^k}\eta_i^k\zeta_i^k(t)
				\bigl(r_i^k(t)\bigr)^{3/2}
				\frac{d}{dt}
				\Big[
				\widetilde\Omega_i^k(x-x_i^k(t))
				\Big]\xi_i\,dt \\
				&=  
				\int_{\mathcal{T}_i^k}
				(\eta_i^k \zeta_i^k)^2
				\operatorname{div}\left(\widetilde\Omega_i^k(x-x_i^k(t))\xi_i\otimes\xi_i\right)\,dt \\
				&= 
				\operatorname{div}
				\left[
				\xi_i\otimes\xi_i
				\int_{\mathcal{T}_i^k}
				(\eta_i^k \zeta_i^k)^2
			\widetilde\Omega_i^{k}(x-x_i^k(t))\,dt
				\right].
			\end{aligned}
		\end{equation}

		Recall $t_0=\tau_{q+1}\left(
		k+\frac{i-1}{4}+\lambda_{q+1}^{-1}
		\right)$. 
	We set $\widetilde{\mathcal{T}}_i^k:=\left[t_0,t_0+M_i^kT_i^k\right]
	\subseteq\bar{\mathcal{T}}_i^k$. Since $\zeta_i^k(t)=1$ for
	$t\in\bar{\mathcal{T}}_i^k$, we write
		
		\begin{equation*}
			\begin{aligned}
				& \int_{\mathcal{T}_i^k}\left(\eta_i^k \zeta_i^k(t)\right)^2 \widetilde{\Omega}_i^k(x-x_i^k(t)) d t \\
			& \quad=(\eta_i^k)^2 \int_{\widetilde{\mathcal{T}}_i^k} \widetilde{\Omega}_i^k(x-x_i^k(t)) d t+(\eta_i^k)^2 \int_{\mathcal{T}_i^k \backslash \widetilde{\mathcal{T}}_i^k}( \zeta_i^k(t))^2 \widetilde{\Omega}_i^k(x-x_i^k(t)) d t=: I+II.
			\end{aligned}
		\end{equation*}

	The term $II$, multiplied by $\xi_i \otimes \xi_i$, will be part of
	the error $G_i^k$. Its $L^1$ estimate is
		\begin{equation}\label{est-II}
		\|II\|_{L^1} \leq\left(\eta_i^k\right)^2\left\|\widetilde{\Omega}_i^k\right\|_{L^1} \int_{\mathcal{T}_i^k \backslash \widetilde{\mathcal{T}}_i^k}\left(\zeta_i^k(t)\right)^2 d t \leq C \delta_{q+1} \tau_{q+1} \lambda_{q+1}^{-1}.
		\end{equation}

		Next, we consider the main term $I$. We do the change of variables $t \rightarrow t(s) \in \widetilde{\mathcal{T}}_i^k$ such that $x_i^k(t(s))=x_i^k\left(t_0\right)+s \xi_i$, where $t(0)=t_0$. 
		Then
		\[
		\frac{dt}{ds}
		=
		\frac{\big(r_i^k(x_i^k(t_0)+s\xi_i)\big)^{3/2}}
		{\eta_i^k\zeta_i^k(t(s))}
		=
		\frac{r_{q+1}^{3/2}a_i^k(x_i^k(t_0)+s\xi_i)}
		{\eta_i^k \zeta_i^k(t(s))}.
		\]
		Employing the change of variables, we compute the term $I$ as follows:
		
		\begin{equation*}
			\begin{aligned}
				I & =\eta_i^k r_{q+1}^{3/2} \int_0^{c_i M_i^k \lambda_{q+1}}a_i^k(x_i^k(t_0)+s\xi_i)\widetilde{\Omega}_i^k\left(x-\left(x_i^k\left(t_0\right)+s \xi_i\right)\right) d s \\
				& =\eta_i^k r_{q+1}^{3/2} a_i^k(x) \int_0^{c_i M_i^k \lambda_{q+1}} \widetilde{\Omega}_i^k\left(x-\left(x_i^k\left(t_0\right)+s \xi_i\right)\right) d s \\
				& +\eta_i^k r_{q+1}^{3/2} \int_0^{c_i M_i^k \lambda_{q+1}}\left(a_i^k(x_i^k(t_0)+s \xi_i)-a_i^k(x)\right) \widetilde{\Omega}_i^k\left(x-\left(x_i^k(t_0)+s \xi_i\right)\right) d s. \\
			& =: I^{\prime}+II^{\prime},
			\end{aligned}
		\end{equation*}

	where $I^{\prime}$ is the main term and $II^{\prime}$ will be part of
	the error $G_i^k$. We rewrite the main term

		\begin{equation}\label{Uq-decomp-4}
			\begin{aligned}
				I^{\prime}
				=&\eta_i^k r_{q+1}^{3/2} a_i^k(x) \int_0^{c_i M_i^k \lambda_{q+1}} \widetilde{\Omega}_i^k\left(x-\left(x_i^k\left(t_0\right)+s \xi_i\right)\right) d s \\
				=&\eta_i^k r_{q+1}^{3/2} c_i M_i^k \lambda_{q+1} a_i^k(x)\\
				=&\left(\tau_{q+1}+E\right) a_i^k(x),
			\end{aligned}
		\end{equation}
		where using \eqref{eq:SQG-period-physical} and \eqref{est-period}, the error $E$ could be estimated by
		\begin{equation}\label{est-E}
			|E| \leq C \frac{\tau_{q+1}}{\lambda_{q+1}}+ C c_i \lambda_{q+1} r_{q+1}^{3/2}\eta_i^k
			\leq C \frac{\tau_{q+1}}{\lambda_{q+1}}+C \delta_{q+1}^{1 / 2} r_{q+1}^{3/2} \lambda_{q+1} \leq C \frac{\tau_{q+1}}{\lambda_{q+1}}.
		\end{equation}
	We now estimate $II^{\prime}$:
		\begin{equation}\label{est-II'}
		\left\|II^{\prime}\right\|_{L^1} \leq C \eta_i^k r_{q+1}^{3/2} c_i M_i^k \lambda_{q+1} \frac{\left\|a_i^k\right\|_{C^1}}{\lambda_{q+1}} \leq C\left(\tau_{q+1}+E\right) \frac{\left\|a_i^k\right\|_{C^1}}{\lambda_{q+1}} \leq C \tau_{q+1} \frac{\left\|a_i^k\right\|_{C^1}}{\lambda_{q+1}} . 
		\end{equation}
		
	Combining \eqref{Uq-decomp-1}--\eqref{Uq-decomp-4} gives, for each
	$i$,
		\begin{equation*}
			\begin{aligned}
				\frac{1}{\tau_{q+1}} \int_{\mathcal{T}_i^k} \mathcal U_{q+1}(x, t) d t=\operatorname{div}\left(a_i^k(x) \xi_i \otimes \xi_i+G_i^k(x)\right), \\
			G_i^k=\frac{1}{\tau_{q+1}}\left(II+II^{\prime}+E a_i^k(x)\right) \xi_i \otimes \xi_i.
			\end{aligned}
		\end{equation*}
		
		Combining the estimates \eqref{est-II}, \eqref{est-E} and \eqref{est-II'}, we obtain
		\begin{equation*}
			\left\|G_i^k\right\|_{L^1} 
			\leq C \frac{\delta_{q+1}}{\lambda_{q+1}}
			+C \frac{\left\|a_i^k\right\|_{L^1}}{\lambda_{q+1}}
			+C \frac{\left\|a_i^k\right\|_{C^1}}{\lambda_{q+1}} 
			\leq C \frac{\left\|a_i^k\right\|_{C^1}}{\lambda_{q+1}}.\qedhere
		\end{equation*}
	\end{proof}

	\subsection{Time corrector}
	\label{subsec:sqg-time-corrector}
	
	The auxiliary source reconstructs the old stress only after averaging over
	a coarse time interval. The time corrector absorbs its zero-average temporal
	oscillation, leaving the averaged source in the new equation.
	
	Let \(\mathbb P_{\mathrm L}\) denote the Leray projection on \(\mathbb T^2\). 
	We define the time corrector \(Q_{q+1}\) by
	\begin{equation*}
		-Q_{q+1}(x,t)
		:=
		\mathbb P_{\mathrm L}
		\left(
		\int_0^t \Bigl(\mathcal{U}_{q+1}(x,s)
		- P_{\tau_{q+1}}\mathcal{U}_{q+1}(x,s)\Bigr)\,ds
		\right).
	\end{equation*}
	Equivalently, since the integral over every previous complete interval
	such as $[k\tau_{q+1},(k+1)\tau_{q+1}]$ vanishes by the definition of
	\(P_{\tau_{q+1}}\), only the current interval contributes:
	\begin{equation}
		\label{eq:time-corrector-local-form-sqg}
		-Q_{q+1}(x,t)
		=
		\mathbb P_{\mathrm L}
		\int_{k\tau_{q+1}}^t
		\left(
		\mathcal U_{q+1}(x,s)
		-
		\fint_{\mathcal T^k}\mathcal U_{q+1}(x,s')\,ds'
		\right)\,ds .
	\end{equation}
	
	Let $\pi_Q$ be the mean-zero scalar determined by
	\begin{equation*}
		\nabla\pi_Q
		:=
		(\operatorname{Id}-\mathbb P_{\mathrm L})
		\bigl(\mathcal U_{q+1}-P_{\tau_{q+1}}\mathcal U_{q+1}\bigr).
	\end{equation*}
	Then the time corrector satisfies
	\begin{equation}
		\label{eq:time-corrector-identity-sqg}
		\partial_tQ_{q+1}
		+
		\mathcal U_{q+1}
		-
		P_{\tau_{q+1}}\mathcal U_{q+1}
		=\nabla\pi_Q.
	\end{equation}
	For every \(1<p<\infty\), by definition
	\eqref{eq:time-corrector-local-form-sqg}, the $L^p$-boundedness of Leray projection $\mathbb{P}_{\mathrm L}$ and \eqref{eq:auxiliary-U-estimate}, we have
	\begin{equation}
		\label{eq:time-corrector-Lp-estimate-sqg}
		\begin{aligned}
			&\|Q_{q+1}\|_{L_t^\infty L_x^p}
			+
			\lambda_{q+1}^{-1}
			\|DQ_{q+1}\|_{L_t^\infty L_x^p}
			+
			\lambda_{q+1}^{-2}
			\|D^2 Q_{q+1}\|_{L_t^\infty L_x^p}\\
			\le
			&C\tau_{q+1}
			\left(
			\|\mathcal U_{q+1}\|_{L_t^\infty L_x^p}
			+
			\lambda_{q+1}^{-1}
			\|D\mathcal U_{q+1}\|_{L_t^\infty L_x^p}
			+
			\lambda_{q+1}^{-2}
			\|D^2\mathcal U_{q+1}\|_{L_t^\infty L_x^p}
			\right)\\
			\leq&C \tau_{q+1} \lambda_{q+1}^{2-\frac2p}
			\lambda_q^{6n+2\beta\sigma}.
		\end{aligned}
	\end{equation}
	Moreover, by \eqref{eq:time-corrector-identity-sqg},
	\begin{equation*}
		\|\partial_tQ_{q+1}\|_{L_t^\infty L_x^p}
		\le
		C\|\mathcal U_{q+1}\|_{L_t^\infty L_x^p}
		\le
		C  \lambda_{q+1}^{2-\frac2p}
		\lambda_q^{6n+2\beta\sigma},
	\end{equation*}
	and
	\begin{equation}
		\label{eq:time-corrector-dtD-estimate-sqg}
		\|\partial_tDQ_{q+1}\|_{L_t^\infty L_x^p}
		\le
		C\|D\mathcal U_{q+1}\|_{L_t^\infty L_x^p}
		\le
		C  \lambda_{q+1}^{3-\frac2p}
		\lambda_q^{6n+2\beta\sigma}.
	\end{equation}
	
	Since \(Q_{q+1}\) is a correction to the potential velocity \(v\), its corresponding
	SQG transport velocity is defined by
	\begin{equation*}
		Q_{q+1}^u
		:=
		\Lambda \mathbb{P}_{\neq0}Q_{q+1}.
	\end{equation*}
	Consequently, for every \(1<p<\infty\),
	\begin{equation}\label{eq:Q-u-estimate}
		\begin{aligned}
			&\|Q_{q+1}^u\|_{L_t^\infty L_x^p}
			+\lambda_{q+1}^{-1}\|D Q_{q+1}^u\|_{L_t^\infty L_x^p}\\
			\leq 
			&C
			\|DQ_{q+1}\|_{L_t^\infty L_x^p}
			+
			C\lambda_{q+1}^{-1}\|D^2 Q_{q+1}\|_{L_t^\infty L_x^p}\\
			\leq
			&C \tau_{q+1} \lambda_{q+1}^{3-\frac2p}
			\lambda_q^{6n+2\beta\sigma}.
		\end{aligned}
	\end{equation}
	Finally, from \eqref{eq:time-corrector-local-form-sqg}, one has
	\begin{equation*}
		Q_{q+1}(x,k\tau_{q+1})=Q^u_{q+1}(x,k\tau_{q+1})=0,
		\qquad
		k\in\mathbb N .
	\end{equation*}

	\subsection{New approximate solution and Reynolds stress}
	\label{subsec:perturbation-new-reynolds-stress-sqg}
	\leavevmode\par\noindent
	
	We now combine the mollified background, the principal perturbations, and
	the time corrector. Expanding the momentum nonlinearity then identifies the
	remaining errors and determines the new Reynolds stress.
	
	Define the principal potential-velocity perturbation by
	\begin{equation*}
		w_{q+1}(x,t)
		:=
		\sum_{k\in\mathbb N}\sum_{i=1}^4 V_i^k(x,t),
	\end{equation*}
	and the corresponding transport-velocity perturbation by
	\begin{equation*}
		W_{q+1}(x,t)
		:=
		\sum_{k\in\mathbb N}\sum_{i=1}^4 U_i^k(x,t).
	\end{equation*}
	By construction,
	\begin{equation*}
		W_{q+1}
		=
		\Lambda \mathbb{P}_{\neq0}w_{q+1}.
	\end{equation*}
	
	We define the new potential velocity by
	\begin{equation}
		\label{eq:vq1-definition-sqg}
		v_{q+1}
		:=
		v_\ell+w_{q+1}+Q_{q+1},
	\end{equation}
	and the new transport velocity by
	\begin{equation}
		\label{eq:uq1-definition-sqg}
		u_{q+1}
		:=
		\Lambda \mathbb{P}_{\neq0}v_{q+1}
		=
		u_\ell+W_{q+1}+Q_{q+1}^u .
	\end{equation}
	The new scalar is
	\begin{equation*}
		\theta_{q+1}
		:=
		-\operatorname{curl}v_{q+1}.
	\end{equation*}
	Extend every $V_i^k$ and $U_i^k$ by zero outside
	$\mathcal T_i^k$. Because \(V_i^k\) and \(Q_{q+1}\) vanish at the endpoints
	\(t=k\tau_{q+1}\),
	\begin{equation*}
		v_{q+1}(x,k\tau_{q+1})
		=
		v_\ell(x,k\tau_{q+1}),
		\qquad
		k\in\mathbb N .
	\end{equation*}
	The transport velocity satisfies the analogous identity
	\begin{equation*}
		u_{q+1}(x,k\tau_{q+1})
		=
		u_\ell(x,k\tau_{q+1}),
		\qquad
		k\in\mathbb N .
	\end{equation*}
	Since $\tau_{q+1}^{-1}$ is an integer, every integer time is a grid
	point. Thus, by \eqref{eq:u-mollification-Du}, for every $m\in\mathbb N$ and $q\in \mathbb N$,
	we have
	\begin{equation*}
		\begin{aligned}
		&\|v_{q+1}(\cdot,m)-v_{q}(\cdot,m)\|_{L^4}
		+\|u_{q+1}(\cdot,m)-u_{q}(\cdot,m)\|_{L^4}\\
		\leq& 
		\|v_\ell-v_q\|_{L_t^\infty L_x^4}
		+
		\|u_\ell-u_q\|_{L_t^\infty L_x^4}\\
		 \leq& 
		 \lambda_0^{-\beta}\lambda_1^{2\beta} \lambda_q^{-n- \beta\sigma}\leq \lambda_q^{-1}.
		\end{aligned}
	\end{equation*}
    which in particular implies \eqref{eq:est-endpoint} by choosing $m=0$ and $m=1$.
	
	Combining \eqref{eq:vl}, \eqref{eq:error-decomposition}, \eqref{eq:Vik}, \eqref{eq:auxiliary-time-average-sqg} and \eqref{eq:time-corrector-identity-sqg}, we have
	\begin{equation*}
		\begin{aligned}
			&\partial_t v_{q+1}+ \sum_{i,k} N(U_i^k,V_i^k) + N(u_\ell,v_\ell)
			+\nabla \Big(p_\ell+\sum_{i,k} P_i^k+P^d-\pi_Q \Big)\\
			=&
			\div \Big(\sum_{i,k}(a_i^k(x)-a_i(x,t))\xi_i\otimes \xi_i\Big)
			+\sum_{i,k}\div(R_i^k+G_i^k)\\
			&\quad+\sum_{i,k}(S_i^k-\widetilde{\Omega}_i^k \xi_i) \frac{d}{dt}\left[
			\eta_i^k\zeta_i^k(t)
			\bigl(r_i^k(x_i^k(t))\bigr)^{3/2}
			\right].
		\end{aligned}
	\end{equation*}
	The nonlinear terms on the left-hand side can be decomposed as
	\begin{equation*}
		\sum_{i,k} N(U_i^k,V_i^k) + N(u_\ell,v_\ell)=N(u_{q+1},v_{q+1})-\mathcal E_{q+1}^{(l)} -\mathcal E_{q+1}^{(c)},
	\end{equation*}
	where
	\begin{equation*}
		\mathcal E_{q+1}^{(l)}
		:=
		N(u_\ell,w_{q+1})
		+
		N(W_{q+1},v_\ell),
	\end{equation*}
		is the linear interaction error, and
	\begin{equation*}
		\begin{aligned}
			\mathcal E_{q+1}^{(c)}
			:=&
			N(u_\ell+W_{q+1},Q_{q+1})
			+
			N(Q_{q+1}^u,v_\ell+w_{q+1})
			\\
			&+
			N(Q_{q+1}^u,Q_{q+1}).
		\end{aligned}
	\end{equation*}
		is the error involving the time corrector $Q_{q+1}$.
	
	The above error terms have zero spatial mean. Indeed, if
	$u_a=\Lambda \mathbb{P}_{\neq0}v_a$, $u_b=\Lambda \mathbb{P}_{\neq0}v_b$ and $\div u_a=\div u_b=0$, then
	\begin{equation*}
		\begin{aligned}
			N(u_a,v_b)+N(u_b,v_a) 
			&=
			(\nabla^{\perp}\cdot v_b)u_a^{\perp}
			+
			(\nabla^{\perp}\cdot v_a)u_b^{\perp}\\
			&=
			\theta_b \mathcal{R} \theta_a
			+\theta_a \mathcal{R} \theta_b,
		\end{aligned}
	\end{equation*}
	where $\theta_a=-\nabla^{\perp}\cdot v_a$ and
	$\theta_b=-\nabla^{\perp}\cdot v_b$. By skew-adjointness of the Riesz
	transforms, we have
	\begin{equation}
		\label{eq:mean-zero-pairing-identity}
		\int_{\mathbb T^2}
		\big(
		N(u_a,v_b)+N(u_b,v_a)
		\big)\,dx=0.
	\end{equation}
	Taking $a=b$ in~\eqref{eq:mean-zero-pairing-identity} also gives
	$\int_{\mathbb T^2}N(Q_{q+1}^u,Q_{q+1})\,dx=0$.
	Thus we define their corresponding Reynolds stress by 
	\begin{equation*}
		\mathring R_{q+1}^{(l)}
		:=
		\mathcal R_0\mathcal E_{q+1}^{(l)},
	\end{equation*}
	and
	\begin{equation*}
		\mathring R_{q+1}^{(c)}
		:=
		\mathcal R_0\mathcal E_{q+1}^{(c)}.
	\end{equation*}
	Meanwhile, since $S_i^k$ and $\widetilde{\Omega}_i^k$ have the same
	spatial average, define the source-replacement error by
	\begin{equation*}
		\mathring R_{q+1}^{(s)}
		:=
		\mathcal R_0\left(\sum_{i,k}(S_i^k-\widetilde{\Omega}_i^k\xi_i) \frac{d}{dt}
		\left[
		\eta_i^k\zeta_i^k(t)
		\bigl(r_i^k(x_i^k(t))\bigr)^{3/2}
		\right]\right).
	\end{equation*}
	
	We now define the $(q+1)$-st pressure:
	\begin{equation}
		\label{eq:pq1-definition-sqg}
		p_{q+1}
		:=
		p_\ell
		+
		\sum_{i,k}P_i^k
		+
		P^d
		-\pi_Q,
	\end{equation}
	and the $(q+1)$-st Reynolds stress:
	\begin{equation}
		\label{eq:Rq1-definition-sqg}
		R_{q+1}
		:=
		\mathring R_{q+1}^{(l)}
		+
		\mathring R_{q+1}^{(c)}
		+
		\mathring R_{q+1}^{(t)}
		+
		\mathring R_{q+1}^{(s)}
		+
		\sum_{i,k}(R_i^k+G_i^k).
	\end{equation}
	Here \(R_i^k\) is extended by zero outside \(\mathcal T_i^k\), as above,
	whereas \(G_i^k(x)\) is understood as the piecewise-in-time tensor
	\(G_i^k(x)\mathbf 1_{\mathcal T^k}(t)\). Thus every sum in
	\eqref{eq:Rq1-definition-sqg} is finite at each time.
	The time-freezing error is defined, for $t\in\mathcal T^k$, by
	\begin{equation*}
		\mathring R_{q+1}^{(t)}(x,t)
		:=
		\sum_{i=1}^4
		\big(a_i^k(x)-a_i(x,t)\big)
		\xi_i\otimes\xi_i .
	\end{equation*}
	Consequently, the quadruple
	$(v_{q+1},u_{q+1},p_{q+1},R_{q+1})$ solves
	\begin{equation*}
		\partial_t v_{q+1}
		+
		u_{q+1}\cdot\nabla v_{q+1}
		-
		(\nabla v_{q+1})^T u_{q+1}
		+
		\nabla p_{q+1}
		=
		\operatorname{div} R_{q+1},
	\end{equation*}
	with
	\begin{equation*}
		\operatorname{div}v_{q+1}=0,
		\qquad
		u_{q+1}=\Lambda \mathbb P_{\neq 0} v_{q+1}.
	\end{equation*}

     \begin{remark}[Locality in time]\label{rem:SQG-time-locality}
	Assume that
	\[
	(v_q,u_q,p_q,R_q)
	\qquad\text{and}\qquad
	(v_q',u_q',p_q',R_q')
	\]
	are two solutions of the $q$-th SQG--Reynolds system satisfying the
	hypotheses of Proposition~\ref{prop:iteration}. Suppose that, for some
	$t_*\geq 1/8$,
	\[
	(v_q,u_q,p_q,R_q)
	=
	(v_q',u_q',p_q',R_q')
	\qquad\text{on }\mathbb T^2\times[0,t_*].
	\]
	Then, by the $(q+1)$-th ansatz \eqref{eq:vq1-definition-sqg}, \eqref{eq:uq1-definition-sqg}, \eqref{eq:pq1-definition-sqg} and \eqref{eq:Rq1-definition-sqg}, we can construct the corresponding $(q+1)$-st stage solutions such that
	\[
	(v_{q+1},u_{q+1},p_{q+1},R_{q+1})
	=
	(v_{q+1}',u_{q+1}',p_{q+1}',R_{q+1}')
	\]
	on
	\[
	\mathbb T^2\times
	[0,t_*-\ell-\tau_{q+1}],
	\]
    where $\tau_{q+1}=\lambda_{q+1}^{-\kappa}\leq \lambda_q^{-1}$ and 
    $\ell= \lambda_0^{-\beta} \lambda_q^{-2n}\delta_{q+1}\leq \lambda_q^{-1}$ is defined in \eqref{def:ell}.
    \end{remark}

	\section{Perturbation and Reynolds-stress estimates}
	This section derives the component estimates used in the target iteration
	statement of Proposition~\ref{prop:iteration}. We first treat the velocity
	bounds, then the six components of the new stress, and finally the parameter
	inequalities.
	
	\subsection{Velocity estimates}
	\leavevmode\par\noindent
	We begin with the $L^\infty_tL^4_x$ norm of $v_{q+1}-v_q$.
	From \eqref{eq:u-mollification-Du}, \eqref{est-Vik}, and
	\eqref{eq:time-corrector-Lp-estimate-sqg},
	\[
	\begin{aligned}
		\|v_{q+1}-v_q\|_{L^\infty_tL^4_x}
		&\le
		\|v_\ell-v_q\|_{L^\infty_tL^4_x}
		+\sup_{i,k}\|V_i^k\|_{L^\infty_tL^4_x}
		+\|Q_{q+1}\|_{L^\infty_tL^4_x}
		\\
		&\le
		\delta_{q+1}^{\frac12}
		+C\delta_{q+1}^{\frac12}
		+C \tau_{q+1} \lambda_{q+1}^{\frac32}
		\lambda_q^{6n+2\beta\sigma}\\
		&\le 
		C\delta_{q+1}^{\frac12}
		+ C\lambda_1^{-\beta}\lambda_{q+1}^{-\kappa+\frac32+\frac{6n}{\sigma}+2\beta+\frac12\beta}\delta_{q+1}^{\frac12}\\
		&\le C\delta_{q+1}^{\frac12},
	\end{aligned}
	\]
	for some universal constant $C>0$, provided 
	\begin{equation*}
	\frac32-\kappa+\frac{6n}{\sigma}+\frac52 \beta\le 0.
	\end{equation*}
	
	Next, 
	\[
	\|v_{q+1}\|_{L^\infty_tL^4_x}
	\le \|v_{q+1}-v_q\|_{L^\infty_tL^4_x}+ \|v_{q}\|_{L^\infty_tL^4_x}
	\le
	2\delta_0^{1/2}-\delta_{q}^{1/2}+M\delta_{q+1}^{\frac12} \le  2\delta_0^{1/2}-\delta_{q+1}^{1/2},
	\]
	provided that $\lambda_0$ is sufficiently large in terms of $M$.
	
	We now estimate the \(L^\infty_tL^{\bar p}\)-norm of \(Dv_{q+1}\). 
	From \eqref{est-Vik} and \eqref{eq:time-corrector-Lp-estimate-sqg}, we obtain
	\begin{align*}
		\|Dv_{q+1}\|_{L^\infty_tL^{\bar p}_x}
		&\le
		\|Dv_\ell\|_{L^\infty_tL^{\bar p}_x}
		+\sup_{i,k} \|D V_i^k\|_{L^\infty_tL^{\bar p}_x}
		+\|DQ_{q+1}\|_{L^\infty_tL^{\bar p}_x}
		\\
		&\le
		\|Dv_q\|_{L^\infty_tL^{\bar p}_x}
		+C\delta_{q+1}^{\frac12-\frac23 s_{\bar p}} r_{q+1}^{-s_{\bar p}}
		+C\tau_{q+1}\lambda_{q+1}^{3-\frac2{\bar p}}\lambda_q^{6n+2\beta\sigma}\\
		&\le 
		\|Dv_q\|_{L^\infty_tL^{\bar p}_x}
		+C\delta_{q+1}^{1/2}
		\bigl(\lambda_{q+1}^{-\mu}\lambda_{q+1}^{-\frac{2}{3}\beta}\lambda_1^{\frac{4}{3}\beta}\bigr)^{-s_{\bar p}}+ 
		C\lambda_{q+1}^{3-\kappa-\frac{2}{\bar{p}}+\frac{6n}{\sigma}+2\beta}\\
		&\le 
		\|Dv_q\|_{L^\infty_tL^{\bar p}_x}
		+
		C\lambda_1^\beta\delta_{q+1}^{1/100}\left(\lambda_{q+1}^{
			\left(\mu+\frac23\beta\right)s_{\bar p}
			-\frac{49}{100}\beta}
		+
		\lambda_{q+1}^{
			\frac{6n}{\sigma}
			+\frac{201}{100}\beta
			-\kappa+3-\frac2{\bar p}}\right)
		\\
		&\le 
		\|Dv_q\|_{L^\infty_tL^{\bar p}_x}+ \lambda_0\delta_{q+1}^{\frac{1}{100}},
	\end{align*}
	
	provided
	\begin{equation*}
		\begin{aligned}
			\beta\sigma<1,\qquad
			\mu s_{\bar{p}}+\frac23 \beta s_{\bar{p}}-\frac{49}{100}\beta < 0,\\
			3-\kappa-\frac{2}{\bar{p}}+\frac{6n}{\sigma}+(2+\frac{1}{100})\beta < 0.
		\end{aligned}
	\end{equation*}
For Hölder continuity in time, by interpolation inequality we have
\begin{align*}
    \|Dv_{q+1}\|_{C_t^\gamma L_x^{\bar p}}
		&\le
		\|Dv_q\|_{C_t^\gamma L_x^{\bar p}}
		+
		\sup_{i,k}\|DV_i^k\|_{C_t^\gamma L_x^{\bar p}}
		+
		\|DQ_{q+1}\|_{C_t^\gamma L_x^{\bar p}}
		\\
		&\le
		\|Dv_q\|_{C_t^\gamma L_x^{\bar p}}
		+
		\sup_{i,k}
		\|DV_i^k\|_{L_t^\infty L_x^{\bar p}}^{1-\gamma}
		\|\partial_tDV_i^k\|_{L_t^\infty L_x^{\bar p}}^\gamma
		\\
		&\qquad
		+
		\|DQ_{q+1}\|_{L_t^\infty L_x^{\bar p}}^{1-\gamma}
		\|\partial_tDQ_{q+1}\|_{L_t^\infty L_x^{\bar p}}^\gamma.
\end{align*}
Then substituting \eqref{est:ptpxV-barp} and
	\eqref{eq:time-corrector-dtD-estimate-sqg} gives
	\begin{align*}
		\|Dv_{q+1}\|_{C_t^\gamma L_x^{\bar p}}
		&\le
		\|Dv_q\|_{C_t^\gamma L_x^{\bar p}}
		+
		C\left(
		\delta_{q+1}^{\frac12-\frac23s_{\bar p}}
		\lambda_{q+1}^{\mu s_{\bar p}}
		\right)^{1-\gamma}
		\left(
		\delta_{q+1}^{\frac12-\frac23s_{\bar p}}
		\lambda_{q+1}^{1+\kappa+\mu s_{\bar p}}
		+
		C\delta_{q+1}^{-\frac23-\frac23s_{\bar p}}
		\lambda_{q+1}^{\mu\left(s_{\bar p}+\frac52\right)}
		\right)^\gamma
		\\
		&\qquad
		+
		C\left(
		\tau_{q+1}
		\lambda_{q+1}^{3-\frac2{\bar p}}
		\lambda_q^{6n+2\beta\sigma}
		\right)^{1-\gamma}
		\left(
		\lambda_{q+1}^{3-\frac2{\bar p}}
		\lambda_q^{6n+2\beta\sigma}
		\right)^\gamma
		\\
		&\le
		\|Dv_q\|_{C_t^\gamma L_x^{\bar p}}
		+
		C\delta_{q+1}^{\frac12-\frac23s_{\bar p}}
		\lambda_{q+1}^{\mu s_{\bar p}+\gamma(1+\kappa)}
		+
		C\delta_{q+1}^{\frac12-\frac23s_{\bar p}-\frac76\gamma}
		\lambda_{q+1}^{\mu s_{\bar p}+\frac52\mu\gamma}
		\\
		&\qquad
		+
		C\tau_{q+1}^{1-\gamma}
		\lambda_{q+1}^{3-\frac2{\bar p}}
		\lambda_q^{6n+2\beta\sigma}
		\\
		&=
		\|Dv_q\|_{C_t^\gamma L_x^{\bar p}}
		+
		C\delta_{q+1}^{\frac12-\frac23s_{\bar p}}
		\lambda_{q+1}^{\mu s_{\bar p}+\gamma(1+\kappa)}
		+
		C\delta_{q+1}^{\frac12-\frac23s_{\bar p}-\frac76\gamma}
		\lambda_{q+1}^{\mu s_{\bar p}+\frac52\mu\gamma}
		\\
		&\qquad
		+
		C\lambda_{q+1}^{
			\frac32+s_{\bar p}
			-\kappa(1-\gamma)
			+\frac{6n}{\sigma}
			+2\beta}\\
		={}&
		\|Dv_q\|_{C_t^\gamma L_x^{\bar p}}
		+
		\delta_{q+1}^{\frac1{100}}
		\Bigg[
		C\lambda_1^{
			\frac{49}{50}\beta-\frac43\beta s_{\bar p}
		}
		\lambda_{q+1}^{
			\mu s_{\bar p}
			+\frac23\beta s_{\bar p}
			-\frac{49}{100}\beta
			+\gamma(1+\kappa)
		}
		\\
		&\qquad\qquad\qquad\qquad
		+
		C\lambda_1^{
			\frac{49}{50}\beta
			-\frac43\beta s_{\bar p}
			-\frac73\beta\gamma
		}
		\lambda_{q+1}^{
			\mu s_{\bar p}
			+\frac23\beta s_{\bar p}
			-\frac{49}{100}\beta
			+\gamma\left(\frac52\mu+\frac76\beta\right)
		}
		\\
		&\qquad\qquad\qquad\qquad
		+
		C\lambda_1^{-\frac{\beta}{50}}
		\lambda_{q+1}^{
			\frac32+s_{\bar p}
			-\kappa(1-\gamma)
			+\frac{6n}{\sigma}
			+\left(2+\frac1{100}\right)\beta
		}
		\Bigg].
	\end{align*}
	For the fixed value $\gamma=10^{-5}$, direct substitution of the
	parameters in~\eqref{eq:admissible-parameters} verifies that the bracketed terms are
	bounded by $\lambda_1^\beta$.

 \noindent For the increment $\|D(v_{q+1}-v_q)\|_{C_t^\gamma L_x^{\bar p}}$,
    by convolution estimates and \eqref{eq:ind-A}, we have
         \begin{align*}
            \|D(v_{q+1}-v_q)\|_{C_t^\gamma L_x^{\bar p}} 
             \leq &
             \|D(v_\ell-v_q)\|_{C_t^\gamma L_x^{\bar p}} 
             +
             \sup_{i,k}\|DV_i^k\|_{C_t^\gamma L_x^{\bar p}}
		      +
		\|DQ_{q+1}\|_{C_t^\gamma L_x^{\bar p}}
        \\
        \leq&
        \|D_x(v_\ell-v_q)\|^{1-\gamma}_{L^\infty_t L_x^{\bar p}} \|D_t D_x(v_\ell-v_q)\|^{\gamma}_{L^\infty_t L_x^{\bar p}} 
        + \lambda_1^{\beta} \delta_{q+1}^{\frac{1}{100}}
        \\
         \leq&
         2\bigl(\ell\|D_{t,x}D_x v_q\|_{L^\infty_t L_x^{\bar p}} \bigr)^{1-\gamma}
         \|D_t D_x v_q\|^{\gamma}_{L^\infty_t L_x^{\bar p}} 
         + \lambda_1^{\beta} \delta_{q+1}^{\frac{1}{100}}
         \\
         \leq &
         C  \ell^{1-\gamma} \lambda_q^n 
         +\lambda_1^{\beta} \delta_{q+1}^{\frac{1}{100}}\\
         \leq &
         C\lambda_0^{-\beta(1-\gamma)-2n(1-2\gamma)}
         \delta_{q+1}^{
         	1-\gamma+\frac{n(1-2\gamma)}{\sigma\beta}
         }
         +\lambda_1^\beta\delta_{q+1}^{\frac1{100}}\\
         \leq &
         2\lambda_1^\beta\delta_{q+1}^{\frac1{100}}.
         \end{align*}

	Next, we estimate the $L^\infty_tL^{4}_x$-norm of
	$D_{x,t}u_{q+1}$. From \eqref{est-Uik} and \eqref{eq:Q-u-estimate}, we have
	\[
	\begin{aligned}
		\|D_x u_{q+1}\|_{L^\infty_tL^{4}_x}
		&\le
		\|D_x u_\ell\|_{L^\infty_tL^{4}_x}
		+\sup_{i,k} \|D_x U_i^k\|_{L^\infty_t L^{4}_x}
		+\|D_x Q^u_{q+1}\|_{L^\infty_t L^{4}_x}
		\\
		&\le
		\lambda_{q+1}^{\frac{n}{\sigma}}
		+C\delta_{q+1}^{-\frac56}r_{q+1}^{-2}
		+C \tau_{q+1} \lambda_{q+1}^{\frac72}
		\lambda_q^{6n+2\beta\sigma}\\
		&\le \frac{1}{20}\lambda_{q+1}^n
		+C\lambda_1^{-\frac{5\beta}{3}}\lambda_{q+1}^{2\mu+\frac{5}{6}\beta}
		+C\lambda_{q+1}^{\frac72-\kappa+2\beta+\frac{6n}{\sigma}}\\
		&\le \frac14\lambda_{q+1}^n,
	\end{aligned}
	\]
	provided
	\begin{equation*}
	2\mu+\frac{5}{6}\beta<n,\quad  \frac72-\kappa+2\beta+\frac{6n}{\sigma}<n.
	\end{equation*}
	From \eqref{est-ptVik} and \eqref{eq:time-corrector-dtD-estimate-sqg}, we have
	\[
	\begin{aligned}
		\|D_t u_{q+1}\|_{L^\infty_tL^{4}_x}
		&\le
		\|D_t u_\ell\|_{L^\infty_tL^{4}_x}
		+\sup_{i,k} \|D_t U_i^k\|_{L^\infty_tL^{4}_x}
		+\|D_tQ^u_{q+1}\|_{L^\infty_tL^{4}_x}
		\\
		&\leq
		\frac{1}{20}\lambda_{q+1}^n+C \left(\delta_{q+1}^{-\frac16}\lambda_{q+1}^{1+\kappa+\mu}
		+ \delta_{q+1}^{-\frac43}\lambda_{q+1}^{\frac72\mu}
		\right)
		+ 
		C\lambda_{q+1}^{\frac52+\frac{6n}{\sigma}+2\beta}
		\\
		&\le \frac14\lambda_{q+1}^n,
	\end{aligned}
	\]
	provided
	\begin{equation*}
	\max\{\kappa+1+\mu+\frac16\beta,\quad\frac43\beta+\frac72\mu,\quad \frac52+\frac{6n}{\sigma}+2\beta \} <n.
	\end{equation*}
	With the above assumptions on the parameters, we can similarly estimate
    \[
    \begin{aligned}
        \|D_x v_{q+1}\|_{L^\infty_tL^{4}_x}
        &\leq
        \|D_x v_\ell\|_{L^\infty_tL^{4}_x}
		+\sup_{i,k} \|D_x V_i^k\|_{L^\infty_t L^{4}_x}
		+\|D_x Q_{q+1}\|_{L^\infty_t L^{4}_x}
		\\
        &\leq
        \lambda_{q+1}^{n / \sigma}+C \lambda_1^{-\beta / 3} \lambda_{q+1}^{\mu+\beta / 6}+C \lambda_{q+1}^{\frac{5}{2}-\kappa+\frac{6 n}{\sigma}+2 \beta}
        \\
        &\leq \frac14\lambda_{q+1}^n,
    \end{aligned}
    \]
and   
     \[
     \begin{aligned}
        \|D_t v_{q+1}\|_{L^\infty_tL^{4}_x}
		&\le
		\|D_t v_\ell\|_{L^\infty_tL^{4}_x}
		+\sup_{i,k} \|D_t V_i^k\|_{L^\infty_tL^{4}_x}
		+\|D_t Q_{q+1}\|_{L^\infty_tL^{4}_x}\\
        &\leq 
        \lambda_q^n+C \delta_{q+1}^{1 / 2} \lambda_{q+1}^{1+\kappa}+C \delta_{q+1}^{-2 / 3} r_{q+1}^{-5 / 2}+C \lambda_{q+1}^{3 / 2} \lambda_q^{6 n+2 \beta \sigma} \\
        &\le
         \frac14\lambda_{q+1}^n.
         \end{aligned}
     \]

	\subsection{Reynolds-stress estimates}
	\leavevmode\par\noindent
	We now estimate each component of the new Reynolds stress. Recall that
	\begin{equation*}
		R_{q+1}
		:=
		\mathring R_{q+1}^{(l)}
		+
		\mathring R_{q+1}^{(c)}
		+
		\mathring R_{q+1}^{(t)}
		+
		\mathring R_{q+1}^{(s)}
		+
		\sum_{i,k}(R_i^k+G_i^k)
	\end{equation*}
		By Lemma~\ref{est:bilinear}, for $p>\frac43$, we estimate
	\[
	\begin{aligned}
		\|\mathring R^{(l)}_{q+1}\|_{L^\infty_t L^1_x}
		&\lesssim_{p} \|\mathcal{R}_0\left(N(u_\ell,w_{q+1})+N(W_{q+1},v_\ell)\right)\|_{L^\infty_t L^{\frac{4p}{4+p}}_x}\\
		&\lesssim_{p} \|u_\ell\|_{L^\infty_t L^4_x}
		\sup_{i,k}\|V_i^k\|_{L^\infty_t L^{p}_x}\\
		&\lesssim_{p}\lambda_q^n \delta_{q+1}^{\frac12} \lambda_{q+1}^{-\frac{1-s_{p}}{\alpha}}\\
		&\lesssim_{p}\lambda_1^{-\beta}\delta_{q+2} 
		\lambda_{q+1}^{\frac{n}{\sigma}+\beta\sigma-\frac{\beta}{2}-\frac{1-s_{p}}{\alpha}}\\
		&\leq \frac{1}{10}\delta_{q+2},
	\end{aligned}
	\]
	provided
	\begin{equation*}
	\frac{1-s_p}{\alpha}+\frac{\beta}{2}-\frac{n}{\sigma}>\beta \sigma,
	\end{equation*}
	with $p=2$ and  $s_p=\frac12$.

	Similarly, for $p>\frac43$, by decomposing
	\[
	\begin{aligned}
	\mathring R^{(c)}_{q+1}
	&=\mathcal R_0\!\left(N(u_\ell,Q_{q+1})+N(Q^u_{q+1},v_\ell)\right)\\
	&\quad+\mathcal R_0\!\left(N(W_{q+1},Q_{q+1})
	+N(Q^u_{q+1},w_{q+1})\right)\\
	&\quad+\mathcal R_0N(Q^u_{q+1},Q_{q+1}),
	\end{aligned}
	\]
	we shall estimate
	\[
	\begin{aligned}
		\|\mathring R^{(c)}_{q+1}\|_{L^\infty_t L^1_x}
		&\lesssim_{p}  \|u_\ell\|_{L^4_x}\|Q_{q+1}\|_{L^\infty_t L^{p}_x}
		+\|W_{q+1}\|_{L^\infty_t L^{p}_x}
		\|Q_{q+1}\|_{L^\infty_t L^{4}_x}
		+
		\|Q_{q+1}\|_{L^\infty_t L^{4}_x}\|DQ_{q+1}\|_{L^\infty_t L^{p}_x}
		\\
		&\lesssim_{p}  \lambda_{q+1}^{\frac{n}{\sigma}}
		\tau_{q+1} \lambda_{q+1}^{2-\frac{2}{p}}
		\lambda_q^{6n+2\beta\sigma}\\
		&\quad+\delta_{q+1}^{\frac12-\frac23 s_{p}}r_{q+1}^{-s_{p}}
		\tau_{q+1}\lambda_{q+1}^{3/2}\lambda_q^{6n+2\beta\sigma}\\
		&\quad+\tau_{q+1}\lambda_{q+1}^{3/2}\lambda_q^{6n+2\beta\sigma}
		\cdot\tau_{q+1}\lambda_{q+1}^{3-\frac2p}
		\lambda_q^{6n+2\beta\sigma}\\
		&\lesssim_{p} \lambda_{q+1}^{\gamma_1}\delta_{q+2}+
		\lambda_{q+1}^{\gamma_2}\delta_{q+2}+
		\lambda_{q+1}^{\gamma_3}\delta_{q+2}\\
		& \leq \frac{1}{10}\delta_{q+2},
	\end{aligned}
	\]
	provided
    \begin{equation*}
      \begin{aligned}
          &\gamma_1=\frac{7n}{\sigma}-\kappa+2-\frac{2}{p}+2\beta+\beta\sigma <0,
          \\
          &\gamma_2=\mu s_{p}-\kappa+\frac32+\frac{6n}{\sigma}+(\frac32+\frac23 s_{p})\beta+ \beta\sigma<0,
          \\
          &\gamma_3=-2\kappa+ \frac92-\frac{2}{p}+\frac{12n}{\sigma}+4\beta+ \beta\sigma<0,
       \end{aligned}
    \end{equation*}
	with $p=\frac85$ and  $s_p=\frac14$.
	\\
	For the time-freezing error, using \eqref{est-time-error-a} and summing
	over the four directions, we obtain
	\[
	\begin{aligned}
	\|\mathring R^{(t)}_{q+1}\|_{L^\infty_tL^1_x}
	&\leq
	\sum_{i=1}^4\sup_{k\in\mathbb N}\sup_{t\in\mathcal T^k}
	\|a_i^k-a_i(\cdot,t)\|_{L_x^1}\\
	&\leq
	4C\tau_{q+1}\lambda_q^{6n+2\beta\sigma}\\
	&\leq  C\lambda_1^{-2\beta}\delta_{q+2}
	\lambda_{q+1}^{\beta\sigma-\kappa+\frac{6n}{\sigma}+2\beta}
	\leq \frac{1}{10}\delta_{q+2},
	\end{aligned}
	\]
	provided
	\begin{equation*}
	\kappa-2\beta-\frac{6n}{\sigma}>\beta\sigma.
	\end{equation*}
	For the auxiliary-average error, at each time there is a unique active
	coarse interval \(\mathcal T^k\). Hence \eqref{eq:Gik-L1-estimate-sqg}
	and \eqref{est-aik} give
	\[
	\begin{aligned}
	\left\|\sum_{i,k}G_i^k\right\|_{L^\infty_tL^1_x}
	&=\sup_{k\in\mathbb N}\left\|\sum_{i=1}^4G_i^k\right\|_{L_x^1}\\
	&\leq 4C\sup_{i,k}
	\frac{\|a_i^k\|_{C^1_x}}{\lambda_{q+1}}\\
	&\leq
	C\lambda_1^{-2\beta}
	\lambda_{q+1}^{-1+\frac{6n}{\sigma}+2\beta+\beta\sigma}
	\delta_{q+2}
	\leq \frac{1}{10}\delta_{q+2}.
	\end{aligned}
	\]
	provided
	\begin{equation*}
	1-2\beta-\frac{6n}{\sigma}>\beta\sigma.
	\end{equation*}
	
	We next estimate the principal building-block error. Since the intervals
	\(\mathcal T_i^k\) are pairwise disjoint and every \(R_i^k\) is extended by
	zero outside \(\mathcal T_i^k\), at most one such tensor is nonzero at each
	time. Therefore
	\[
	\begin{aligned}
		\left\|\sum_{i,k}R_i^k\right\|_{L^\infty_t L^1_x}
		&\leq \sup_{i,k}\|R_i^k\|_{L^\infty_t L^1_x}\\
		&\leq  C_p\delta_{q+1}r_{q+1}^{\frac2p-1}
		\lambda_q^{6n+3\beta\sigma+
			\left(\frac2p-1\right)
			\left(\frac{8n}{3}
			+
			\frac{2\beta\sigma}{3}
			\right)}\\
		&\leq C_p
		\delta_{q+2}
		\lambda_{q+1}^{
			\frac{10p+16}{3p}\frac{n}{\sigma}
			+\beta\sigma
			+\frac{4(p+1)}{3p}\beta
			+\frac{p-2}{p}\mu}
		\leq \frac{1}{10}\delta_{q+2},
	\end{aligned}
	\]
	provided
	\begin{equation*}
	\frac{10p+16}{3p}\frac{n}{\sigma}
	+\beta\sigma
	+\frac{4(p+1)}{3p}\beta
	+\frac{p-2}{p}\mu<0 \quad\text{with}\quad p=\frac{11}{10}.
	\end{equation*}

	For the source-replacement term, by \eqref{eq:SQG-support-condition}, we have
	\[
	\begin{aligned}
		\|\mathring R_{q+1}^{(s)}\|_{L^\infty_t L^1_x}
		&\le C_p \lambda_{q+1}^{-1}\sup_{i,k}
		\left\{\left(\|S_i^k\|_{L^\infty_t L^{p}_x}
		+\|\widetilde{\Omega}_i^k\|_{L^\infty_t L^{p}_x}\right)
		\sup_t \left|\frac{d}{dt}
		\left[
		\eta_i^k\zeta_i^k(t)
		\bigl(r_i^k(x_i^k(t))\bigr)^{3/2}
		\right]\right|\right\}\\
		&\leq C_p \lambda_{q+1}^{-1}  
		r_{q+1}^{\frac{2}{p}-2}\delta_{q+1}^{\frac{4}{3p}-\frac43}
		\lambda_q^{6n+2\beta\sigma}\\
		&\leq C_p \delta_{q+2}
		\lambda_1^{-\frac{2\beta(7p-4)}{3p}}
		\lambda_{q+1}^{
			\frac{6n}{\sigma}
			+\beta\sigma
			+\frac{10p-4}{3p}\beta
			+\frac{2(p-1)}{p}\mu-1}
		\\
		& \leq \frac{1}{10} \delta_{q+2},
	\end{aligned}
	\]
	provided 
	\begin{equation*}
	\frac{6n}{\sigma}
	+\beta\sigma
	+\frac{10p-4}{3p}\beta
	+\frac{2(p-1)}{p}\mu-1<0,\quad\text{with} \quad p=\frac{21}{20}.
	\end{equation*}
	
	\subsection{Parameter compatibility}
	\label{subsec:parameter-compatibility}
	
	We collect here all the restrictions on the parameters used in the iteration.
	Recall that
	\begin{equation*}
		\lambda_{q+1}=\lambda_q^\sigma,
		\qquad
		\delta_q=\lambda_1^{2\beta}\lambda_q^{-\beta},
		\qquad
		r_{q+1}=\lambda_{q+1}^{-\mu},
		\qquad
		\tau_{q+1}=\lambda_{q+1}^{-\kappa}.
	\end{equation*}
	We require
	\begin{equation*}
		n,\sigma,\kappa\in\mathbb N_{\geq1},
		\qquad
		\beta,\mu,\gamma>0,
		\qquad
		0<\alpha<1,
		\qquad
		\frac43<\bar p<4,
	\end{equation*}
	and set
	\begin{equation*}
		s_{\bar p}:=\frac32-\frac2{\bar p}>0.
	\end{equation*}
	
	The parameter conditions are the following:
	\begin{enumerate}
		\item The mollification estimate requires
		\begin{equation*}
			n\geq\frac12\beta\sigma.
		\end{equation*}
		
		\item The spatial localization of the source cores requires
		\begin{equation*}
			\frac{8n}{3\sigma}+\frac23\beta
			<
			\mu-\frac1\alpha.
		\end{equation*}
		
		\item The time-periodicity argument requires
		\begin{equation*}
			\frac32\mu-\beta-2-\kappa>0.
		\end{equation*}
		
		\item The estimates involving the time derivative of the building blocks require
		\begin{equation*}
			\mu>\frac{6n}{\sigma}+3\beta.
		\end{equation*}
		
		\item The \(L_t^\infty L_x^4\) estimate for the velocity increment requires
		\begin{equation*}
			\frac32-\kappa+\frac{6n}{\sigma}+\frac52\beta
			\leq0.
		\end{equation*}
		
		\item The \(L_t^\infty L_x^{\bar p}\) estimate for the derivative of the
		potential velocity requires
		\begin{align}
			\beta\sigma&<1,\notag
\\
			\left(\mu+\frac23\beta\right)s_{\bar p}
			-\frac{49}{100}\beta
			&<0,
			\label{eq:param-pbar-main}\\
			3-\kappa-\frac2{\bar p}
			+\frac{6n}{\sigma}
			+\frac{201}{100}\beta
			&<0.\notag
		\end{align}
		
		\item The \(C_t^\gamma L_x^{\bar p}\) estimate requires
		\begin{align}
			\left(\mu+\frac23\beta\right)s_{\bar p}
			-\frac{49}{100}\beta
			+\gamma(1+\kappa)
			&<0,\notag
\\
			\left(\mu+\frac23\beta\right)s_{\bar p}
			-\frac{49}{100}\beta
			+\gamma\left(\frac52\mu+\frac76\beta\right)
			&<0,\notag
\\
			\frac32+s_{\bar p}
			-\kappa(1-\gamma)
			+\frac{6n}{\sigma}
			+\frac{201}{100}\beta
			&<0.
			\label{eq:param-holder-3}
		\end{align}
		
		\item The estimates for the first space-time derivatives of \(u_{q+1}\)
		and \(v_{q+1}\) require
		\begin{align*}
			2\mu+\frac56\beta
			&<n,
			\\
			\frac72-\kappa+2\beta+\frac{6n}{\sigma}
			&<n,
			\\
			\kappa+1+\mu+\frac16\beta
			&<n,
			\\
			\frac43\beta+\frac72\mu
			&<n,
			\\
			\frac52+\frac{6n}{\sigma}+2\beta
			&<n.
		\end{align*}
		
		\item The linear interaction error requires
		\begin{equation*}
			\frac1{2\alpha}+\frac{\beta}{2}-\frac n\sigma
			>
			\beta\sigma.
		\end{equation*}
		
		\item The corrector interaction error requires
		\begin{align*}
			\frac{7n}{\sigma}-\kappa+\frac34
			+2\beta+\beta\sigma
			&<0,
			\\
			\frac14\mu-\kappa+\frac32+\frac{6n}{\sigma}
			+\frac53\beta+\beta\sigma
			&<0,
			\\
			-2\kappa+\frac{13}{4}+\frac{12n}{\sigma}
			+4\beta+\beta\sigma
			&<0.
		\end{align*}
		
		\item The time-freezing error requires
		\begin{equation*}
			\kappa-2\beta-\frac{6n}{\sigma}
			>
			\beta\sigma.
		\end{equation*}
		
		\item The auxiliary-average error requires
		\begin{equation*}
			1-2\beta-\frac{6n}{\sigma}
			>
			\beta\sigma.
		\end{equation*}
		
		\item The principal building-block error requires
		\begin{equation*}
			\frac{90}{11}\frac n\sigma
			+\beta\sigma
			+\frac{28}{11}\beta
			-\frac9{11}\mu
			<0.
		\end{equation*}
		
		\item The source-replacement error requires
		\begin{equation*}
			\frac{6n}{\sigma}
			+\beta\sigma
			+\frac{130}{63}\beta
			+\frac{2}{21}\mu
			<1.
		\end{equation*}
		
	\end{enumerate}
	The conditions listed above are mutually compatible. More precisely, one may
	choose
	\begin{equation}\label{eq:admissible-parameters}
		\boxed{
			n=14,\qquad
			\sigma=250,\qquad
			\kappa=3,\qquad
			\beta=\frac1{1000},\qquad
			\alpha=\frac45,\qquad
			\mu=\frac{17}{5}.
		}
	\end{equation}
	Combining \eqref{eq:param-pbar-main}-\eqref{eq:param-holder-3}, the small quantity $s_{\bar{p}}$ satisfies
    \[
s_{\bar p}
<
\min\left\{
\frac{\frac{49}{100}\beta-\gamma(1+\kappa)}
{\mu+\frac23\beta},
\frac{\frac{49}{100}\beta
-\gamma\left(\frac52\mu+\frac76\beta\right)}
{\mu+\frac23\beta},
\kappa(1-\gamma)-\frac32-\frac{6n}{\sigma}
-\frac{201}{100}\beta
\right\}.
\]
Substituting the above parameter values, we may choose the following:
	\begin{equation*}
		\boxed{
			\bar p=\frac{400003}{300000}
			=\frac43+\frac1{100000},
			\qquad
			\gamma=\frac1{100000}.
		}
	\end{equation*}
	After all these parameters have been fixed, $\lambda_0\in\mathbb N$ may be
	chosen sufficiently large, depending only on the fixed parameters and the
	constants appearing in the preceding estimates.

	Direct substitution verifies every strict inequality listed above.
	The four-direction factors in the time-freezing and auxiliary-average
	estimates are fixed constants. Since all the powers of \(\lambda_{q+1}\)
	appearing in the six stress estimates are strictly negative, these and all
	other fixed constants can be absorbed by increasing \(\lambda_0\), so that
	each entire stress group is bounded by \(\delta_{q+2}/10\).
	 Assuming that estimate, the triangle inequality and
	\eqref{eq:Rq1-definition-sqg} give, for the piecewise-in-time construction,
	\[
	\begin{aligned}
	\|R_{q+1}\|_{L_t^\infty L_x^1}
	&\leq
	\|\mathring R_{q+1}^{(l)}\|_{L_t^\infty L_x^1}
	+\|\mathring R_{q+1}^{(c)}\|_{L_t^\infty L_x^1}
	+\|\mathring R_{q+1}^{(t)}\|_{L_t^\infty L_x^1}
	+\|\mathring R_{q+1}^{(s)}\|_{L_t^\infty L_x^1}\\
	&\quad
	+\left\|\sum_{i,k}R_i^k\right\|_{L_t^\infty L_x^1}
	+\left\|\sum_{i,k}G_i^k\right\|_{L_t^\infty L_x^1}\\
	&\leq 6\,\frac{\delta_{q+2}}{10}
	<\delta_{q+2}.
	\end{aligned}
	\]

	\appendix
    \section{Auxiliary analytic estimates}
	We collect here the two analytic tools used repeatedly in the iteration:
	a localized symmetric anti-divergence estimate and a bilinear null-form
	estimate for the SQG interaction.
	
	\subsection{Symmetric anti-divergence}
	
	The anti-divergence operator converts mean-zero vector errors into symmetric
	Reynolds stresses. Its localized form also supplies a small factor equal to
	the diameter of the support.
	
	On the torus $\mathbb{T}^2$, we consider the operator
	\begin{equation*}
		\mathcal R_0v
		:=
		\nabla\Delta^{-1}v
		+(\nabla\Delta^{-1}v)^T
		-(\operatorname{div}\Delta^{-1}v)I.
	\end{equation*}
	for every $v\in C^\infty(\mathbb{T}^2;\mathbb{R}^2)$ such that
	$\int_{\mathbb{T}^2} v(x)\,dx=0$.
	This operator maps mean-zero vector fields to symmetric tensors:
	\begin{equation*}
		\mathcal{R}_0 : C^\infty_{0}(\mathbb{T}^2; \mathbb{R}^2)
		\to C^\infty (\mathbb{T}^2; \operatorname{Sym}_2),
	\end{equation*}
	where $\operatorname{Sym}_2$ is the space of symmetric tensors in
	$\mathbb{R}^2$ and the subscript $0$ denotes zero spatial mean. A
	direct calculation gives
	$\operatorname{div}(\mathcal R_0v)=v$, while $D\mathcal R_0$ and
	$\mathcal R_0\operatorname{div}$ are Calderón--Zygmund operators. Thus, for
	every $1<p<\infty$,
	\begin{align*}
		\|\mathcal R_0v\|_{W^{1,p}}
		&\le C_p\|v\|_{L^p},\\
		\|\mathcal R_0\operatorname{div}A\|_{L^p}
		&\le C_p\|A\|_{L^p}.
	\end{align*}
	The mean-zero condition gives the following scale-localized estimate.
	
	\begin{proposition}[Localized \(L^p\) estimate for the symmetric anti-divergence]
		\label{prop:antidiv-compact}
		Let $0<r<1/4$. Assume that $v\in C^\infty_c(\mathbb{T}^2;\mathbb{R}^2)$
		is supported in a ball of radius $r$ and
		$\int_{\mathbb{T}^2}v(x)\,dx=0$. Then
		\begin{equation*}
			\|\mathcal{R}_0(v)\|_{L^p} \le C(p) r \| v \|_{L^p} \, ,
			\quad \text{for every $p\in (1,\infty)$}\, .
		\end{equation*}  
	\end{proposition}

	\begin{proof}
		Let \(p'=p/(p-1)\), and denote the supporting ball by \(B\).
		Since \(r<1/4\), the ball \(B\) is contained in a Euclidean coordinate
		chart of \(\mathbb T^2\), with constants independent of \(r\). Let
		\(\mathcal R_0^*\) denote the \(L^2\)-adjoint of \(\mathcal R_0\).
		For every smooth matrix field \(G\), the support and zero-mean
		assumptions on \(v\) give
		\[
		\begin{aligned}
			\left|
			\int_{\mathbb T^2}\mathcal R_0v:G\,dx
			\right|
			&=
			\left|
			\int_B v\cdot
			\left(
				\mathcal R_0^*G
				-\fint_B\mathcal R_0^*G\,dx
			\right)\,dx
			\right|\\
			&\leq
			\|v\|_{L^p(B)}
			\left\|
				\mathcal R_0^*G
				-\fint_B\mathcal R_0^*G\,dx
			\right\|_{L^{p'}(B)}\\
			&\lesssim_p
			r\|v\|_{L^p(\mathbb T^2)}
			\|\nabla\mathcal R_0^*G\|_{L^{p'}(B)}\\
			&\lesssim_p
			r\|v\|_{L^p(\mathbb T^2)}
			\|G\|_{L^{p'}(\mathbb T^2)}.
		\end{aligned}
		\]
		The third line is the scaled Poincaré inequality on \(B\). The last
		follows because each component of \(\nabla\mathcal R_0^*\) is an
		order-zero periodic Calderón--Zygmund operator: up to a sign, its
		multiplier is the adjoint of the corresponding multiplier of
		\(D\mathcal R_0\). Taking the dual supremum proves the claim.
	\end{proof}

	\subsection{Bilinear null-form estimate}
	
	The following estimate exploits the symmetric SQG interaction to avoid
	placing a full derivative on the concentrated perturbation. The argument
	below reduces the periodic estimate to the bilinear Calder\'on commutator
	bound~\eqref{eq:bilinear-Calderon-commutator}, which is an explicit analytic
	input to the conditional scheme.
	
	\begin{lemma}[A bilinear null-form estimate]\label{est:bilinear}
		Let $\mathcal R_0$ be the standard symmetric anti-divergence
		operator on $\mathbb T^2$, namely
		\[
		(\mathcal R_0 h)_{ij}
		:=
		\partial_i\Delta^{-1}h_j
		+
		\partial_j\Delta^{-1}h_i
		-
		\delta_{ij}\partial_\ell\Delta^{-1}h_\ell,
		\]
		where $\Delta^{-1}$ is defined on mean-zero functions. Let
		\[
		\Lambda:=(-\Delta)^{1/2},
		\qquad
		\mathcal R:=\nabla\Lambda^{-1},
		\]
		with the zero Fourier mode set equal to zero.
		
		Assume that
		\[
		1<p_1,p_2<\infty,
		\qquad
		\frac1s=\frac1{p_1}+\frac1{p_2}<1.
		\]
		Then, for every pair of smooth mean-zero scalar functions
		$f,g\in C^\infty(\mathbb T^2)$, one has
		\[
		\left\|
		\mathcal R_0\bigl(g\mathcal R f+f\mathcal R g\bigr)
		\right\|_{L^s(\mathbb T^2)}
		\lesssim_{p_1,p_2}
		\|f\|_{L^{p_1}(\mathbb T^2)}
		\|\Lambda^{-1}g\|_{L^{p_2}(\mathbb T^2)}.
		\]
		The same estimate holds with $\mathcal R$ replaced by
		$\mathcal R^\perp=\nabla^\perp\Lambda^{-1}$.
	\end{lemma}
	\begin{proof}
		We argue componentwise. Set
		\[
		h:=g\mathcal Rf+f\mathcal Rg.
		\]
		First, \(h\) has zero spatial mean. Indeed, for every
		\(\ell\in\{1,2\}\), the Riesz transform \(\mathcal R_\ell\) is
		skew-adjoint, and therefore
		\[
		\begin{aligned}
			\int_{\mathbb T^2}h_\ell\,dx
			&=
			\int_{\mathbb T^2}g\mathcal R_\ell f\,dx
			+
			\int_{\mathbb T^2}f\mathcal R_\ell g\,dx \\
			&=
			-\int_{\mathbb T^2}f\mathcal R_\ell g\,dx
			+
			\int_{\mathbb T^2}f\mathcal R_\ell g\,dx
			=0.
		\end{aligned}
		\]
		Thus \(\mathcal R_0h\) is well defined.
		
		Define the vector field
		\[
		g_1:=\nabla\Delta^{-1}g=-\nabla\Lambda^{-2}g.
		\]
		Since \(g\) is mean-free,
		\[
		\operatorname{div}g_1
		=
		\Delta\Delta^{-1}g
		=
		g.
		\]
		Moreover,
		\[
		\|g_1\|_{L^{p_2}}
		=
		\|\nabla\Lambda^{-2}g\|_{L^{p_2}}
		\lesssim_{p_2}
		\|\Lambda^{-1}g\|_{L^{p_2}}.
		\]
		Since the Riesz transforms commute with derivatives, we have
		\[
		\mathcal R_\ell g
		=
		\mathcal R_\ell\partial_jg_{1j}
		=
		\partial_j\mathcal R_\ell g_{1j}.
		\]
		Consequently,
		\begin{equation}\label{eq:null-form-divergence-decomposition}
			h_\ell
			=
			\partial_jg_{1j}\,\mathcal R_\ell f
			+
			f\,\partial_j\mathcal R_\ell g_{1j}.
		\end{equation}
		
		Let \(\mathcal K_{ab\ell}\) denote the periodic convolution kernel of
		the operator \(\mathcal R_0\), so that
		\[
		(\mathcal R_0h)_{ab}(x)
		=
		\int_{\mathbb T^2}
		\mathcal K_{ab\ell}(x-y)h_\ell(y)\,dy.
		\]
		Let \(\mathscr K_\ell\) denote the periodic kernel of
		\(\mathcal R_\ell\). All repeated indices are summed. The calculations
		below are first performed with smooth truncations of the kernels; the
		estimates are uniform in the truncation parameter, so the identities
		follow by passing to the limit.
		
		Using \eqref{eq:null-form-divergence-decomposition}, write
		\[
		(\mathcal R_0h)_{ab}=I_{ab}+II_{ab},
		\]
		where
		\[
		I_{ab}(x)
		:=
		\int_{\mathbb T^2}
		\mathcal K_{ab\ell}(x-y)
		\partial_jg_{1j}(y)
		\mathcal R_\ell f(y)\,dy
		\]
		and
		\[
		II_{ab}(x)
		:=
		\int_{\mathbb T^2}
		\mathcal K_{ab\ell}(x-y)
		f(y)\partial_j\mathcal R_\ell g_{1j}(y)\,dy.
		\]
		
		Integrating by parts in the first term gives
		\[
		\begin{aligned}
			I_{ab}(x)
			&=
			-\int_{\mathbb T^2}
			g_{1j}(y)
			\partial_{y_j}
			\left[
			\mathcal K_{ab\ell}(x-y)
			\mathcal R_\ell f(y)
			\right]dy \\
			&=
			-\iint_{\mathbb T^2\times\mathbb T^2}
			g_{1j}(y)
			\partial_{y_j}
			\left[
			\mathcal K_{ab\ell}(x-y)
			\mathscr K_\ell(y-z)
			\right]
			f(z)\,dz\,dy.
		\end{aligned}
		\]
		On the other hand,
		\[
		II_{ab}(x)
		=
		\iint_{\mathbb T^2\times\mathbb T^2}
		\mathcal K_{ab\ell}(x-y)
		f(y)\partial_{y_j}\mathscr K_\ell(y-z)
		g_{1j}(z)\,dz\,dy.
		\]
		Interchanging \(y\) and \(z\), and using that
		\(\mathscr K_\ell\) is odd and hence
		\(\partial_j\mathscr K_\ell\) is even, we obtain
		\[
		II_{ab}(x)
		=
		\iint_{\mathbb T^2\times\mathbb T^2}
		g_{1j}(y)f(z)
		\mathcal K_{ab\ell}(x-z)
		\partial_{y_j}\mathscr K_\ell(y-z)
		\,dy\,dz.
		\]
		Combining the two identities yields
		\[
		\begin{aligned}
			(\mathcal R_0h)_{ab}(x)
			={}&
			-\iint
			g_{1j}(y)f(z)
			\mathscr K_\ell(y-z)
			\partial_{y_j}\mathcal K_{ab\ell}(x-y)
			\,dy\,dz \\
			&+
			\iint
			g_{1j}(y)f(z)
			\left[
			\mathcal K_{ab\ell}(x-z)
			-
			\mathcal K_{ab\ell}(x-y)
			\right]
			\partial_{y_j}\mathscr K_\ell(y-z)
			\,dy\,dz.
		\end{aligned}
		\]
		Since
		\[
		\partial_{y_j}\mathcal K_{ab\ell}(x-y)
		=
		-\partial_{x_j}\mathcal K_{ab\ell}(x-y),
		\]
		we may write
		\begin{equation*}
			(\mathcal R_0h)_{ab}
			=
			T^{(1)}_{ab}+T^{(2)}_{ab},
		\end{equation*}
		where
		\[
		T^{(1)}_{ab}
		=
		\partial_j\mathcal K_{ab\ell}
		*
		\bigl(g_{1j}\mathcal R_\ell f\bigr)
		\]
		and
		\[
		\begin{aligned}
			T^{(2)}_{ab}(x)
			:=
			\iint_{\mathbb T^2\times\mathbb T^2}
			&g_{1j}(y)f(z)
			\left[
			\mathcal K_{ab\ell}(x-z)
			-
			\mathcal K_{ab\ell}(x-y)
			\right] \\
			&\qquad\qquad\times
			\partial_{y_j}\mathscr K_\ell(y-z)
			\,dy\,dz.
		\end{aligned}
		\]
		
		We first estimate \(T^{(1)}\). Since \(\mathcal R_0\) has order
		\(-1\), the kernel \(\partial_j\mathcal K_{ab\ell}\) defines a
		zero-order Calderón--Zygmund operator. Therefore, using Hölder's
		inequality and the boundedness of the Riesz transforms,
		\[
		\begin{aligned}
			\|T^{(1)}\|_{L^s}
			&\lesssim_s
			\|g_1\mathcal Rf\|_{L^s} \\
			&\le
			\|g_1\|_{L^{p_2}}
			\|\mathcal Rf\|_{L^{p_1}} \\
			&\lesssim_{p_1,p_2}
			\|g_1\|_{L^{p_2}}
			\|f\|_{L^{p_1}}.
		\end{aligned}
		\]
		
		For the second term, we use the bilinear Christ--Journ\'e commutator
		estimate; see~\cite[Theorem~1.1]{SSS19}, as well as the general-kernel
		formulation in~\cite[Theorem~2.8]{SSS19}. In the notation needed here,
		this estimate reads
		\begin{equation}\label{eq:bilinear-Calderon-commutator}
			\begin{aligned}
				\Bigg\|
				\iint F_j(y)F_0(z)
				\big[
				\mathcal K(x-z)-\mathcal K(x-y)
				\big]
				\partial_{y_j}\mathscr K(y-z)
				\,dy\,dz
				\Bigg\|_{L^s_x}
				\\
				\lesssim_{p_1,p_2}
				\|F_0\|_{L^{p_1}}
				\|F\|_{L^{p_2}},
			\end{aligned}
		\end{equation}
		whenever
		\[
		1<p_1,p_2<\infty,
		\qquad
		\frac1s=\frac1{p_1}+\frac1{p_2}<1.
		\]
		Indeed, pairing the expression on the left-hand side of
		\eqref{eq:bilinear-Calderon-commutator} with
		$F_2\in L^{s'}(\mathbb R^2)$ produces a trilinear
		Christ--Journ\'e-type singular integral form. The exponents satisfy
		\[
		\frac1{p_1}+\frac1{p_2}+\frac1{s'}=1,
		\]
		so the cited multilinear estimate, followed by duality, gives
		\eqref{eq:bilinear-Calderon-commutator} in the Euclidean setting.
		For completeness, the cancellation in
		\eqref{eq:bilinear-Calderon-commutator} follows from the first-order
		difference identity
		\[
		\begin{aligned}
			\mathcal K(x-z)-\mathcal K(x-y)
			=
			\int_0^1
			(y-z)_m
			\partial_m\mathcal K
			\bigl(x-y+t(y-z)\bigr)\,dt.
		\end{aligned}
		\]
		Indeed, multiplication by \(y-z\) cancels one derivative of
		\(\partial_j\mathscr K(y-z)\), and the resulting kernel is a standard
		bilinear Calderón commutator kernel. Its periodic version follows by
		decomposing each periodic kernel into its Euclidean singular part and
		a smooth remainder; the latter is estimated directly by Hölder's
		inequality.
		
		Applying \eqref{eq:bilinear-Calderon-commutator} with
		\[
		F=g_1,
		\qquad
		F_0=f,
		\]
		gives
		\[
		\|T^{(2)}\|_{L^s}
		\lesssim_{p_1,p_2}
		\|g_1\|_{L^{p_2}}
		\|f\|_{L^{p_1}}.
		\]
		Combining this estimate with the bound for \(T^{(1)}\), we obtain
		\[
		\left\|
		\mathcal R_0
		\bigl(g\mathcal Rf+f\mathcal Rg\bigr)
		\right\|_{L^s}
		\lesssim_{p_1,p_2}
		\|g_1\|_{L^{p_2}}
		\|f\|_{L^{p_1}}.
		\]
		Finally,
		\[
		\|g_1\|_{L^{p_2}}
		\lesssim_{p_2}
		\|\Lambda^{-1}g\|_{L^{p_2}},
		\]
		and hence
		\[
		\left\|
		\mathcal R_0
		\bigl(g\mathcal Rf+f\mathcal Rg\bigr)
		\right\|_{L^s}
		\lesssim_{p_1,p_2}
		\|f\|_{L^{p_1}}
		\|\Lambda^{-1}g\|_{L^{p_2}}.
		\]
		
		The proof for
		\(\mathcal R^\perp=\nabla^\perp\Lambda^{-1}\) is identical: its
		kernel is obtained from the Riesz kernel by a fixed \(90^\circ\)
		rotation, so it satisfies the same size, smoothness, oddness, and
		Calderón--Zygmund estimates.
	\end{proof}

\end{document}